\documentclass[smallextended]{svjour3}
\smartqed
\usepackage{amsmath,amssymb,epsfig,float,color,xcolor}
\usepackage[colorlinks=false,breaklinks=true,linkcolor=blue]{hyperref}
\usepackage{diagbox}

\def\CT{\mathcal{T}}
\def\CM{\mathcal{X}}
\def\CN{\mathcal{Y}}

\newcommand\err{\texttt{err}}
\newcommand\eff{\texttt{eff}}

\numberwithin{equation}{section}

\newcommand\bT{\boldsymbol{T}}

\newcommand\beps{\boldsymbol{\varepsilon}}

\newcommand\bsig{\boldsymbol{\sigma}}

\newcommand\bu{\boldsymbol{u}}
\newcommand\bv{\boldsymbol{v}}

\newcommand\R{\mathbb{R}}

\renewcommand\O{\Omega}
\newcommand\G{\Gamma}

\renewcommand\H{\mathrm{H}}
\renewcommand\L{\mathrm{L}}
\newcommand\Q{\mathrm{Q}}

\newcommand\X{\mathrm{X}}

\newcommand\bdiv{\mathop{\mathbf{div}}\nolimits}
\newcommand\vdiv{\mathop{\mathrm{div}}\nolimits}

\newcommand\tr{\mathop{\mathrm{tr}}\nolimits}

\renewcommand\sp{\mathop{\mathrm{sp}}\nolimits}

\newcommand\bn{\boldsymbol{n}}

\newcommand\bw{\boldsymbol{w}}

\numberwithin{equation}{section}
\numberwithin{theorem}{section}
\numberwithin{lemma}{section}
\numberwithin{corollary}{section}
\numberwithin{proposition}{section}
\numberwithin{remark}{section}

\def\CE{{\mathcal E}}
\def\CT{{\mathcal T}}

\newcommand{\vertiii}[1]{{\left\vert\kern-0.25ex\left\vert\kern-0.25ex\left\vert #1 
		\right\vert\kern-0.25ex\right\vert\kern-0.25ex\right\vert}}
\allowdisplaybreaks
\begin{document}
\title{Finite element analysis of the nearly incompressible
linear elasticity eigenvalue problem with variable coefficients}
\author{Arbaz Khan \and Felipe Lepe \and David Mora \and Jesus Vellojin}
\institute{F. Lepe \and D. Mora \and J. Vellojin \at
GIMNAP, Departamento de Matem\'atica,
Facultad de Ciencias, \\
Universidad del B\'io-B\'io, Casilla 5-C, Concepci\'on, Chile.\\
 \email{\{flepe,dmora,jvellojin\}@ubiobio.cl} \and
D. Mora \at
CI$^2$MA, Universidad de Concepci\'on, Casilla 160-C, Concepci\'on, Chile.
\and
A. Khan \at
Department of Mathematics, Indian Institute of Technology Roorkee, Roorkee 247667, India.\\
\email{arbaz@ma.iitr.ac.in}
}

\date{Received: date / Revised version: date \hfill Updated: \today}

\maketitle

\begin{abstract}
In this paper we present a mathematical and numerical
analysis of an eigenvalue problem associated to
the elasticity-Stokes equations stated in two and three dimensions.
Both problems are related through the Herrmann pressure.
Employing the Babu\v ska--Brezzi theory, it is proved that the
resulting continuous and discrete variational formulations are well-posed.
In particular, the finite element method is based on general inf-sup stables pairs
for the Stokes system, such that, Taylor--Hood finite elements.
By using a general approximation theory for compact operators,
we obtain optimal order error estimates for the eigenfunctions
and a double order for the eigenvalues.
Under mild assumptions, we have that these estimates hold
with constants independent of the Lam\'e coefficient $\lambda$.
In addition, we carry out the reliability and efficiency
analysis of a residual-based a posteriori error estimator for the spectral problem.
We report a series of numerical tests in order to assess the
performance of the method and its
behavior when the nearly incompressible case of elasticity is considered.

\end{abstract}


\section{Introduction}

The accurate solution of eigenvalue problems coming from continuum mechanics is crucial because the stability
of systems relying on fluids or structures, or their coupling, hinges on understanding vibration modes associated with spectral problems.
Notable monographs currently addressing this subject are cited here for reference \cite{MR2652780}.

One important aspect that is still under investigation is understanding the eigenvalue properties in parameter limits between Stokes and elasticity equations. It is well-established that in elasticity equations where the Poisson ratio is close to $1/2$, classical numerical methods exhibit the locking phenomenon,
which can be mitigated, for example, with  mixed formulations, least-squares, among others
(see, e.g. \cite{MR4570534,meddahi2013finite,MR4542511,MR4396855,MR4271564,MR4445536,MR3973678}).
In this incompressible regime, the equations align with Stokes equations. A study in \cite{MR3962898} demonstrates that the spectrum of the elasticity eigenvalue problem converges to that of the Stokes eigenvalue problem as the Poisson ratio tends to $1/2$. Recent research, as seen in \cite{zhong2023spectral}, proposes  a mixed method for the linear elasticity where   the stress, rotations, and displacement are considered as unknowns. Here the convergence of the method also works correctly for the case when Poisson ratio tends to $1/2$. Also, novel strategies like least squares formulations have been applied to elasticity eigenvalue problems \cite{MR4271564}.
		
In the realm of a posteriori analysis for elasticity eigenvalue problems, we have, for example, \cite{MR3973678}, where a locking-free mixed formulation based on the finite element method has been proposed for the load problem, accompanied by a priori and a posteriori error estimates. The authors provide estimation strategies that prove to be robust with respect to the material parameters. They ensure robustness in elasticity, which involves multiple parameters compared to the single viscosity parameter in the Stokes problem. We have also the work of \cite{bertrand2021adaptive}. Here, the authors aim to expand the results obtained in a previous study of a conforming mixed finite elements for the elasticity problem to address the related eigenvalue problem. Their study shows that using a residual-type error estimator can effectively steer an adaptive scheme Additionally, they explore how implementing a postprocessing technique can improve the convergence rate of the scheme. Postprocessing techniques have also been explored recently in \cite{MR4471016}, where a reliable and efficient  a posteriori error estimators for mixed formulations for the elasticity eigenvalue problem is studied,  including an analysis for the perfectly incompressible case.

The concepts presented in \cite{MR3973678} can be extended to the eigenvalue problem under the same norm.  It is well-known that such a problem lacks a unique solution due to eigenvalues/eigenfunctions exhibiting a multiplicity greater than one. Adding to this, we have a parameter dependent problem that connects the elasticity eigenvalue problem with the Stokes one, where the Poisson ratio serves as the involved parameter. This is an important issue, since according this parameter changes, the spectrum of the elasticity operator changes and hence, the multiplicity of the eigenvalues starts to behave different according to this. Since we have a fixed eigenvalue problem (Stokes) as a limit, the elasticity eigenvalue problem change its spectrum according to the change of the Poisson ratio when we are approximating to the spectrum of Stokes. Then, the question that emerge is how does the spectrum of the parameter-dependent elasticity problem converge towards the Stokes spectrum. This issue is difficult to handle and must be studied.

Our contribution focuses on analyzing both the Stokes and elasticity eigenproblems within a unified locking-free system. The key lies in considering the Herrmann pressure that connects the fluid and elastic representations. The introduction of variable coefficients, reflecting the physical attributes of the structure, adds another layer of complexity to our analysis. Specifically, these variable coefficients demands the use of a weighted norm with spatial-dependent parameters in order to ensure robustness of the estimates. It is worth emphasizing that, to the best of the authors' knowledge, this work represents a pioneering effort in examining the elasticity/Stokes eigenvalue problem with variable coefficients.  We also propose a scheme that expand the numerical analysis to encompass two families of finite elements known for their inf-sup stability: the Taylor-Hood elements and the mini element.
Moreover, others inf-sup stable finite elements could be also considered.
However, for an accurate analysis of the eigenvalue problem, certain adjustments become essential. More precisely, a suitable  scaling of the problem with respect to physical parameters is required in order   to achieve robustness for the spectral problem. The approach is to leverage simple tools for the analysis, and a key argument is the problem's regularity, which we will use to conduct both a priori and a posteriori error estimation with the aid of the compact operators theory \cite{BO}.
Finally, we stress that our locking-free method is very competitive in terms of computational cost.

The outline of our manuscript is the following: In Section \ref{sec:model} we present the model problem, describing the parameters involved on the elasticity eigensystem and their relation with the Stokes eigenvalue problem through the Herrmann's pressure. We outline the scaling of the problem and the respective variational formulation which we analyze under a suitable weighted norm. Regularity results are presented together with the parameter dependent solution operator.  Also, we report an analysis of the relation between the spectrums of the   elasticity and the Stokes eigenvalue problems when the Poisson ratio tends to $1/2$. Section \ref{sec:fem} is dedicated to the discretization of the eigenvalue problem, where we introduce the finite element spaces for our study, the corresponding discrete version of the nearly incompressible eigenvalue problem, and we derive convergence of the method, and   error estimates for the eigenvalues and eigenfunctions.  In Section \ref{sec:apost} we design an a posteriori error estimator which we prove under standard techniques it reliability and efficiency. Finally in Section \ref{sec:numerical-experiments} we report a series of numerical tests in order to study the performance of the method in two and three dimensions. Here the aims are two: first, to  confirm the results related to the a priori analysis, where we compute the spectrum and order of convergence for the approximation and secondly, to analyze the performance of the a posteriori estimator where the aim is to recover the optimal order of convergence when the configuration of the problem does not allow sufficiently smooth eigenfunctions, deteriorating the a priori convergence rates. 
\section{The model problem}
\label{sec:model}


In this section we present the eigenvalue problem of our interest.
Let $\Omega\subset\mathbb{R}^d$ with $d\in\{2,3\}$ an open and bounded
domain with Lipschitz boundary $\partial\Omega$. Let us assume that
the boundary is divided into two parts $\partial\Omega:=\Gamma_D\cup\Gamma_N$,
where $|\Gamma_D|>0$. In particular, we assume that the structure is fixed on $\Gamma_D$ and free of stress on $\Gamma_N$.

The eigenvalue problem of interest is given by:  Find $(\widehat{\kappa}, \bu)$ such that
\begin{equation}
	\label{eq:eigenvalue-problem}
	\begin{aligned}
		-\bdiv\left[2\widehat{\mu}\beps(\bu) + \widehat{\lambda}\tr(\beps(\bu))\mathbb{I}\right] &= \widehat{\kappa}\rho\bu &\text{ in } \Omega,\\
		\bu &= \boldsymbol{0} & \mbox{ on } \G_D,\\
		\bsig\bn  & =  {\boldsymbol{0} }& \mbox{ on } \G_N,
	\end{aligned}
\end{equation}
where $\sqrt{\widehat{\kappa}}$ is the natural frequency, $\bu$ is the displacement, $\rho$ is the material density,
the Lam\'e parameters $\widehat{\lambda}$ and $\widehat{\mu}$ are defined by
\begin{equation*}
	\widehat{\lambda}:=\frac{E(x)\nu}{(1+\nu)(1-2\nu)}\quad\text{and}\quad \widehat{\mu}:=\frac{E(x)}{2(1+\nu)},
\end{equation*}
and $\beps(\cdot)$ represents the strain tensor given by $\beps(\bv):=\frac{1}{2}(\nabla\bv+(\nabla\bv)^{\texttt{t}})$ where $\texttt{t}$ represents the transpose operator. On the other hand, regarding to the physical parameters,  $E(x)$ represents  the space-variable Young modulus and $\nu$ represents  the Poisson ratio which we assume as a constant.  Let us  assume the existence of two positive constants
$E_{\min}$ and $E_{\max}$ such that the Young's modulus is bounded in the sense 
$
E_{\min}\leq E(x)\leq E_{\max}.
$
Scaling the first equation of \eqref{eq:eigenvalue-problem} by $(1+\nu)$ we obtain the following eigenvalue problem: Find $(\kappa, \bu)$ such that
\begin{equation}
	\label{eq:eigenvalue-problem-scaled}
	\begin{aligned}
		-\bdiv\left[2\mu\beps(\bu) + \lambda\tr(\beps(\bu))\mathbb{I}\right] &= \kappa\rho\bu &\text{ in } \Omega,\\
		\bu &= \boldsymbol{0} & \mbox{ on } \G_D,\\
		\bsig\bn  & =  \boldsymbol{0} & \mbox{ on } \G_N,
	\end{aligned}
\end{equation}
where the scaled eigenvalue and the Lam\'e coefficients are given by
\begin{equation}
\label{eq:scaled}
\kappa:=(1+\nu)\widehat{\kappa},\qquad\lambda:=\frac{E(x)\nu}{1-2\nu},\qquad \mu:=\frac{E(x)}{2}.
\end{equation}
Note that for a given $E(x)$, the eigenvalue problem \eqref{eq:eigenvalue-problem-scaled} has the advantage of having $\mu$ independent of $\nu$. The incompressibility of the material is explicitly given only by the behavior of the constant $\nu$ through the Lamé parameter $\lambda$.

With the scaled parameters specified in \eqref{eq:scaled} we rewrite system \eqref{eq:eigenvalue-problem} as the following system:
We seek the eigenvalue $\kappa$, a vector field $\bu$, 
and a scalar field $p$ such that
\begin{align}\label{stokes-cont}
-\bdiv(2\mu\beps(\bu)) +\nabla p & =
 \kappa \rho  \bu & \mbox{ in } \O, \nonumber\\
\vdiv\bu +\frac{1}{\lambda}p& =  0 & \mbox{ in } \O, \nonumber  \\
\bu & = \boldsymbol{0} & \mbox{ on } \G_D,\\
\bsig\bn  & =  \boldsymbol{0} & \mbox{ on } \G_N. \nonumber
 \end{align}
We note from \eqref{eq:scaled} that if  the Poisson ratio approaches $1/2$,
then $\lambda \uparrow \infty$.
In such case, we recover the Stokes system.
This is an important matter in view of deriving robust stability and error bounds.   From now and on, and for simplicity,
in the analysis below we consider $\rho=1$.

\label{stab_section}
\subsection{Variational formulation and preliminary results}

The aim of this section is to introduce a variational formulation
for  system \eqref{stokes-cont}. From now and on the relation $\texttt{a} \lesssim \texttt{b}$ indicates that $\texttt{a} \leq C \texttt{b}$, with a positive constant $C$ which is independent of $\texttt{a}$, $\texttt{b}$. 

Since we are focused in the pressure and displacement, the resulting formulation is a saddle-point problem that can be studied according to the theory of Babu\v ska--Brezzi \cite{MR3097958}.

 Let us begin by  defining  the spaces $\H:=\left\{\bv\in \H^1(\Omega)^d\;:\; \bv =\boldsymbol{0} \text{ on }\Gamma_D\right\}$ and $\Q:=\L^2(\Omega)$ where we will seek the displacement and the pressure, respectively. Now, by testing
 system \eqref{stokes-cont} with adequate
functions and imposing the boundary conditions,
we end up with the following  saddle point variational formulation:
Find $(\kappa,(\bu,p))\in\R\times\H\times\Q$ with $(\bu,p)\neq (\boldsymbol{0},0)$ such that
\begin{align*}
2\int_{\O}\mu(x)\beps(\bu):\beps(\bv)\, dx-\int_{\O}p\vdiv\bv\, dx&=
\kappa\int_{\O}\bu\cdot\bv\, dx&
\quad\forall\bv\in\H,\\ 
-\int_{\O}q\vdiv\bu\, dx-\int_{\O}\frac{1}{\lambda(x)}pq\,dx&=\,0&\quad\forall q\in\Q.
\end{align*}
This variational problem can be rewritten as follows:
{\em Find $(\kappa,(\bu,p))\in\R\times\H\times\Q$ such that}
\begin{equation}\label{def:limit_system_eigen_complete}
	\left\{
	\begin{array}{rcll}
a(\bu,\bv)      +b(\bv,p)&=&\kappa d(\bu,\bv)&\forall\bv\in\H,\\
b(\bu,q)     - c(p,q)  & =&\;0 &\forall q\in\Q,
\end{array}
\right.
\end{equation}
where the bilinear forms
$a:\H\times\H\to\R$,
$b:\H\times\Q\to\R$,
$c:\Q\times\Q\to\R$,
and
$d:\H\times\H\to\R$
are defined by
\begin{align*}
a(\bu,\bv)&:=2\int_{\O}\mu(x)\beps(\bu):\beps(\bv)\, dx,\quad
b(\bv,q):=-\int_{\O}q\vdiv\bv\, dx,\\
d(\bu,\bv)&:=\int_{\O} \bu \cdot\bv\, dx,\quad  c(p,q):=\int_{\O} \frac{1}{\lambda(x)}p\, q\, dx,
\end{align*}
for all $\bu,\bv\in\H$, and $p,q\in\Q$.

To perform the analysis, we need a suitable
norm which in particular depends on the parameters of the problem.
With this in mind, and for all $(\bv,q)\in\H\times Q$,
we define the following weighted norm,
\begin{equation*}
	\label{eq:normHQ}
	\vertiii{(\bv,q)}_{\H\times Q}^2:=\Vert\mu(x)^{1/2}\nabla\bv\Vert_{0,\O}^2 + \Vert\mu(x)^{-1/2}q\Vert_{0,\O}^2 +\Vert\lambda(x)^{-1/2}q\Vert_{0,\O}^2.
\end{equation*}
Moreover, in what follows, we assume that $\mu(x),1/\lambda(x)\in \mathrm{W}^{1,\infty}(\Omega)$.

Let us introduce and define  the kernel of $b(\cdot,\cdot)$ by 
\begin{equation*}
\X := \{ \bv \in \H\,:\, b(\bv,q)=0,\quad \forall \, q\in Q\}\,=\,\{ \bv \in \H\,:\, \vdiv \bv \,=\,0\,\,\,\, {\rm in}\,\, \Omega\},
\end{equation*}
and let us recall that the bilinear form $b(\cdot,\cdot)$ satisfies the inf-sup condition:
\begin{equation}\label{eq:inf-sup-b2}
\sup_{\stackrel{\scriptstyle\bv\in\H}{\bv\ne0}}\frac{\vert
  b(\bv,q)\vert}{\Vert\mu(x)^{1/2}\nabla\bv\Vert_{0,\O}}\ge \beta_2
  \Vert\mu(x)^{-1/2}q\Vert_{0,\O} \quad\forall q\in\Q,
\end{equation}
with an inf-sup constant $\beta_2>0$ only depending on $\Omega$; see \cite{MR3973678}, for instance. 

Now, we are in a position to introduce the solution operator which we denote by
$\bT_{\nu}: \L^2(\Omega)^d\rightarrow \L^2(\Omega)^d$ and is such that $\boldsymbol{f}\mapsto \bT_{\nu} \boldsymbol{f}:=\widehat{\bu}$, where the pair $(\widehat{\bu},\widehat{p})\in \H\times \Q$ is the solution of the following well posed source problem: Given $\boldsymbol{f}\in \L^2(\O)^d$ find $(\widehat{\bu},\widehat{p})\in \H\times \Q$ such that
\begin{equation}\label{def:system_source_complete}
	\left\{
	\begin{array}{rcll}
a(\widehat{\bu},\bv)      +b(\bv,\widehat{p})&=&d(\boldsymbol{f} ,\bv)&\forall\bv\in\H,\\
b(\widehat{\bu},q)     - c(\widehat{p},q)  & =&\;0 &\forall q\in\Q.
\end{array}
\right.
\end{equation}
Thanks to the  Babu\v{s}ka-Brezzi theory
we have that $\bT_{\nu}$ is well defined and the following estimate holds
\begin{equation*}\label{eq:strong-norm-bound-continuous}
\Vert\bT_{\nu}\boldsymbol{f}\Vert_{0,\O}\le\vertiii{ (\bT_{\nu}\boldsymbol{f},\widehat{p})} \lesssim\|\boldsymbol{f}\|_{0,\Omega},
\end{equation*}
with a constant depending on the inverse of $\mu$ and the Poincaré constant.
It is easy to check that $\bT_{\nu}$ is a selfadjoint
operator with respect to the $\L^2$ inner product. Also, let $\chi$ be a real number such that  $\chi\neq 0$. Notice that $(\chi,\boldsymbol{u})\in \mathbb{R}\times\H$ is an eigenpair of $\bT_\nu$ if and only if  there exists $p\in \Q$ such that,  $(\kappa,(\boldsymbol{u}, p))$ solves problem \eqref{def:limit_system_eigen_complete}  with $\chi:=1/\kappa$.

The following result states an additional regularity
result for problem \eqref{def:system_source_complete}.
\begin{lemma}
\label{rmrk:additional}
Let $(\widehat{\bu},\widehat{p})$ be the unique solution
of \eqref{def:system_source_complete}, then there exists
$s\in(0,1]$ such that $(\widehat{\bu},\widehat{p})\in \H^{1+s}(\O)^{d}\times \H^{s}(\O)$
and the following estimate holds (see for instance \cite{grisvard1986problemes} or  \cite{MR4430561})
\begin{equation*}
\|\widehat{\bu}\|_{1+s,\O}+\|\widehat{p}\|_{s,\O}\lesssim \|\boldsymbol{f}\|_{0,\O},
\end{equation*}
where the hidden constant depends on $\Omega$ and the Lam\'e coefficients.
\end{lemma}
Let us remark that Lemma~\ref{rmrk:additional} holds true
even for $\lambda\uparrow\infty$ (Further comments on this fact can be seen in \cite{MR3962898}).

As a consequence of this result, we conclude compactness of  $\bT_{\nu}$.
In fact, since $\H^{1+s}(\Omega)^d\hookrightarrow \L^2(\Omega)^d$
we have that $\bT_{\nu}$ is a compact operator. Hence, we are in position to establish the spectral characterization of $\bT_{\nu}$.
\begin{theorem}
The spectrum of $\bT_{\nu}$  satisfies $\sp(\bT_{\nu})=\{0\}\cup\{\chi_k\}_{k\in\mathbb{N}}$, where $\{\chi_k\}_{k\in\mathbb{N}}$ is a sequence of positive eigenvalues.
\end{theorem}

On the other hand, from  \cite{MR0975121} we have the following regularity result  for the eigenfunctions of the eigenproblem \eqref{def:limit_system_eigen_complete}.
\begin{theorem}
	\label{th:reg_velocity}
	If $(\kappa, (\bu,p))\in \mathbb{R}\times\H\times\Q$ solves \eqref{def:limit_system_eigen_complete}, then, there exists $r>0$ such that $\bu\in\H^{1+r}(\Omega)^d$ and $p\in \H^r(\Omega)$. Moreover, the following estimate  holds
	\begin{equation*}
		\|\bu\|_{1+r,\O}+\|p\|_{r,\O}\lesssim \|\bu\|_{0,\O}.
	\end{equation*}
\end{theorem}
Now,
Let us end this section introducing the bilinear form
$\mathcal{A}:(\H\times \Q)\times(\H\times\Q)\rightarrow\mathbb{R}$ defined by
\begin{equation}
\label{eq:form_A}
\mathcal{A}((\bu,p),(\bv,q)):=a(\bu,\bv)+b(\bv,p)+b(\bu,q)+ c(p,q),
\end{equation}
which allows us to rewrite the eigenvalue problem as follows. Find $\kappa\in\mathbb{R}$ and $(\boldsymbol{0},0)\neq (\bu,p)\in\H\times\Q$ such that
\begin{equation*}
\label{eq:eigen_A}
\mathcal{A}((\bu,p),(\bv,q))=\kappa d(\bu,\bv)\quad\forall (\bv,q)\in\H\times Q.
\end{equation*}
Let us remark that $\mathcal{A}(\cdot,\cdot)$ is a bounded bilinear form, in the sense that there exists a positive constant $\widehat{C}:=\max\{2,\mu_{min}^{-1/2}\mu_{max}^{1/2}\}$
such that
\begin{equation*}
\mathcal{A}((\bu,p),(\bv,q))\leq\widehat{C}\vertiii{(\bu,p)}_{\H\times Q}\vertiii{(\bv,q)}_{\H\times Q}\quad\forall(\bu,p),(\bv,q)\in\H\times\Q.
\end{equation*}
Next we state the stability result for the bilinear form $\mathcal{A}$.
\begin{lemma}\label{lem:stab-A}
For any $(\bu,p)\in\H\times\Q$, there exists $(\bv,q)\in\H\times\Q$ with $\vertiii{(\bv,q)}\lesssim \vertiii{(\bu,p)}$ such that
\begin{align*}
\vertiii{(\bu,p)}^2\lesssim \mathcal{A}((\bu,p),(\bv,q)).
\end{align*}
\end{lemma}
\begin{proof}
Recalling the definition of the bilinear form $\mathcal{A}(\cdot,\cdot)$, we find
\begin{align*}
\mathcal{A}((\bu,p),(\bu,-p))=\|(2\mu)^{1/2} \beps(\bu)\|^2_{0,\Omega}+\|\lambda^{-1/2} p\|_{0,\Omega}^2.
\end{align*}
Thanks to inf-sup condition \eqref{eq:inf-sup-b2}, for any $p\in \Q$, we find $\tilde{\bv}\in\H$ with $\|(2\mu)^{1/2}\beps(\tilde{\bv})\|_{0,\Omega}\le C_1\|(2\mu)^{-1/2}p\|_{0,\Omega}$ such that $C_2\|(2\mu)^{-1/2}p\|_{0,\Omega}^2\le b(\tilde{\bv},p)$.
This allows us to conclude  that
\begin{multline*}
\mathcal{A}((\bu,p),(\tilde{\bv},0))= a(\bu,\tilde{\bv})+b(\tilde{\bv},p)
\ge C_2\|(2\mu)^{-1/2}p\|_{0,\Omega}^2-C_1\|(2\mu)^{1/2}\beps(\bu)\|_{0,\Omega}\|(2\mu)^{-1/2}p||_{0,\Omega}\\
\ge (C_2-1/2\epsilon)\|(2\mu)^{-1/2}p||_{0,\Omega}^2-(C_1^2\epsilon/2) \|(2\mu)^{1/2}\beps(\bu)||_{0,\Omega}.
\end{multline*}
Making the specific choices $\bv:=\bu+\delta \tilde{\bv}$ and $q=-p$, we have
\begin{align*}
\mathcal{A}((\bu,p),(\bv,q))&=\mathcal{A}((\bu,p),(\bu,-p))+\delta\mathcal{A}((\bu,p),(\tilde{\bv},0))\\
&\ge \|(2\mu)^{1/2} \beps(\bu)\|^2_{0,\Omega}+\|\lambda^{-1/2} p\|_{0,\Omega}^2\\
&\quad+\delta((C_2-1/2\epsilon)\|(2\mu)^{-1/2}p\|_{0,\Omega}^2-(C_1^2\epsilon/2) \|(2\mu)^{1/2}\beps(\bu)\|_{0,\Omega})
\end{align*} 
Choosing $\epsilon=1/C_2$ and $\delta=C_2/C_1^2$ gives $\mathcal{A}((\bu,p),(\bv,q))\gtrsim \vertiii{(\bu,p)}^2$. 
Additionally, we have
\begin{align*}
\vertiii{(\bv,q)}^2=\vertiii{(\bu+\delta\tilde{\bv},-p)}^2\lesssim \vertiii{(\bu,p)}^2.
\end{align*}
The proof is complete.
\end{proof}
\subsection{The limit problem}
The solution operator $\bT_{\nu}$ depends strongly on the values of the Poisson ratio. This must be understood in the sense that  there exists a family of problems depending on $\nu$ that $\bT_{\nu}$ must solve. As we observe on the definition of the bilinear form $c(\cdot,\cdot)$, when $\lambda$ blows up, i.e. $\lambda\uparrow\infty$, then $c(\cdot,\cdot)\downarrow 0$ and hence, we are in presence of a Stokes-type of problem. So the natural question is what happens with the spectrum of $\bT_{\nu}$ when $\lambda\uparrow\infty$. Let us begin by  introducing the following eigenvalue problem: Find $(\kappa_0,(\bu_0,p_0))\in\mathbb{R}\times\H\times\Q$ such that
\begin{equation}\label{def:limit_system_eigen}
	\left\{
	\begin{array}{rcll}
a(\bu_0,\bv)      +\;b(\bv,p_0)&=&\;\kappa_0 d(\bu_0,\bv)&\forall\bv\in\H,\\
b(\bu_0,q)    & =&0 &\forall q\in\Q.
\end{array}
\right.
\end{equation}
For this eigenvalue problem, let us introduce the following solution operator $\bT_0:\L^2(\Omega)^d\rightarrow\L^2(\Omega)^d$ defined by $\boldsymbol{f}\mapsto \bT_0\boldsymbol{f}:=\widehat{\bu}_0$, where the pair  $(\widehat{\bu}_0,\widehat{p}_0)$ solves the following source problem: Find $(\widehat{\bu}_0,\widehat{p}_0)\in\H\times\Q$ such that
\begin{equation}\label{def:limit_system_source}
	\left\{
	\begin{array}{rcll}
a(\widehat{\bu}_0,\bv)     +b(\bv,\widehat{p}_0)&=&d(\boldsymbol{f} ,\bv)&\forall\bv\in\H,\\
b(\widehat{\bu}_0,q)     &=& 0& \forall q\in\Q.
\end{array}
\right.
\end{equation}
This operator $\bT_0$ is well defined, self-adjoint, and compact
and its spectral characterization is given as follows.
\begin{lemma}
The spectrum of $\bT_{0}$  satisfies $\sp(\bT_{0})=\{0\}\cup\{\chi_k^0\}_{k\in\mathbb{N}}$,
where $\{\chi_k^0\}_{k\in\mathbb{N}}$ is a sequence of
positive eigenvalues such that $\chi_k^0\downarrow 0$ as $k\uparrow+\infty$.
\end{lemma}
For the operators $\bT_{\nu}$ and $\bT_0$, we have the following result.
\begin{lemma}\label{fghtyu}
The following estimate holds
\begin{equation}\label{fghtyu1}
\Vert(\bT_{\nu}-\bT_0)\boldsymbol{f} \Vert_{0,\O}\le\vertiii{(\bT_{\nu}-\bT_0)\boldsymbol{f}}\lesssim {\|\lambda^{-1/2}\widehat{p}_0\|_{0,\O}\lesssim (2(1-2\nu)/\nu)^{1/2}\|\boldsymbol{f}\|_{0,\Omega}},\quad \forall\boldsymbol{f}\in \L^2(\Omega)^d,
\end{equation}
where the hidden constants are  independent of the Poisson ratio.
\end{lemma}
\begin{proof}
Let $\boldsymbol{f}\in\L^2(\Omega)^d$. Subtracting \eqref{def:limit_system_source} from \eqref{def:system_source_complete} we obtain 
\begin{align*}
\displaystyle 2\int_{\Omega}\mu(x)\beps(\widehat{\bu}-\widehat{\bu}_0):\beps(\bv)\, dx-\int_{\Omega}(\widehat{p}-\widehat{p}_0)\vdiv\bv\, dx&=0\quad\forall\bv\in\H\\
\displaystyle-\int_{\Omega}q\vdiv(\widehat{\bu}-\widehat{\bu}_0)\, dx+\int_{\Omega}\frac{\widehat{p}q}{\lambda}\, dx&=0\quad\forall q\in\Q.
\end{align*}
Moreover, it also holds:
\begin{align*}
\displaystyle 2\int_{\Omega}\mu\beps(\widehat{\bu}-\widehat{\bu}_0):\beps(\bv)-\int_{\Omega}(\widehat{p}-\widehat{p}_0)\vdiv\bv&=0\quad\forall\bv\in\H\\
\displaystyle-\int_{\Omega}q\vdiv(\widehat{\bu}-\widehat{\bu}_0)+\int_{\Omega}\frac{(\widehat{p}-\widehat{p}_0)q}{\lambda}&=-\int_{\Omega}\frac{\widehat{p}_0q}{\lambda}\quad\forall q\in\Q.
\end{align*}
Using stability result gives:
\begin{align*}
\vertiii{(\widehat{\bu}-\widehat{\bu}_0,\widehat{p}-\widehat{p}_0)}\lesssim \|\lambda^{-1/2}\widehat{p}_0\|_{0,\Omega}\lesssim (2(1-2\nu)/\nu)^{1/2}\|\mu(x)^{-1/2}\boldsymbol{f}\|_{0,\Omega}.
\end{align*}
This concludes the proof.
\end{proof}
As a consequence of this result,
$\bT_{\nu}$ converges to $\bT_0$ as $h\downarrow 0$. Then we have that standard properties
of separation of isolated parts of the spectrum yield the following result:

\begin{lemma}\label{fghtrb}
Let $\chi^0>0$ be an eigenvalue of $\bT_{0}$ of multiplicity $m$. Let $D$
be any disc in the complex plane centered at $\chi^0$ and containing
no other element of the spectrum of $\bT_{0}$. Then, for $\nu$ goes to $\frac12$,
$D$ contains exactly $m$ eigenvalues of $\bT_{\nu}$ (repeated according to their respective
multiplicities). Consequently, each eigenvalue $\chi^0>0$ of $\bT_{0}$
is a limit of eigenvalues $\chi>0$ of $\bT_{\nu}$, as $\nu$ goes to $\frac12$.
\end{lemma}


%

\section{Numerical discretization}
\label{sec:fem}
In this section our aim is to describe a finite element discretization of problem \eqref{def:limit_system_eigen_complete}. To do this task, we will introduce two families of inf-sup stable finite elements  whose properties also hold for the eigenvalue problem. We begin by introducing some preliminary definitions and notations to perform the analysis. Let $\mathcal{T}_h=\{T\}$ be a conforming partition  of $\overline{\O}$ into closed simplices $T$ with size $h_T=\text{diam}(T)$. Define $h:=\max_{T\in\mathcal{T}_h}h_T$. 
%
Given a mesh $\mathcal{T}_h\in\mathbb{T}$, we denote by $\H_h$ and $\Q_h$ the finite element spaces that approximate the velocity field and the pressure, respectively. In particular, our study is focused in the following two elections:
\begin{itemize}
\item[(a)] The mini element \cite[Section 4.2.4]{MR2050138}: Here,
\begin{align*}
&\H_h=\{\bv_{h}\in\boldsymbol{C}(\overline{\O})\ :\ \bv_{h}|_T\in[\mathbb{P}_1(T)\oplus\mathbb{B}(T)]^{d} \ \forall \ T\in\mathcal{T}_h\}\cap\H,\\
&\Q_h=\{ q_{h}\in C(\overline{\O})\ :\ q_{h}|_T\in\mathbb{P}_1(T) \ \forall \ T\in\mathcal{T}_h \}\cap \Q,
\end{align*}
where $\mathbb{B}(T)$ denotes the space spanned by local bubble functions.
\item[(b)] The lowest order Taylor--Hood element \cite[Section 4.2.5]{MR2050138}: In this case,
\begin{align*}
&\H_h=\{\bv_{h}\in\boldsymbol{C}(\overline{\O})\ :\ \bv_h|_T\in[\mathbb{P}_2(T)]^{d} \ \forall \ T\in\mathcal{T}_h\}\cap\H,\\
\label{eq:P_TH}
&\Q_h=\{ q_{h}\in C(\overline{\O})\ :\ q_{h}|_T\in\mathbb{P}_1(T) \ \forall\ T\in\mathcal{T} \}\cap \Q.
\end{align*}
\end{itemize}
The discrete analysis will be performed in a general manner, where the both families of finite elements, namely Taylor-Hood and mini element, are considered with no difference. If some difference must be claimed, we will point it out when is necessary.

\subsection{The discrete eigenvalue problem}
Now we are in position to  introduce the discrete counterpart of the eigenvalue problem \eqref{def:limit_system_eigen_complete}: Find $(\kappa_h,\bu_h,p_h)\in\R\times\H\times\Q$ such that
\begin{equation}\label{def:limit_system_eigen_complete_disc}
	\left\{
	\begin{array}{rcll}
a(\bu_h,\bv_h)      +b(\bv_h,p_h)&=&\kappa_h d(\bu_h,\bv_h)&\forall\bv_h\in\H_h,\\
b(\bu_h,q_h)     - c(p_h,q_h)  & =&\;0 &\forall q_h\in\Q_h.
\end{array}
\right.
\end{equation}
Since the choice of mini element and Taylor-Hood are inf-sup stable, we are in position to introduce the discrete solution operator $\bT_{\nu,h}:\L^2(\Omega)^d\rightarrow \H_h$ defined by $\boldsymbol{f}\mapsto \bT_{\nu,h}:=\widehat{\bu}_h$, where the pairs $(\widehat{\bu}_h,\widehat{p}_h)\in\H_h\times\Q_h$ solves the following discrete source problem
\begin{equation*}
\label{def:limit_system_source_complete_disc}
	\left\{
	\begin{array}{rcll}
a(\widehat{\bu}_h,\bv_h)      +b(\bv_h,\widehat{p}_h)&=& d(\boldsymbol{f},\bv_h)&\forall\bv_h\in\H_h,\\
b(\widehat{\bu}_h,q_h)     - c(\widehat{p}_h,q_h)  & =&\;0 &\forall q_h\in\Q_h.
\end{array}
\right.
\end{equation*}
This problem is clearly well posed and the solutions satisfies the following estimate
\begin{equation*}
\Vert\widehat{\bu}_h\Vert_{0,\O}\le \vertiii{(\widehat{\bu}_h,\widehat{p}_h)}\lesssim \|\boldsymbol{f}\|_{0,\O}.
\end{equation*}
As in the continuous case, we have that $\kappa_h$
is an eigenvalue of problem \eqref{def:limit_system_eigen_complete_disc}
if and only if $\displaystyle\chi_h:=\frac{1}{\kappa_h}$ is an
eigenvalue of $\bT_{\nu,h}$ with the
same multiplicity and corresponding eigenfunctions.

Moreover, if $\mathcal{L}_h$ represents the Lagrange interpolator operator and $\mathcal{P}_h$ represents the $\L^2$-projection,   the solutions of the continuous and discrete problems satisfy
\begin{multline*}
\vertiii{(\widehat{\bu}-\widehat{\bu}_h,\widehat{p}-\widehat{p}_h)}^2\lesssim\inf_{(\bv,q)\in\H_h\times\Q_h}\vertiii{(\widehat{\bu}-\bv,\widehat{p}-q)}^2\lesssim\vertiii{(\widehat{\bu}-\mathcal{L}_h\widehat{\bu},\widehat{p}-\mathcal{P}_h\widehat{p})}^2\\
=\left(\|\mu(x)^{1/2}\nabla(\widehat{\bu}-\mathcal{L}_h\widehat{\bu})\|_{0,\O}^2+\|\mu(x)^{-1/2}(\widehat{p}-\mathcal{P}_h\widehat{p})\|_{0,\O}^2+\|\lambda(x)^{-1/2}(\widehat{p}-\mathcal{P}_h\widehat{p})\|_{0,\O}^2 \right)\\
\lesssim \left(E_{\max}\|\nabla(\widehat{\bu}-\mathcal{L}_h\widehat{\bu})\|_{0,\O}^2+E_{\min}^{-1}\|(\widehat{p}-\mathcal{P}_h\widehat{p})\|_{0,\O}+E_{\min}^{-1}(1-2\nu)\|(\widehat{p}-\mathcal{P}_h\widehat{p})\|_{0,\O} \right)\\
\lesssim h^{2s}\max\{E_{\max}, E_{\min}^{-1}(1+(1-2\nu)) \}\left(\|\widehat{\bu}\|_{1+s,\O}^2+\|\widehat{p}\|_{s,\O}^{2} \right),
\end{multline*}
Hence, from Lemma \ref{rmrk:additional} we obtain
\begin{equation}
\label{eq:error_triple}
\vertiii{(\widehat{\bu}-\widehat{\bu}_h,\widehat{p}-\widehat{p}_h)}\lesssim h^s\left(\|\widehat{\bu}\|_{1+s,\O}+\|\widehat{p}\|_{s,\O}\right) \lesssim h^s\Vert \boldsymbol{f}\Vert_{0,\O}.
\end{equation}

Similar as the continuous case, with the aid of the bilinear form $\mathcal{A}(\cdot,\cdot)$ (cf. \eqref{eq:form_A}) we can rewrite the discrete eigenvalue problem \eqref{def:limit_system_eigen_complete_disc} as follows: Find $\kappa_h\in\mathbb{R}$ and $(\boldsymbol{0},0)\neq (\bu_h,p_h)\in\H_h\times\Q_h$ such that
\begin{equation*}
\label{eq:eigen_Ah}
\mathcal{A}((\bu_h,p_h),(\bv_h,q_h))=\kappa_h d(\bu_h,\bv_h)\quad\forall (\bv_h,q_h)\in\H_h\times \Q_h.
\end{equation*}
%

Now we prove an error estimate in $\L^2$ for the displacement  which we prove by using a standard duality argument. This estimate will be helpful to establish the convergence in norm of the solution operators.
\begin{lemma}
\label{lmm:l2_error}
There following estimate holds 
\begin{align*}
||\widehat{\bu}-\widehat{\bu}_h||_{0,\O}\lesssim h^s \vertiii{(\widehat{\bu}-\widehat{\bu}_h,\widehat{p}-\widehat{p}_h)}.
\end{align*}
where the hidden constant is  independent of $\nu$ and $s$ is as in Lemma \ref{rmrk:additional}.
\end{lemma}
\begin{proof}
Assume that $(\widehat{\phi},\widehat{\psi})\in\H\times\Q$ solves the following dual source problem  
\begin{equation*}\label{def:dual_system_source_complete_disc}
	\begin{array}{rcll}
\mathcal{A}((\widehat{\phi},\widehat{\psi}),(\bv,q))&=& d(\widehat{\bu}-\widehat{\bu}_h,\bv)&\qquad\forall(\bv,q)\in\H\times \Q,
\end{array}
\end{equation*}
Let us assume that  $(\widehat{\phi},\widehat{\psi})$ is the unique solution
of \eqref{def:system_source_complete}, then there exists $s\in(0,1]$
such that the following estimate holds
\begin{equation}\label{eq:phi-psi-estimate}
	\|\widehat{\phi}\|_{1+s,\O}+\|\widehat{\psi}\|_{s,\O}\lesssim \|\widehat{\bu}-\widehat{\bu}_h\|_{0,\O}.
\end{equation}
Making a specific choice $\bv=\widehat{\bu}-\widehat{\bu}_h$ and $q=\widehat{p}-\widehat{p}_h$, we find that
\begin{align*}
||\widehat{\bu}-\widehat{\bu}_h||_{0,\O}^2=\mathcal{A}((\widehat{\phi},\widehat{\psi}),(\widehat{\bu}-\widehat{\bu}_h,\widehat{p}-\widehat{p}_h)).
\end{align*}
Using the fact that $\mathcal{A}((\widehat{\bu}-\widehat{\bu}_h,\widehat{p}-\widehat{p}_h),(\widehat{\phi}_h,\widehat{\psi}_h))=0$
we obtain 
\begin{align*}
||\widehat{\bu}-\widehat{\bu}_h||_{0,\O}^2=\mathcal{A}((\widehat{\phi}-\widehat{\phi}_h,\widehat{\psi}-\widehat{\psi}_h),(\widehat{\bu}-\widehat{\bu}_h,\widehat{p}-\widehat{p}_h)).
\end{align*}
Using the Cauchy-Schwarz, estimate \eqref{eq:phi-psi-estimate}, the approximation results and the stability result leads to the stated result.
\end{proof}
Let us prove the key result that establish the convergence of  $\bT_{\nu,h}$ to $\bT_{\nu}$ as $h\downarrow 0$.
\begin{lemma}
\label{lmm:conv_norm}
Given $\boldsymbol{f}\in \L^2(\O)^d$, the following estimate holds,
independent of $\nu$, such that
\begin{equation*}
\|(\bT_\nu-\bT_{\nu,h})\boldsymbol{f}\|_{0,\O}\lesssim h^{2s}\|\boldsymbol{f}\|_{0,\O},
\end{equation*}
where the hidden constant is  independent of $\nu$ and $s$ is as in Lemma \ref{rmrk:additional}.
\end{lemma}
\begin{proof}
Let $\boldsymbol{f}\in \L^2(\O)^d$. Then
\begin{equation*}
\|(\bT_\nu-\bT_{\nu,h})\boldsymbol{f}\|_{1,\O}=\|\widehat{\bu}-\widehat{\bu}_h\|_{1,\O}=\|\widehat{\bu}-\widehat{\bu}_h\|_{0,\O}+\|\nabla(\widehat{\bu}-\widehat{\bu}_h)\|_{0,\O}.
\end{equation*}
For the gradient error we have
\begin{multline}
	\label{eq:grad_error}
	\|\nabla(\widehat{\bu}-\widehat{\bu}_h)\|_{0,\O}\lesssim \mu_{\min}^{-1/2}\|\mu(x)^{1/2}\nabla(\widehat{\bu}-\widehat{\bu}_h)\|_{0,\O}\\
	\lesssim \mu_{\min}^{-1/2}\|\mu(x)^{1/2}\nabla(\widehat{\bu}-\widehat{\bu}_h)\|_{0,\O}+\|\mu(x)^{-1/2}(\widehat{p}-\widehat{p}_h)\|_{0,\O}+\|\lambda(x)^{-1/2}(\widehat{p}-\widehat{p}_h)\|_{0,\O}\\
	\lesssim h^s\max\{\mu_{\min}^{-1/2},1\}\left(\|\widehat{\bu}\|_{1+s,\O}+\|\widehat{p}\|_{s,\O}\right) \lesssim h^s \max\{\mu_{\min}^{-1/2},1\}\Vert \boldsymbol{f}\Vert_{0,\O},
\end{multline}
where we have used the definition of $\vertiii{\cdot}$ and \eqref{eq:error_triple}.  On the other hand, the error $\|\widehat{\bu}-\widehat{\bu}_h\|_{0,\O}$ is clearly estimated, first by Lemma \ref{lmm:l2_error} and then by \eqref{eq:error_triple} leading to
\begin{equation}
\label{eq:dsiplacement_l2}
\|\widehat{\bu}-\widehat{\bu}_h\|_{0,\O}\lesssim h^{2s}\|\boldsymbol{f}\|_{0,\O}.
\end{equation}
Finally, combining \eqref{eq:dsiplacement_l2} and \eqref{eq:grad_error} we conclude the proof.
\end{proof}

From the previous lemma, we can conclude the convergence
in norm for $\bT_{\nu,h}$ to $\bT_\nu$ as $h \downarrow 0$.
This is a key ingredient in order to obtain error estimates
for eigenvalues and eigenfunctions.
In fact, an immediate consequence of
Lemma~\ref{lmm:conv_norm}
is that isolated parts of $\sp(\bT_\nu)$ are approximated
by isolated parts of $\sp(\bT_{\nu,h})$.
It means that if $\chi$ is a nonzero eigenvalue of $\bT_\nu$ with algebraic
multiplicity $m$, hence there exist $m$ eigenvalues
$\chi_h^{(1)},\ldots,\chi_h^{(m)}$ of $\bT_{\nu,h}$
(repeated according to their respective multiplicities)
that will converge to $\mu$ as $h$ goes to zero.

In what follows, we denote by $\mathfrak{E}$ the eigenspace of $\bT_\nu$
associated to the eigenvalue $\chi$ and by $\mathfrak{E}_h$
the invariant subspace of $\bT_{\nu,h}$
spanned by the eigenspaces of $\bT_{\nu,h}$ associated to $\chi_h^{(1)},\ldots,\chi_h^{(m)}$.

Now our  goal of this section is deriving error estimates
for the eigenfunctions and eigenvalues. With this goal in mind,
first  we need to recall the definition of the  gap $\widehat{\delta}$
between two closed subspaces $\CM$ and $\CN$ of  $\L^{2}(\Omega)$:
\begin{equation*}
\widehat{\delta}(\CM,\CN) := \max\{\delta (\CM, \CN), \delta (\CM, \CN)\},
\end{equation*}
where
\begin{equation*}
\delta (\CM,\CN) := \underset{x \in \CM : \|x\|_{0, \Omega} = 1}{\sup} \left\lbrace \underset{y \in \CN}{\inf} \quad \|x-y\|_{0, \Omega} \right\rbrace.
\end{equation*}
We also define  $$\gamma_h:=\sup\limits_{\boldsymbol{v}\in \mathcal{E}:\Vert \boldsymbol{v}\Vert_{0,\O}=1}
\Vert(\bT_{\nu}-\bT_{\nu,h})\boldsymbol{v}\Vert_{0,\O}.$$
In order to obtain $\nu$-independent error estimates,
in what follows, we assume that $\chi$ is an eigenvalue
of $\bT_{\nu}$ converging to a simple eigenvalue $\chi_0$
of $\bT_{0}$, as $\nu$ goes to $\frac12$.
The following error estimates for the approximation
of eigenvalues and eigenfunctions hold true.
The result can be obtained from 
from \cite[Theorems 7.1 and 7.3]{BO}.

\begin{theorem}
\label{gap}
There exists a strictly positive constant $C$,
independent of $\nu$ and $h$, such that
\begin{align}
\widehat{\delta}(\mathfrak{E},\mathfrak{E}_h)
& \lesssim \gamma_h,\label{estrfg1}\\
\left|\chi-\chi_h\right|
& \lesssim \gamma_h.\label{estrfg2}
\end{align}
\end{theorem}
\begin{proof}
On the one hand, We have that \eqref{estrfg1} is a direct consequence of Theorem 7.1 in \cite{BO},
with a constant $C>0$ depending on the constant in \eqref{eq:dsiplacement_l2}
(which is independent of $\nu$) and on the inverse of the distance
of $\chi$ to the rest of the spectrum of $\bT_\nu$.
Then, from Lemma~\ref{fghtrb}, \eqref{fghtyu1} implies that
for $\nu$ goes to $\frac12$ this distance can be
bounded below in terms of the distance of $\chi_0$
to the rest of the spectrum of $\bT_{0}$, which
does not depend on Poisson ratio.
On the other hand, estimate \eqref{estrfg2}
follows from Theorem 7.3 in \cite{BO}.
\end{proof}
Moreover, by using the additional regularity of the eigenfunctions,
we obtain the following result.
\begin{theorem}\label{gapr}
There following estimates hold 
\begin{align}
&\Vert(\bT_{\nu}-\bT_{\nu,h})\boldsymbol{v}\Vert_{0,\O}
\lesssim h^{2r}\|\boldsymbol{v}\|_{0,\O}\qquad \forall
\boldsymbol{v}\in \mathfrak{E},\label{bou_gamma_h}
\end{align}
and as a consequence,
\begin{align}
&\gamma_h \lesssim  h^{2r}, \label{bound1r}
\end{align}
where the hidden constants are independent of $\nu$ and $h$ and $r>0$ is as in Theorem \ref{th:reg_velocity}.
\end{theorem}
\begin{proof}
The inequality~\eqref{bou_gamma_h} is obtained repeating
the proof of Lemma~\ref{lmm:conv_norm} and Theorem \ref{th:reg_velocity}. Estimate \eqref{bound1r} follows from the definition of $\gamma_h$.
\end{proof}
Finally, we have the following result for vibration frequencies
of the elasticity eigenproblem, which is an immediate consequence of \eqref{estrfg2}.
\begin{corollary}\label{rtyuj}
The following estimate holds
\begin{equation*}
\label{eq:double_order}
|\kappa-\kappa_h|\lesssim  h^{2r},
\end{equation*}
where the hidden constant is  independent of $\nu$ and $h$ and $r>0$ is as in Theorem \ref{th:reg_velocity}.
\end{corollary}

\section{A posteriori error analysis}
\label{sec:apost}
The aim of this section is to introduce a suitable
residual-based error estimator for the Stokes/elasticity
eigenvalue problem. The main aim is to prove that the proposed estimator  is equivalent  with the error.
Moreover, on the forthcoming analysis we will focus only on eigenvalues with multiplicity
1. With  this purpose, we introduce the following definitions
and notations. For any element  $K\in \CT_h$, we denote by $\CE_{K}$ the set of facets of $K$
and 
$$\CE_h:=\bigcup_{K\in\CT_h}\CE_{K}.$$
We decompose $\CE_h=\CE_{\O}\cup\CE_{\partial\O}$,
where  $\CE_{\partial\O}:=\{\ell\in \CE_h:\ell\subset \partial\O\}$
and $\CE_{\O}:=\CE\backslash\CE_{\partial \O}$.
For each inner edge $\ell\in \CE_{\O}$ and for any  sufficiently smooth  function
$\bv$, we define the jump of its normal derivative on $\ell$ by
$$\left[\!\!\left[ \dfrac{\partial \bv}{\partial{ \boldsymbol{n}}}\right]\!\!\right]_\ell:=\nabla (\bv|_{K})  \cdot \boldsymbol{n}_{K}+\nabla ( \bv|_{K'}) \cdot \boldsymbol{n}_{K'} ,$$
where $K$ and $K'$ are  the two elements in $\CT_{h}$  sharing the
edge $\ell$ and $\boldsymbol{n}_{K}$ and $\boldsymbol{n}_{K'}$ are the respective outer unit normal vectors.


\subsection{Residual based a posteriori error estimator}
 The local element residual $(\eta_K)$ and  the edge jump-residual estimator $(\eta_J)$ are defined as follows:
 \begin{equation*}
 \eta_K^2:= h_K^2\|\rho_1\textbf{R}_{1,T}\|_{0,T}^2+\|\rho_2^{1/2}R_{2,T}\|_{0,T}^2,\quad \eta_J^2:=h_E\|\rho_EJ_\ell \|_{0,\ell}^2
 \end{equation*}
 where 
 \begin{equation*}
 \textbf{R}_{1,T} := \bdiv(2\mu_h\beps(\bu_h)) -\nabla p_h + \rho \kappa \bu_h \quad\text{and}\quad
R_{2,T} := \vdiv\bu_h +\frac{1}{\lambda}p_h.
 \end{equation*}
For the jump terms, we define 
 \begin{align*}
 J_\ell &:= \begin{cases}
\displaystyle \frac{1}{2}[\![{(2\mu_h\beps(\bu_h) -p_h\mathbb{I})\textbf{n}}]\!], & E\in \Omega\cap\mathcal{E}_h\\
  {(2\mu_h\beps(\bu_h) -p_h\mathbb{I})\textbf{n}}, & E\in \Gamma_N\cap\mathcal{E}_h\\
  \boldsymbol{0}, & E\in \Gamma_D\cap\mathcal{E}_h.
 \end{cases}
 \end{align*}
 Here, the parameters $\rho_1$, $\rho_2$ and $\rho_3$ are defined by
$$
 	\rho_{1,K}:=  (2\mu_h)^{-1/2},\quad  \rho_2:= \left[(2\mu_h)^{-1}+\lambda^{-1}\right]^{-1},\quad
 	\rho_E := (2\mu_{h})^{-1/2}/\sqrt{2}, 
$$
where $\mu_h$ correspond to the $\L^2$ polynomial projection of $\mu$.  The local data oscillation is defined as $ \Theta^2_K=\|\rho_1(\mu-\mu_h)\beps(\bu_h)\|^2_0$.
 Finally, we discuss the definition of the global a posteriori estimatior $\eta$ and the global data oscillation error as
 \begin{equation*}
 \eta: = \left(\sum_{K\in\mathcal{T}_h}\eta_K^2\right)^{1/2},\quad \Theta := \left(\sum_{K\in\mathcal{T}_h}\Theta_K^2\right)^{1/2}.
 \end{equation*}
 \subsection{Reliability}

The task here is to prove that the error is upper bounded by the estimator, together with 
the high order terms. This is stated in the following result.
\begin{theorem}[Reliability]\label{th:reliability}

 Let $(\kappa,(\bu,p))\in\mathbb{R}\times\H\times\Q$ be a solution of the spectral problem \eqref{def:limit_system_eigen_complete} and let $(\kappa_h,(\bu_h,p_h))\in\mathbb{R}\times\H_h\times\Q_h$ be the finite element approximation of $(\kappa,(\bu,p))$ given as the solution of \eqref{def:limit_system_eigen_complete_disc}. Then, for every $h_0\geq h$ there holds
		\begin{equation*}\label{eq:reliabilityprimal}
			 \vertiii{(\bu-{\bu}_h,p-p_h)}\lesssim \eta+\Theta+|\kappa-\kappa_h|+\kappa\|\bu-\bu_h\|_{0,\O}.
		\end{equation*}
		where the hidden constant is independent of the mesh size, $\nu$ and the  discrete solutions. Additionally, the following reliability of the eigenvalue also holds:
  \begin{align*}
|\kappa-\kappa_h|\lesssim \eta^2+\Theta^2.
 \end{align*} 
\end{theorem}
\begin{proof}
Given $(\bv,q)\in\H_h\times\Q_h$, we subtract the continuous formulation (cf. \eqref{def:limit_system_eigen_complete}) and its discrete counterpart (cf. \eqref{def:limit_system_eigen_complete_disc}) in order to obtain the error equation
	\begin{equation}
		\label{lem:reliability-eq001}
		\mathcal{A}((\texttt{e}_{\bu},\texttt{e}_{p}),(\bv,q))=(\kappa\bu - \kappa_h\bu_h,\bv)_{0,\O}\quad\forall (\bv,q)\in\H_h\times\Q_h.
	\end{equation}
	
Thanks to  Lemma \ref{lem:stab-A}, we have that for $(\bu,p)\in\H\times\Q$, there exists $(\bw,r)\in\H\times\Q$, with
	\begin{equation*}
		\label{lem:reliability-eq002}
		\vertiii{(\bw,r)} \lesssim \vertiii{(\bu-{\bu}_h,p-\widehat{p}_h)},
	\end{equation*}
	such that 
	\begin{equation}
		\label{lem:reliability-eq003}
		\vertiii{(\bu-{\bu}_h,p-\widehat{p}_h)}^2  \lesssim \mathcal{A}((\texttt{e}_{\bu},\texttt{e}_p),(\bw,r)).
	\end{equation}
	Let us define $\bw_I\in \H_h$ as the Cl\'ement interpolant (see \cite[Chapter 2]{koelink2006partial}) and $r_I\in \Q_h$ as the standard  $\L^2$-projection of $r$. Then, from \eqref{lem:reliability-eq001} and \eqref{lem:reliability-eq003} we obtain
	 \begin{multline*}
		\label{lem:reliability-eq004}
		\vertiii{(\bu-{\bu}_h,p-p_h)}   \leq\kappa(\bu,\bw)_{0,\O}-\mathcal{A}((\bu_h,p_h);(\bw,r)) \\
		= (\kappa\bu-\kappa_{h}\bu_h,\bw)_{0,\O}+\kappa_h(\bu_h,\bw-\bw_I)_{0,\O} - \mathcal{A}((\bu_{h},p_h);(\bw-\bw_I,r-r_I))  \\
		= \underbrace{-\mathcal{A}((\bu_{h},p_h);(\bw - \bw_I,r - r_I)) + (\kappa_h\bu_h,\bw-\bw_I)_{0,\O}}_{\Lambda_1}
		+\underbrace{ (\kappa\bu - \kappa_h\bu_h,\bw)_{0,\O}}_{\Lambda_2}.
	\end{multline*}
An application of integration by parts together with Cauchy--Schwarz and the approximation results of the Cl\'ement interpolant  gives 
\begin{equation*}
|\Lambda_1|\lesssim \eta \vertiii{(\bw,r)} , \quad
 |\Lambda_2|\lesssim \|\mu_h^{1/2}(\kappa\bu-\kappa_h\bu_h)\|_{0,\O}\vertiii{(\bw,r)}.
\end{equation*}
This completes the proof. Using 
\[(\kappa-\kappa_h)d(\bu_h,\bu_h)= \mathcal{A}((\bu-\bu_h,p-p_h); (\bu-\bu_h,p-p_h))-\kappa d(\bu-\bu_h,\bu-\bu_h)\]
with the above result leads to the second stated result.
\end{proof}
 \subsection{Efficiency}
The analysis for the efficiency follows the classic arguments based on bubble functions.  An interior bubble function $\psi_{K}\in\H_0^1(K)$ for a polytope $K$
can be constructed piecewise as the sum of the  cubic
bubble functions  for each tetrahedron of the
sub-triangulation $\CT_h^{K}$ that attain the value 1 at the barycenter of the tetrahedron. On the other hand, a facet bubble function $\psi_{\ell}$ for
$\ell\in\partial K$ is a piecewise quadratic function
attaining the value 1 at the barycenter of $\ell$ and vanishing
on the polytope $K\in\CT_h$  that do not contain $\ell$
on its boundary.

The following results which establish standard estimates
for bubble functions will be useful in what follows (see \cite{MR1885308,MR3059294}).
\begin{lemma}[Interior bubble functions]
\label{burbujainterior}
For any $K\in \CT_h$, let $\psi_{K}$ be the corresponding interior bubble function.
Then, there holds
\begin{align*}
\|q\|_{0,K}^2&\lesssim \int_{K}\psi_{K} q^2\leq \|q\|_{0,K}^2\qquad \forall q\in \mathbb{P}_k(K),\\
\| q\|_{0,K}&\lesssim \|\psi_{K} q\|_{0,K}+h_K\|\nabla(\psi_{K} q)\|_{0,K}\lesssim \|q\|_{0,K}\qquad \forall q\in \mathbb{P}_k(K),
\end{align*}
where the hidden constants are 
independent of  $h_K$.
\end{lemma}
\begin{lemma}[Facet bubble functions]
\label{burbuja}
For any $K\in \CT_h$ and $\ell\in\CE_{K}$, let $\psi_{\ell}$
be the corresponding facet bubble function. Then, there holds
 \begin{equation*}
\|q\|_{0,\ell}^2\lesssim \int_{\ell}\psi_{\ell} q^2 \leq \|q\|_{0,\ell}^2\qquad
\forall q\in \mathbb{P}_k(\ell).
\end{equation*}
Moreover, for all $q\in\mathbb{P}_k(\ell)$, there exists an extension of  $q\in\mathbb{P}_k(K)$ (again denoted by $q$) such that
 \begin{align*}
h_K^{-1/2}\|\psi_{\ell} q\|_{0,K}+h_K^{1/2}\|\nabla(\psi_{\ell} q)\|_{0,K}&\lesssim \|q\|_{0,\ell},
\end{align*}
where the hidden constants are independent of  $h_K$.
\end{lemma}

Now, we are in a position to establish the efficiency $\eta$, which is stated in the following result.
\begin{lemma}(Efficiency) The following estimate holds 
$$\eta\lesssim \vertiii{(\bu-{\bu}_h,p-p_h)}+\Theta+h.o.t,$$
here the hidden constant is independent of the meshsize, $\nu$, and the discrete solution.
\end{lemma}
\begin{proof}
Also, we define $\bv_T:=\psi_T h_T^2\rho_1^2\textbf{R}_{1,T}$, where $\psi_T$ is the bubble that satisfies  the  properties of Lemma \ref{burbujainterior}. Now we compute a bound for the term $h_T^2\|\rho_1\textbf{R}_{1,T}\|_{0,T}$. To do this task, let us recall that 
the continuous problem satisfies $-\bdiv(2\mu(x)\beps(\bu)) +\nabla p  -
 \rho \kappa \bu=\boldsymbol{0}$. Hence
\begin{multline*}
h_T^2\|\rho_1\textbf{R}_{1,T}\|_{0,T}^2\leq\int_T(\bdiv(2\mu_h(x)\beps(\bu_h)) -\nabla p_h  +
 \rho \kappa_h \bu_h)\cdot\bv_T\\
=\int_T(-\bdiv(2\mu_h(x)\beps(\texttt{e}_{\bu})) +\nabla \texttt{e}_p  +
 \rho \kappa_h \bu_h-\rho\kappa\bu-\bdiv(2 \texttt{e}_\mu(x)\beps({\bu})))\cdot\bv_T\\
=\underbrace{\int_T(-\bdiv(2\mu_h(x)\beps(\texttt{e}_{\bu})) +\nabla \texttt{e}_p -\bdiv(2 \texttt{e}_\mu(x)\beps({\bu})))\cdot\bv_T}_{\textrm{T}_1}+\underbrace{\int_T(\rho\kappa_h\bu_h-\rho\kappa\bu)\cdot\bv_T}_{\textrm{T}_2}.
\end{multline*}
To establish the bound of $\textrm{T}_1$, we apply integration by parts which gives
\begin{multline*}
T_1= \int_T(-\bdiv(2\mu_h(x)\beps(\texttt{e}_{\bu})) +\nabla \texttt{e}_p -\bdiv(2 \texttt{e}_\mu(x)\beps({\bu})))\cdot\bv_T\\
= \int_T((2\mu_h(x)\beps(\texttt{e}_{\bu})) - \texttt{e}_pI + \texttt{e}_\mu(x)\beps({\bu}) )\cdot \nabla \bv_T
\le (\vertiii{(\bu-{\bu}_h,p-p_h)}_T+\Theta_K)h_T\|\rho_1\textbf{R}_{1,T}\|_{0,T}.
\end{multline*}
For the term $\textrm{T}_2$ is enough to follow the proof of \cite[Theorem 3.2]{MR2473688} in order to obtain
\begin{equation*}
\textrm{T}_2\lesssim h_T^2(|\kappa-\kappa_h|+\kappa\|\texttt{e}_{\bu}\|_{0,T})\|\rho_1\textbf{R}_{1,T}\|_{0,T}.
\end{equation*}
Recalling the definition of $R_{2,T}$ and the continuous problem satisfies $\vdiv\bu +1/\lambda \;p =0$, we find 
\begin{equation*}
\|\rho_2^{1/2}R_{2,T}\|_{0,T}=\|\rho_2^{1/2}(\vdiv(\bu-\bu_h +1/\lambda (p-p_h))\|_{0,T}
\lesssim \vertiii{(\bu-{\bu}_h,p-p_h)}_T.
\end{equation*}
 Next, we define $\bv_\ell:=\psi_\ell h_E\rho_E^2[\![{(2\mu_h\beps(\bu_h) -p_h\mathbb{I})\textbf{n}}]\!]$, where $\psi_\ell$ is the bubble that satisfies  the  properties of Lemma \ref{burbuja}. Now we compute a bound for the term $h_E\|\rho_EJ_{\ell}\|_{0,\ell}^2$.  Then, we obtain
 \begin{align}\label{burbuja11}
 h_E\|\rho_EJ_{\ell}\|_{0,\ell}^2&\lesssim ([\![{(2\mu_h\beps(\bu_h) -p_h\mathbb{I})\textbf{n}}]\!], \bv_\ell)_\ell=  ([\![{(2\mu_h\beps(\bu_h)-2\mu\beps(\bu) -(p_h-p)\mathbb{I}})\textbf{n}]\!], \bv_\ell)_\ell
 \end{align}
 Using Integration by parts on $\omega_T$ gives
 \begin{align*}
  ([\![{(2\mu_h\beps(\bu_h)-2\mu\beps(\bu) -(p_h-p)\mathbb{I})\textbf{n}}]\!], \bv_\ell)_\ell=\sum_{T\in\omega_T} &\Big((2\mu\beps(\bu)-2\mu_h(x)\beps({\bu}), \beps(\bv_\ell))_T +( \texttt{e}_p, \vdiv \bv_\ell)_T\\
  &\quad +(\textbf{R}_{1,T}, \bv_\ell)_T
  +( \rho \kappa_h \bu_h- \rho \kappa \bu, \bv_\ell)_T\Big).
 \end{align*}
 Applying Cauchy-Schwarz with Lemma \ref{burbuja} and combining with (\ref{burbuja11}) implies that
 \begin{align*}
 h_E\|\rho_EJ_{\ell}\|_{0,\ell}^2&\lesssim (\vertiii{(\bu-{\bu}_h,p-p_h)}_{\omega_T}+\Theta_{\omega_T}+h.o.t.) h_E^{1/2}\|\rho_EJ_{\ell}\|_{0,\ell}.
 \end{align*}
 Combining the above estimates leads to the the stated efficiency result.

\end{proof}

\section{Numerical results}
\label{sec:numerical-experiments}
This section is devoted to analyze the performance of the proposed finite element method through different experiments, considering different values of $\mu$ and $\lambda$. The aim is to compare the  computationally  the  spectrums of  the Stokes limit and elasticity eigenvalue problems. To do this task, we have implemented a  FEniCS  script  \cite{AlnaesBlechta2015a}  for the codes. The meshes have been constructed using the GMSH software \cite{geuzaine2009gmsh}.

 The convergence rates of the eigenvalues have been  obtained with a standard least square fitting and highly refined meshes.  The Taylor-Hood and mini-element families can be used indistinctively as observed in the theory provided in Section \ref{stab_section}. Taylor-Hood elements will be preferred for the 2D and 3D representations, whereas the adaptive test are performed with the mini-element family.

We also denote by $N$ the mesh refinement level, whereas $\texttt{dof}$ denotes the number of degrees of freedom.  We denote by $\kappa_{h_i}$ the i-th discrete eigenfrequency, and by $\widehat{\kappa}_{h_i}:=\kappa_{h_i}/(1+\nu)$ the unscaled version of $\kappa_{h_i}$.

Hence, we denote the error on the $i$-th eigenvalue by $\err( \widehat{\kappa}_i)$ with 
\begin{equation*}
\err(\widehat{\kappa}_i):=\vert \widehat{\kappa}_{h_i}-\widehat{\kappa}_{i}\vert.
\end{equation*}
We remark that computing $\sqrt{\widehat{\kappa}_{h_i}}$ or $\sqrt{\widehat{\kappa}_{i}}$ gives the corresponding discrete and exact eigenfrequency.

With the aim of assessing the performance of our estimator,  we consider domains   with singularities in two and three dimensions in order to observe the improvement of the convergence rate.  On each adaptive iteration, we use the blue-green marking strategy to refine each $T'\in \CT_{h}$ whose indicator $\eta_{T'}$ satisfies
$$
\eta_{T'}\geq 0.5\max\{\eta_{T}\,:\,T\in\CT_{h} \}.
$$
The effectivity indexes with respect to $\zeta$ and the eigenvalue $\widehat{\kappa_i}$ are defined by
$$
\eff(\widehat{\kappa_i}):=\frac{\err(\widehat{\kappa_i})}{\eta^2}.
$$


\subsection{2D square with constant Lamé parameters}\label{subsec:numerical-experiment-square2D}
In this experiment we test our method on the unit square domain $\Omega:=(0,1)^2$. We borrow the experiment from \cite[Section 6.1]{meddahi2013finite} consisting of a square domain, clamped at the bottom and free on the rest of the boundary. The physical parameters for this case are given by
$$
E=1.44\times 10^{11} \text{Pa}, \quad \rho = 7.7\times 10^3 \text{Kg/m}^3.
$$

The following values of $\nu$, together with their respective Sobolev exponent are potrayed in Table \ref{table:nus-mus-lambdas} (see for instance \cite{grisvard1986problemes}).
\begin{table}[hbt!]
	{\footnotesize
		\centering
		\begin{tabular}{|c|c|c|c|}
			\hline
			\hline
		$\nu$             &  $\mu$         &   $\lambda$ & $r$ \\
		\hline
		0.35  & 7.2000e+10 & 1.6800e+11  & 0.6797\\
		0.49 & 7.2000e+10 & 3.5280e+12 & 0.5999 \\
		0.49999 & 7.2000e+10  & 3.5999e+15 & 0.5947 \\
		0.5 & 7.2000e+10 & $\infty$ & 0.5649 \\
		\hline
		\hline
		\end{tabular}
	\caption{Test \ref{subsec:numerical-experiment-square2D}. Different values of $\nu$, together with the corresponding computed values of $\mu$, $\lambda$ and the Sobolev exponent.}
	\label{table:nus-mus-lambdas}}
\end{table}
Tables \ref{tabla:square-MINI-SALIM} and \ref{tabla:square-TH-SALIM} present the first six lowest order  computed eigenfrequencies compared with those of \cite{meddahi2013finite}. Here, we observe that, in both tables, the convergence rate for the third and fifth mode behaves like $\mathcal{O}(h^r)$, with $r\geq 2$, whereas the rest of the modes converge to the extrapolated values with $2r\geq 1.36$ for $\nu=0.35$, and $2r\geq 1.2$ for the Stokes limit case $\nu\approx0.5$.

We now pass to adaptive refinements. Note that the results provided in Tables \ref{tabla:square-MINI-SALIM} - \ref{tabla:square-TH-SALIM} shows that the first eigenvalue rate of convergence is suboptimal since we have point singularites, related to the change of boundary conditions, in $(x,y)=(0,0)$ and $(x,y)=(1,0)$.  We study the adaptive refinements for the first eigenvalue. Figure \ref{fig:cuadradoAFEMerror} depicts the different convergence behavior for each value of $\nu$. It is clear that the optimal rate of convergence is recovered, behaving like $\mathcal{O}(\texttt{dof}^{-1})\simeq \mathcal{O}(h^2)$. To study the behavior of the estimator, we have computed the values of $\eta^2$ and $\eff(\lambda_1)$ on Figure \ref{fig:cuadradoAFEMeffectivity}. Here, we observe that the estimator behaves like $\mathcal{O}(h^2)$, yielding to a  roughly constant efficiency behavior, which oscillates between $0.16$ and $0.19$ for the selected values of $\nu$.  We also present the adaptive meshes on Figure \ref{fig:squares-2D-adaptive}, where we observe that the refinements are concentrated near the two singular corners. Also, the amount of elements marked for $\nu=0.35$ is significantly higher than those marked in the limit case. This is associated to the additional contribution on the volumetric and jump terms associated to $\lambda$.

\begin{table}[hbt!]
	{\footnotesize
		\begin{center}
			\caption{Test \ref{subsec:numerical-experiment-square2D}. Lowest computed eigenvalues for the mini-element family and different values of $\nu$. }
				\label{tabla:square-MINI-SALIM}
			\begin{tabular}{|c c c c |c| c |c|}
				\hline
				\hline
				$N=20$             &  $N=30$         &   $N=40$         & $N=50$ & Order & $\sqrt{\widehat{\kappa}_{extr}}$ & \cite{meddahi2013finite} \\ 
				\hline
				\multicolumn{6}{c}{$\nu=0.35$} & \\
				\hline
				2967.3574  &  2957.0361  &  2952.7091  &  2950.4117  & 1.53 &  2944.7395  &2944.295 \\
				7372.3677  &  7362.3090  &  7357.8981  &  7355.4940  & 1.41 &  7348.9456  &7348.840 \\
				7908.8464  &  7893.3931  &  7887.7663  &  7885.1085  & 1.94 &  7880.1286  &7880.084 \\
				12831.2888  & 12784.9811  & 12768.5793  & 12760.9513  & 2.01 & 12747.2677  & 12746.802\\
				13137.3479  & 13095.2059  & 13078.7328  & 13070.4722  & 1.73 & 13052.7434  &13051.758 \\
				14968.4108  & 14926.3074  & 14911.1135  & 14903.9436  & 1.96 & 14890.7177  & 14890.114\\
				\hline
				\multicolumn{6}{c}{$\nu=0.49$} & \\
				\hline
				3081.8429  &  3059.1117  &  3048.8806  &  3043.1623  & 1.32 &  3026.3099  & 3025.120\\
				8032.0292  &  7997.6323  &  7981.9373  &  7973.0925  & 1.29 &  7946.4129  & 7945.193\\
				8077.3101  &  8060.6008  &  8054.6502  &  8051.8958  & 2.01 &  8046.9569  &8046.967 \\
				12789.6444  & 12733.3815  & 12709.3393  & 12696.4308  & 1.48 & 12663.2748  &12660.250 \\
				13234.2998  & 13193.7694  & 13179.4864  & 13172.8689  & 2.03 & 13161.1898  &13161.057 \\
				15754.1992  & 15663.4991  & 15628.0896  & 15610.2044  & 1.73 & 15572.0923  &15567.043 \\
				\hline
				\multicolumn{6}{c}{$\nu=0.5$} & \\
				\hline
				3096.0660  &  3071.3635  &  3060.2058  &  3053.9532  & 1.31 &  3035.3659  &3034.018 \\
				8090.3425  &  8052.5697  &  8035.2770  &  8025.5048  & 1.28 &  7995.7653  &7994.348 \\
				8098.9778  &  8081.7157  &  8075.5863  &  8072.7571  & 2.01 &  8067.6493  &8067.720 \\
				12775.5424  & 12716.6767  & 12691.3138  & 12677.6156  & 1.45 & 12641.4334  & 12638.546\\
				13268.7664  & 13228.2525  & 13213.9786  & 13207.3680  & 2.03 & 13195.6921  &13195.563 \\
				15792.0288  & 15697.3630  & 15660.1027  & 15641.1543  & 1.70 & 15599.7755  &15594.866 \\
				\hline
				\hline
			\end{tabular}
	\end{center}}

\end{table}

\begin{table}[hbt!]
	{\footnotesize
		\begin{center}
			\caption{Test \ref{subsec:numerical-experiment-square2D}. Lowest computed eigenvalues for the lowest order Taylor-Hood family and different values of $\nu$. }
			\label{tabla:square-TH-SALIM}
			\begin{tabular}{|c c c c |c| c |c|}
				\hline
				\hline
				$N=20$             &  $N=30$         &   $N=40$         & $N=50$ & Order & $\sqrt{\widehat{\kappa}_{extr}}$ & \cite{meddahi2013finite} \\ 
				\hline
				\multicolumn{6}{c}{$\nu=0.35$} & \\
				\hline
				 2947.9433  &  2946.3921  &  2945.7107  &  2945.3380  & 1.41 &  2944.3200  &2944.295 \\
				7353.0241  &  7351.2280  &  7350.4365  &  7350.0038  & 1.40 &  7348.8078  & 7348.840 \\
				7880.6474  &  7880.4027  &  7880.3240  &  7880.2879  & 2.26 &  7880.2343  & 7880.084\\
				12747.5800  & 12747.4026  & 12747.3597  & 12747.3431  & 3.04 & 12747.3271  & 12746.802 \\
				13059.0437  & 13055.8615  & 13054.5253  & 13053.8104  & 1.53 & 13052.0477  & 13051.758 \\
				14891.6919  & 14891.1322  & 14890.9288  & 14890.8270  & 1.91 & 14890.6395  & 14890.114 \\
				\hline
				\multicolumn{6}{c}{$\nu=0.49$} & \\
				\hline
				3036.8461  &  3032.3233  &  3030.2238  &  3029.0275  & 1.25 &  3025.2710  & 3025.120 \\
				7963.1337  &  7956.1517  &  7952.8860  &  7951.0235  & 1.23 &  7945.0430  & 7945.193\\
				8047.2590  &  8047.1328  &  8047.1054  &  8047.0967  & 3.41 &  8047.0887  &8046.967 \\
				12680.2906  & 12672.3064  & 12668.6743  & 12666.6299  & 1.30 & 12660.4821  &12660.250 \\
				13161.3264  & 13161.1886  & 13161.1630  & 13161.1555  & 3.83 & 13161.1499  & 13161.057\\
				15583.8476  & 15577.3533  & 15574.4389  & 15572.8137  & 1.34 & 15568.0999  & 15567.043\\
				\hline
				\multicolumn{6}{c}{$\nu=0.5$} & \\
				\hline
				3046.7873  &  3041.8988  &  3039.6212  &  3038.3196  & 1.24 &  3034.1970 &3034.018  \\
				8014.1126  &  8006.4753  &  8002.8886  &  8000.8375  & 1.22 &  7994.1937  & 7994.348\\
				8068.0073  &  8067.8864  &  8067.8616  &  8067.8543  & 3.59 &  8067.8479  &8067.720 \\
				12660.2310  & 12651.6456  & 12647.7223  & 12645.5075  & 1.29 & 12638.7963  &12638.546 \\
				13195.8533  & 13195.7132  & 13195.6869  & 13195.6789  & 3.78 & 13195.6730  & 13195.563 \\
				15613.2772  & 15606.0211  & 15602.7442  & 15600.9105  & 1.33 & 15595.5636  & 15594.866\\
				\hline
				\hline
			\end{tabular}
	\end{center}
	}
\end{table}

\begin{figure}[hbt!]
	\centering
	\includegraphics[scale=0.45]{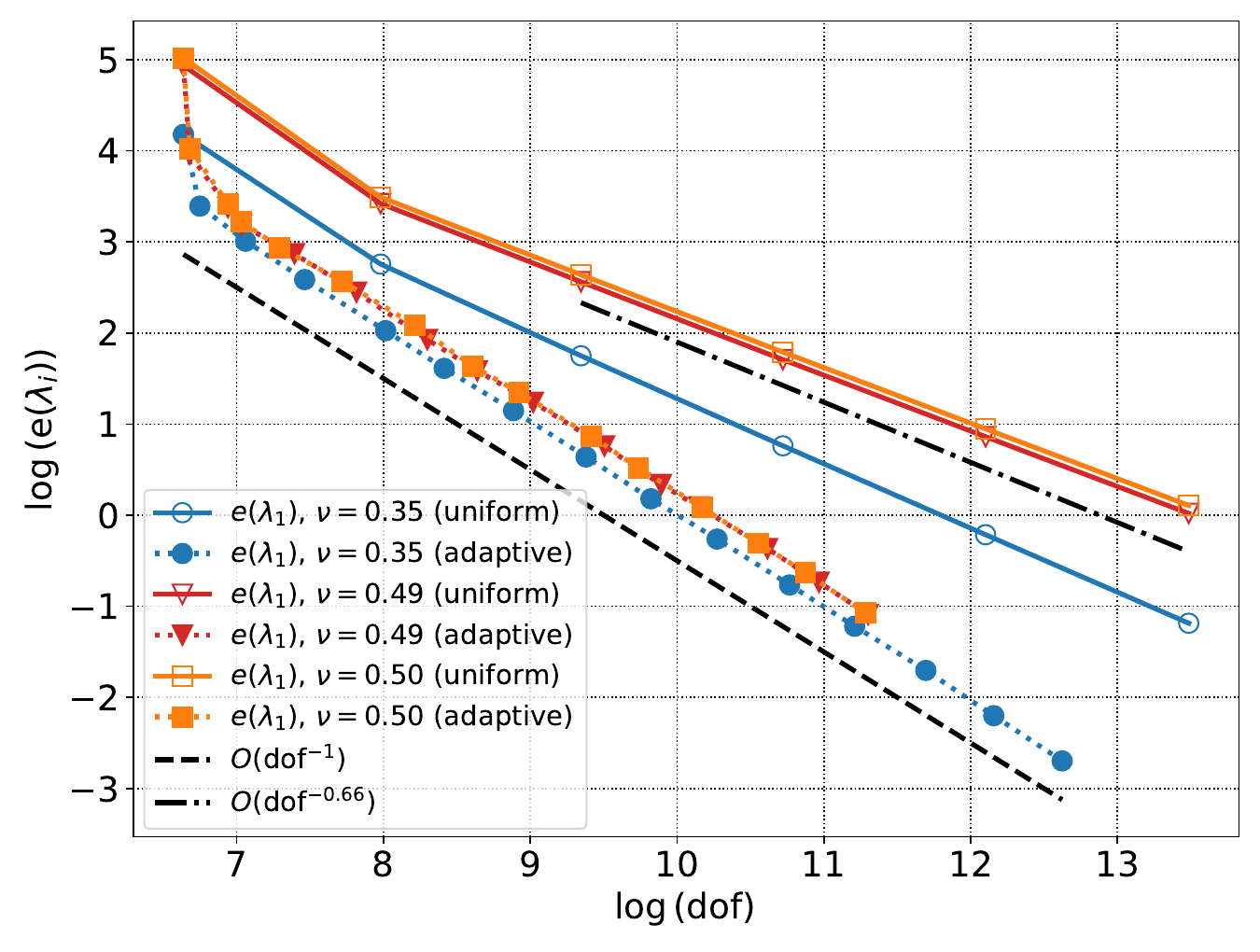}
	\caption{Test \ref{subsec:numerical-experiment-square2D}.  Error curves obtained from the adaptive algorithm in the bottom-clamped square domain compared with the lines $\mathcal{O}(\texttt{dof}^{-0.66})$ and $\mathcal{O}(\texttt{dof}^{-1})$.}
	\label{fig:cuadradoAFEMerror}
\end{figure}

\begin{figure}[hbt!]
	\centering
	\begin{minipage}{0.49\linewidth}\centering
	\includegraphics[scale=0.45]{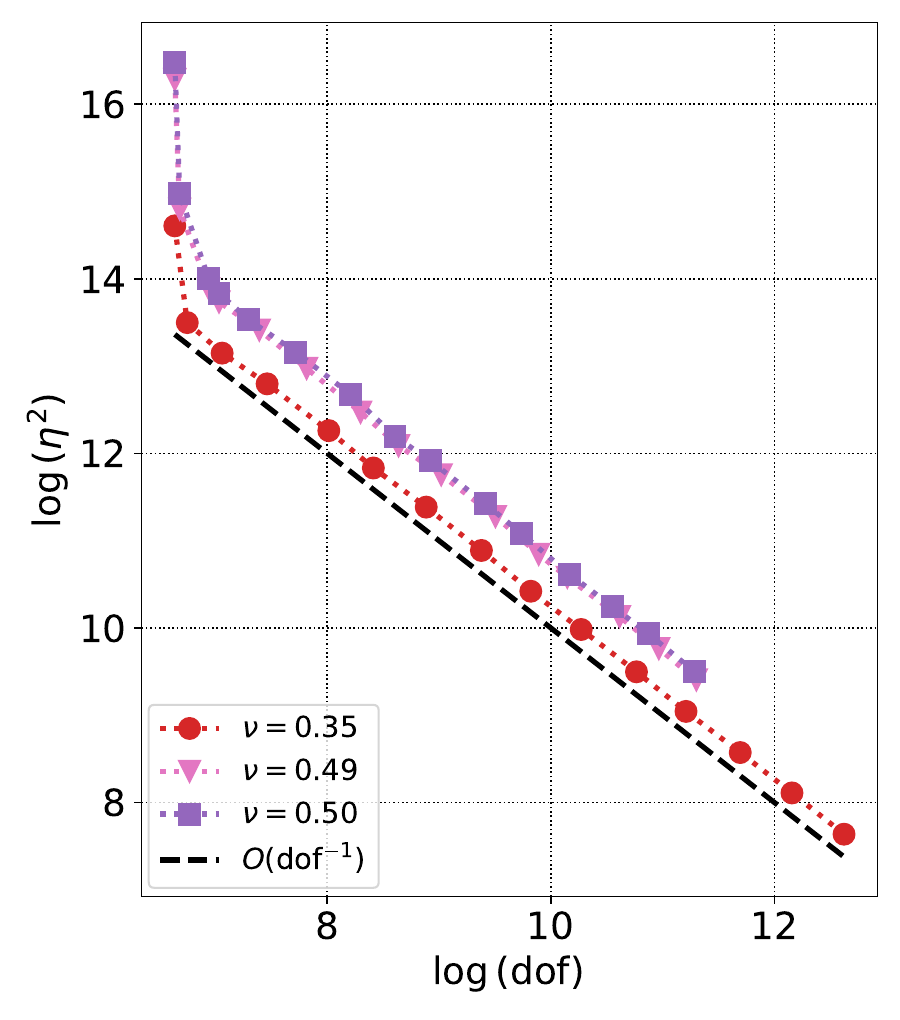}
\end{minipage}
\begin{minipage}{0.49\linewidth}\centering
	\includegraphics[scale=0.45]{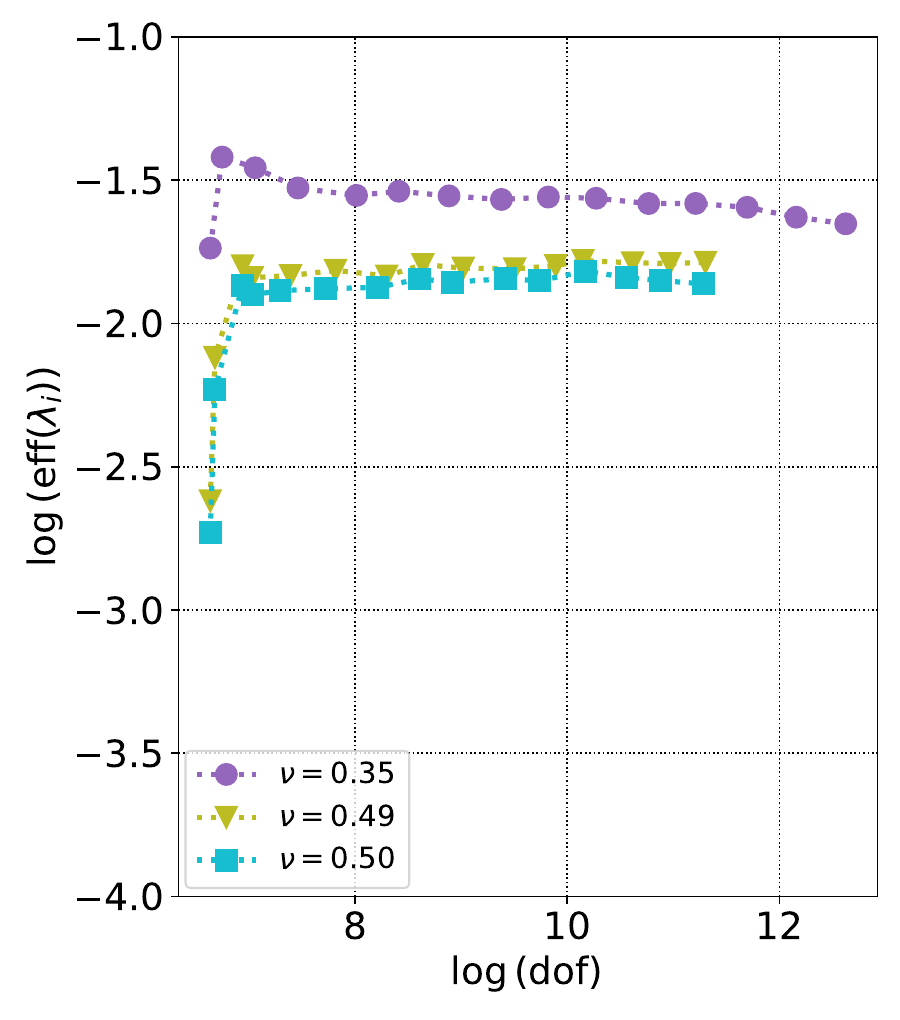}
\end{minipage}
	\caption{Test \ref{subsec:numerical-experiment-square2D}.  Estimator and efficiency curves obtained from the adaptive algorithm in the bottom-clamped square domain for different values of $\nu$.}
	\label{fig:cuadradoAFEMeffectivity}
\end{figure}

\begin{figure}[hbt!]
	\centering
	\begin{minipage}{0.32\linewidth}
		\centering
		\includegraphics[scale=0.07,trim=41cm 2cm 41cm 2cm,clip]{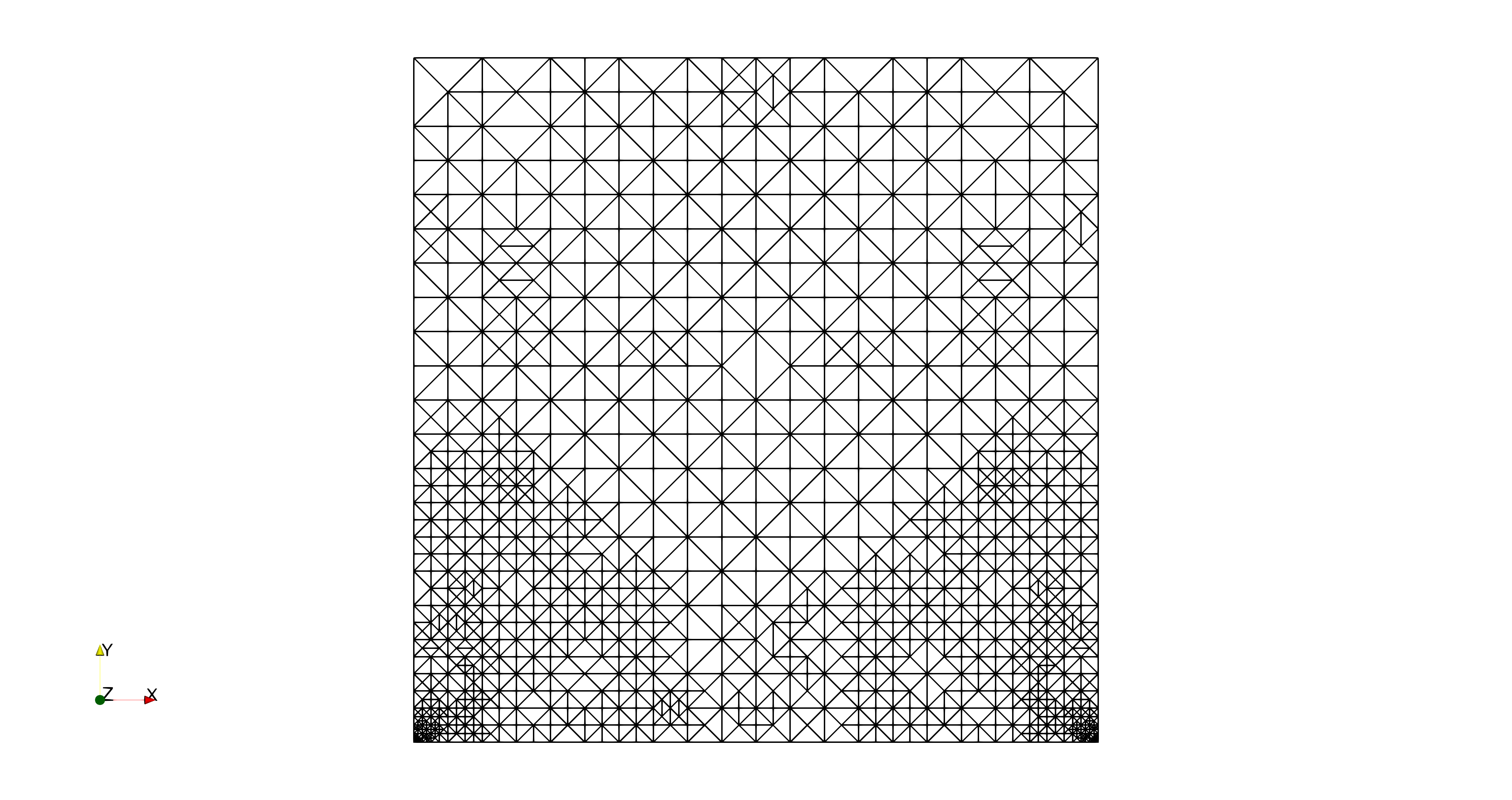}\\
	\end{minipage}
	\begin{minipage}{0.32\linewidth}
		\centering
		\includegraphics[scale=0.07,trim=41cm 2cm 41cm 2cm,clip]{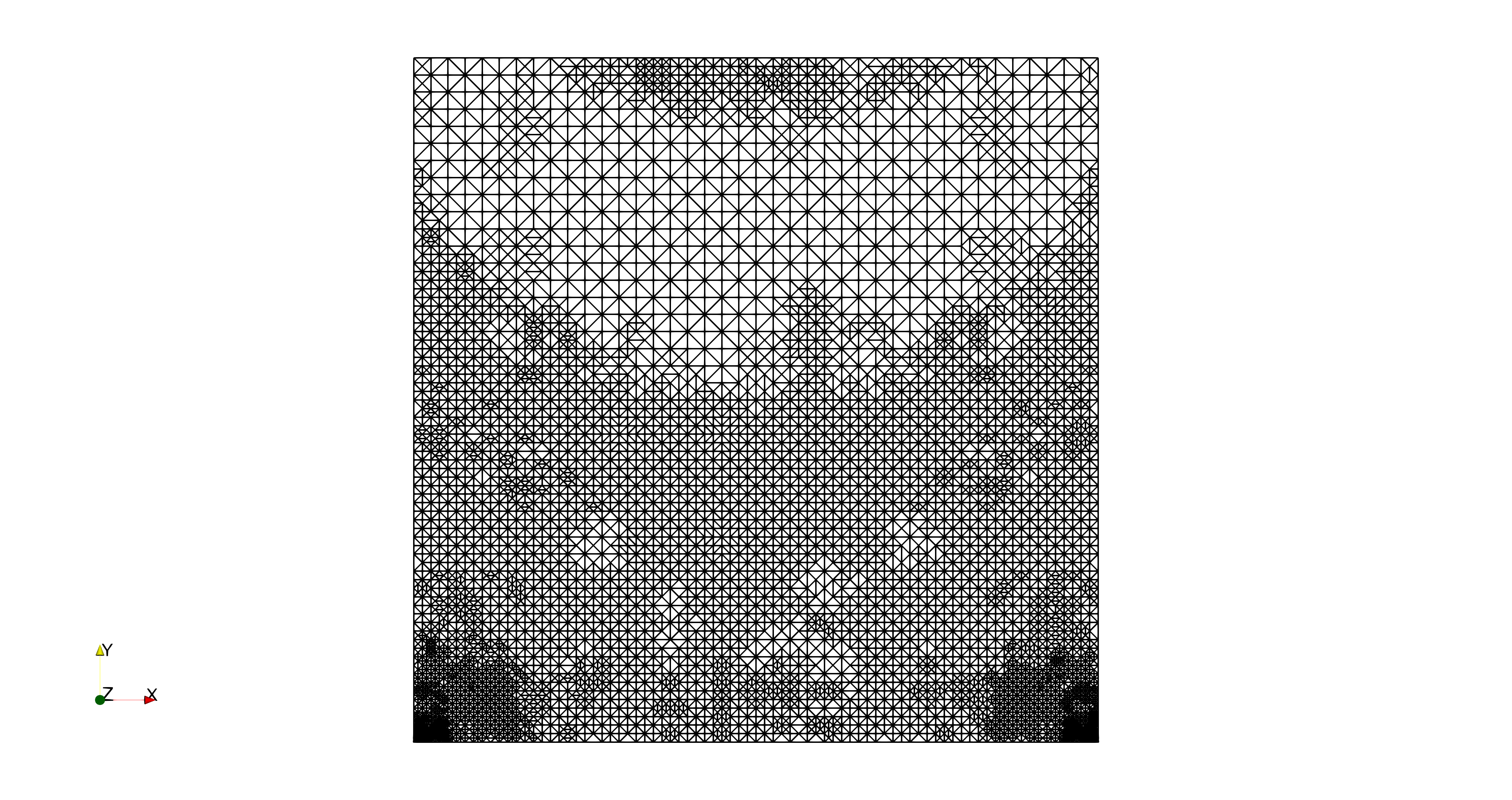}\\
	\end{minipage}
	\begin{minipage}{0.32\linewidth}
		\centering
		\includegraphics[scale=0.07,trim=41cm 2cm 41cm 2cm,clip]{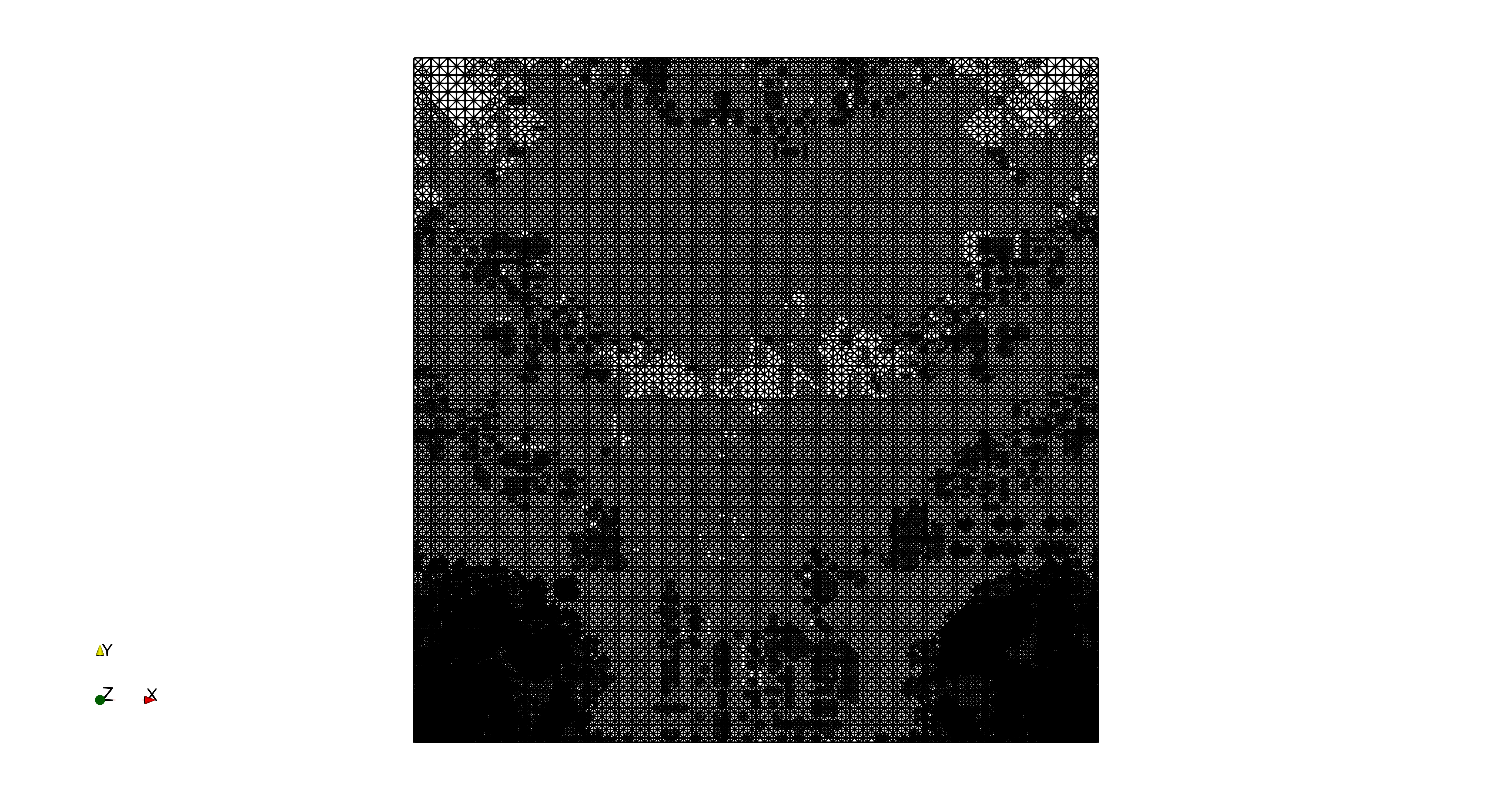}\\
	\end{minipage}\\
	\begin{minipage}{0.32\linewidth}
		\centering
		\includegraphics[scale=0.07,trim=41cm 2cm 41cm 2cm,clip]{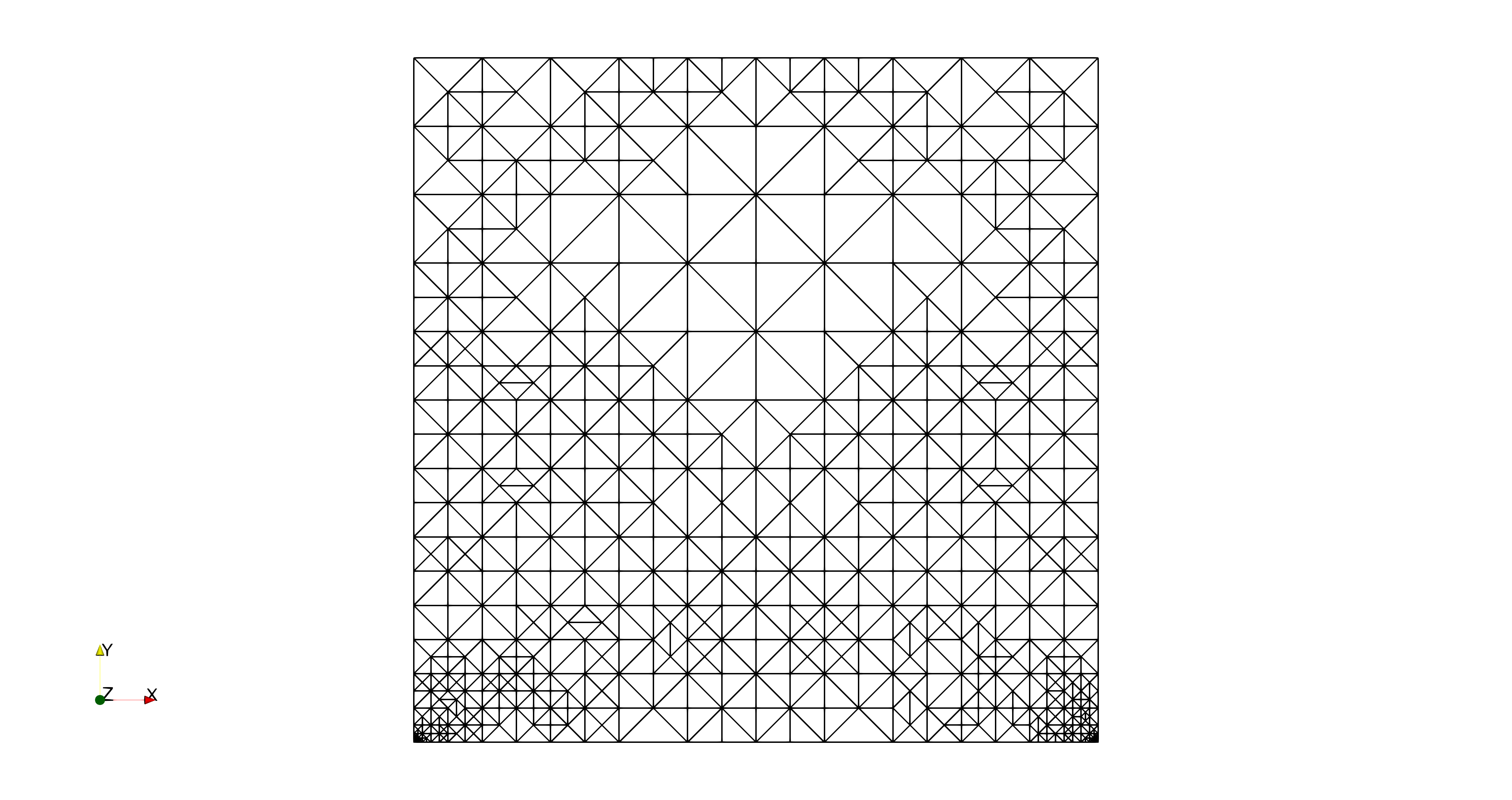}\\
	\end{minipage}
	\begin{minipage}{0.32\linewidth}
		\centering
		\includegraphics[scale=0.07,trim=41cm 2cm 41cm 2cm,clip]{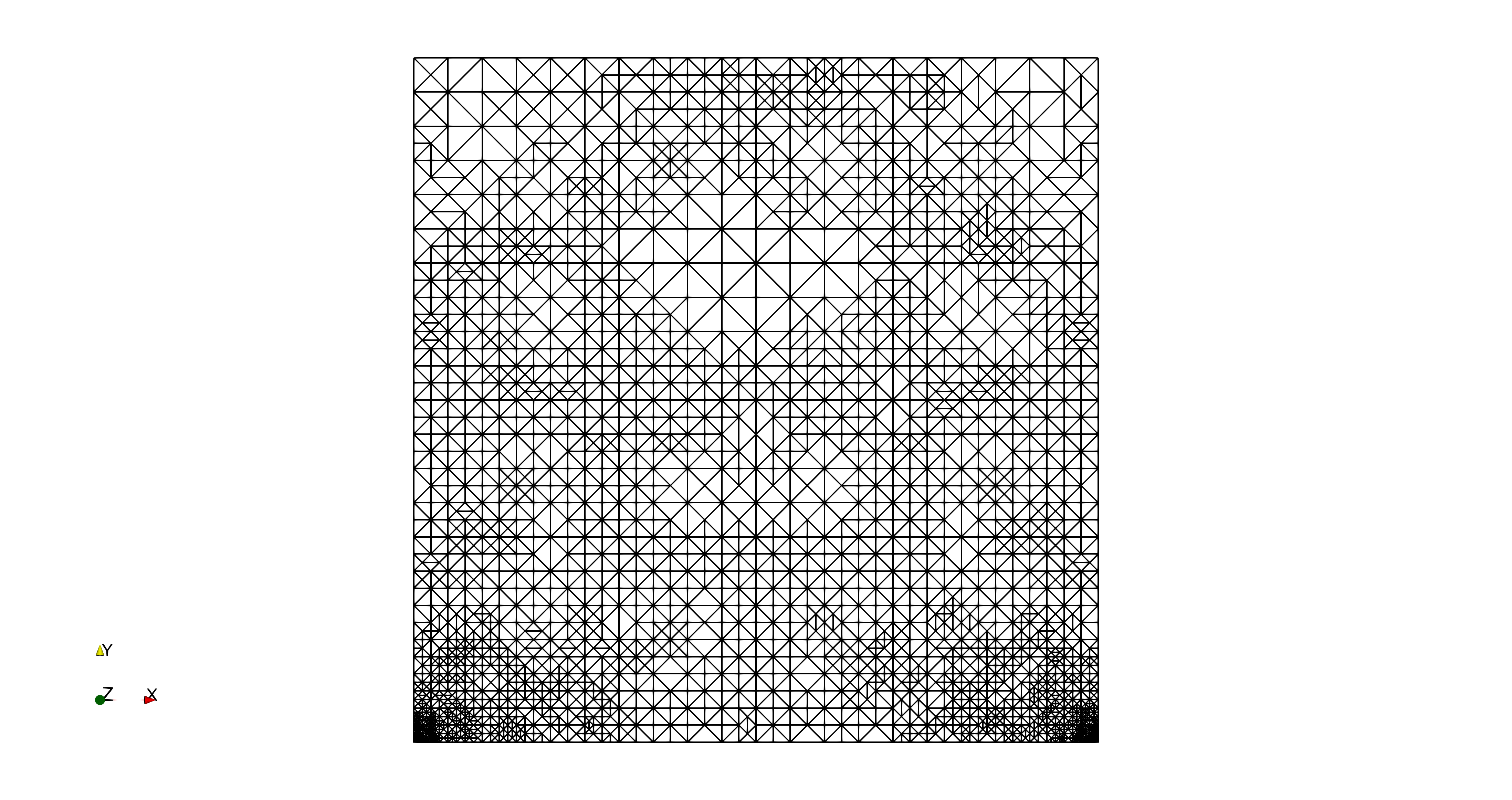}\\
	\end{minipage}
	\begin{minipage}{0.32\linewidth}
		\centering
		\includegraphics[scale=0.07,trim=41cm 2cm 41cm 2cm,clip]{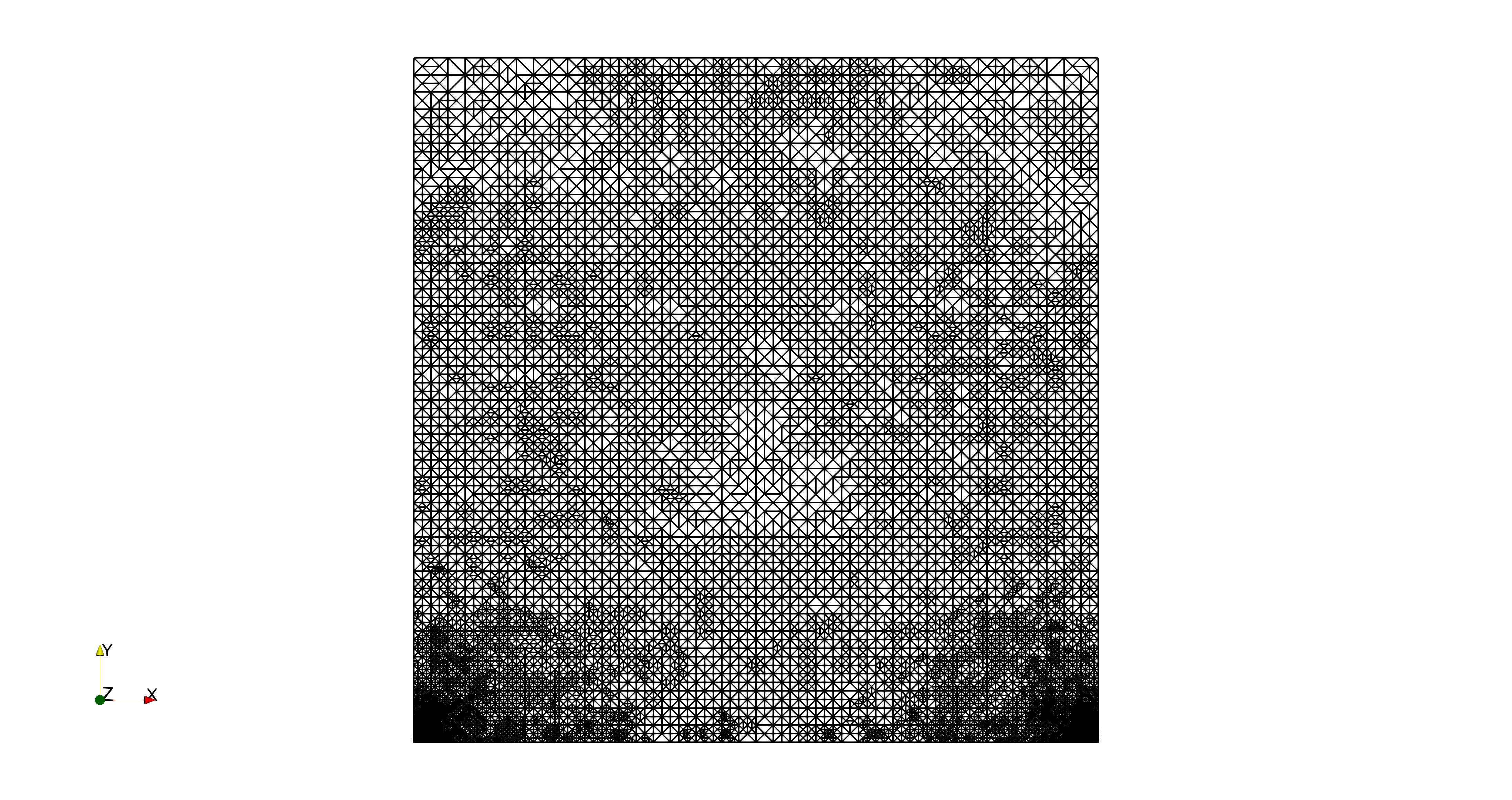}\\
	\end{minipage}
	\caption{Test \ref{subsec:numerical-experiment-square2D}. Intermediate meshes of the unit square domain obtained with the adaptive algorithm and different values of $\nu$. Top row: Meshes with 7229, 47399 and 303396 dofs with $\eta=0.35$. Bottom row: Meshes with 3686, 16917  and 79705 dofs with $\eta=0.5$. }
	\label{fig:squares-2D-adaptive}
\end{figure}

\subsection{2D square with discontinuous Lamé parameters}\label{subsec:numerical-experiment-square2D-variablenu}
In this experiment we again consider the unit square $\Omega:=(0,1)^2$, but this time we take a Young's modulus such that the Lamé parameters are discontinuous. More precisely, we subdivide $\Omega$ into three regions $\Omega:=\Omega_1\cup\Omega_2\cup\Omega_3$, given by,
$$\Omega_1:=(0,1/3)\times (0,1),\quad \Omega_2:=(1/3,2/3)\times (0,1),\quad \Omega_3:=(2/3,1)\times (0,1),$$
such that the Young's modulus is defined by
\begin{equation*}
E(x)=\left\{
\begin{aligned}
	&2, &\text{ if } x\in\Omega_1,\\
	&1, &\text{ if } x\in\Omega_2,\\
	&3, &\text{ if } x\in\Omega_3.\\
\end{aligned}
\right.
\end{equation*}
The material density is assumed to be 1. This implies that both $\mu$ and $\lambda$ are discontinuous functions. In Figure \ref{fig:cuad-div}  we can observe the division of domain $\Omega$, and the triangulation for this case. Here, the mesh size is given by $h\approx 1/N$, where $N$ is the mesh level.
\begin{figure}[hbt!]
	\centering
	\includegraphics[scale=0.07,trim=35cm 0cm 35cm 0cm,clip]{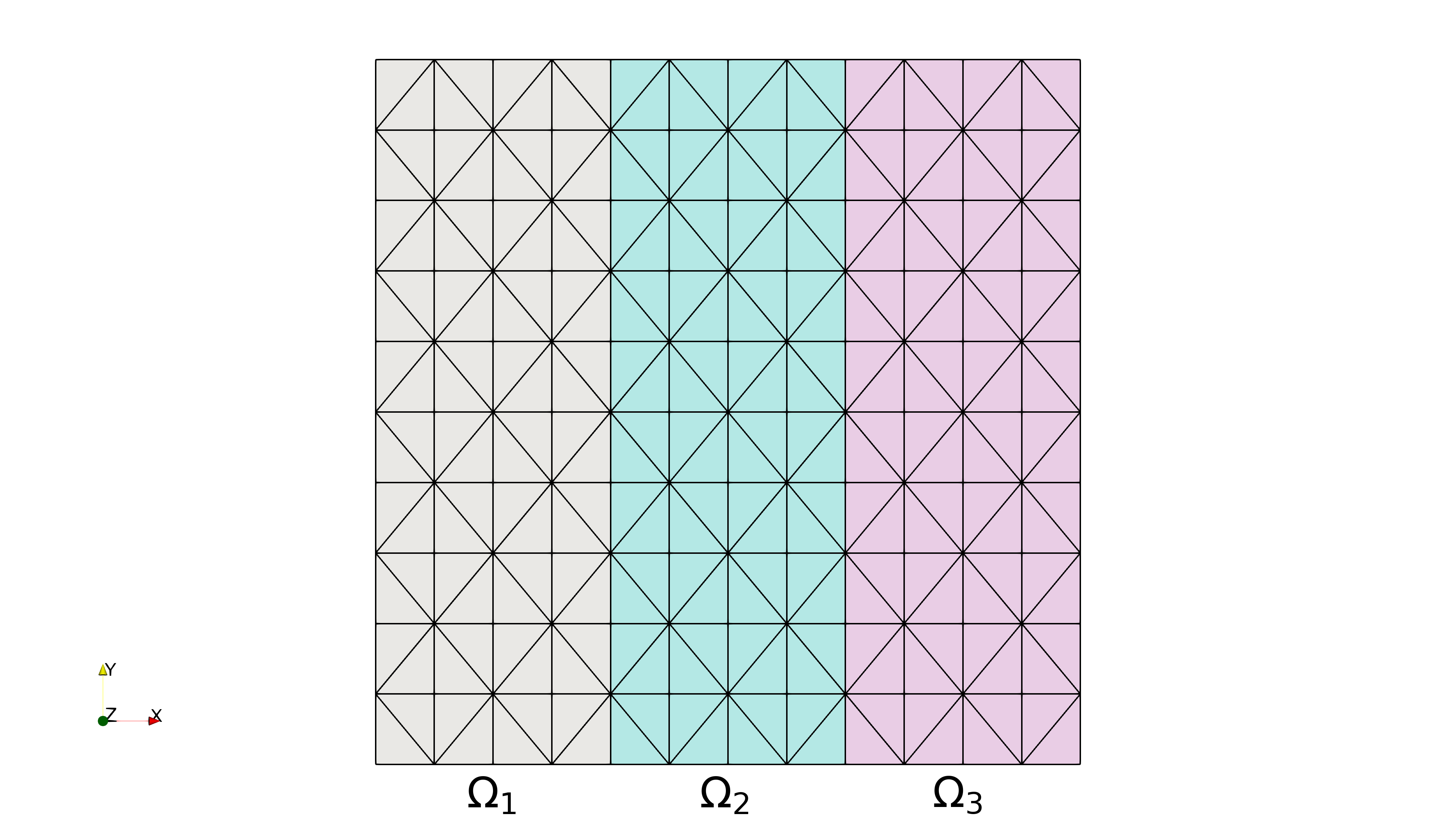}
	\caption{Test \ref{subsec:numerical-experiment-square2D-variablenu}. Sample mesh of the domain $\Omega$ subdivided in $\Omega_i$, $i=1,2,3$ subdomains.}
	\label{fig:cuad-div}
\end{figure}

Table \ref{tabla:square-MINI-constant-discontinuous-E} presents the result for the mini-element family. Here, we observe that a quadratic rate of convergence is attained for all the considered values of $\nu=0.35$.  However, for the other cases, the first eigenmode is affected by the pressure discontinuity and behaves like $\mathcal{O}(h^{1.7})$. This contrast with the results when using the Taylor-hood family, where a quadratic rate is observe for the first eigenmode in all cases. However, the fifth eigenmode also behaves like $\mathcal{O}(h^{1.7})$. Since we are dealing with discontinuous Lamé parameters, Taylor-Hood elements are not capable of giving more than $\mathcal{O}(h^2)$. We also present the eigenmodes for the velocity fields and pressure modes in Figures \ref{fig:velocities-nu035} and \ref{fig:velocities-nu05} for $\nu=0.35$ and $\nu=0.49999$, respectively. It is observed that the flow distributes according to the values of $E$, showing high velocities around the center of the geometry. In order to observe the pressure jumps, we have projected the pressure into the lowest discontinuous Lagrange family of elements $\mathbb{P}_0(\CT_h)$. One can also compute the respective 2D contour plot.
\begin{table}[hbt!]
	{\footnotesize
		\begin{center}
			\caption{Test \ref{subsec:numerical-experiment-square2D-variablenu}. Lowest computed eigenvalues using the mini-element family for different values of $\nu$ and discontinuous $\mu$ and $\lambda$. }
			\label{tabla:square-MINI-constant-discontinuous-E}
			\begin{tabular}{|c c c c |c| c |}
				\hline
				\hline
				$N=20$             &  $N=30$         &   $N=40$         & $N=50$ & Order & $\sqrt{\widehat{\kappa}_{extr}}$  \\ 
				\hline
				\multicolumn{6}{c}{$\nu=0.35$}  \\
				\hline
				5.2021  &     5.1924  &     5.1883  &     5.1869  & 2.01 &     5.1838  \\
				6.0415  &     6.0167  &     6.0063  &     6.0023  & 2.00 &     5.9943  \\
				6.0956  &     6.0851  &     6.0806  &     6.0788  & 1.91 &     6.0751  \\
				7.8128  &     7.7540  &     7.7310  &     7.7219  & 2.12 &     7.7057  \\
				7.8784  &     7.8445  &     7.8305  &     7.8252  & 2.04 &     7.8148  \\
				\hline
				\multicolumn{6}{c}{$\nu=0.49$}  \\
				\hline
				5.8413  &     5.8145  &     5.8025  &     5.7974  & 1.75 &     5.7857  \\
				7.0706  &     7.0224  &     7.0016  &     6.9937  & 1.93 &     6.9767  \\
				7.3501  &     7.2883  &     7.2630  &     7.2519  & 1.95 &     7.2311  \\
				8.6728  &     8.5649  &     8.5216  &     8.5036  & 2.03 &     8.4706  \\
				9.0578  &     8.9419  &     8.8960  &     8.8759  & 2.02 &     8.8403  \\
				\hline
				\multicolumn{6}{c}{$\nu=0.49999$}  \\
				\hline
				5.8308  &     5.8037  &     5.7912  &     5.7859  & 1.70 &     5.7732  \\
				7.1643  &     7.1124  &     7.0902  &     7.0818  & 1.96 &     7.0641  \\
				7.3614  &     7.2979  &     7.2716  &     7.2599  & 1.91 &     7.2375  \\
				8.7113  &     8.5990  &     8.5540  &     8.5352  & 2.03 &     8.5009  \\
				9.0746  &     8.9554  &     8.9079  &     8.8868  & 2.00 &     8.8494  \\
				\hline
				\hline
			\end{tabular}
	\end{center}}
	
\end{table}
\begin{table}[hbt!]
	{\footnotesize
		\begin{center}
			\caption{Test \ref{subsec:numerical-experiment-square2D-variablenu}. Lowest computed eigenvalues using the lowest order Taylor-Hood family for different values of $\nu$ and discontinuous $\mu$ and $\lambda$. }
			\label{tabla:square-TH-constant-discontinuous-E}
			\begin{tabular}{|c c c c |c| c| }
				\hline
				\hline
				$N=20$             &  $N=30$         &   $N=40$         & $N=50$ & Order & $\sqrt{\widehat{\kappa}_{extr}}$  \\ 
				\hline
				\multicolumn{6}{c}{$\nu=0.35$}  \\
				\hline
				5.1820  &     5.1836  &     5.1842  &     5.1844  & 2.20 &     5.1848  \\
				5.9934  &     5.9945  &     5.9949  &     5.9951  & 2.28 &     5.9953  \\
				6.0687  &     6.0723  &     6.0739  &     6.0745  & 1.84 &     6.0759  \\
				7.6973  &     7.7013  &     7.7030  &     7.7037  & 1.96 &     7.7050  \\
				7.8064  &     7.8111  &     7.8131  &     7.8139  & 1.88 &     7.8157  \\
				\hline
				\multicolumn{6}{c}{$\nu=0.49$}  \\
				\hline
				 5.7859  &     5.7880  &     5.7888  &     5.7890  & 2.28 &     5.7895  \\
				 6.9730  &     6.9762  &     6.9774  &     6.9779  & 2.20 &     6.9787  \\
				 7.2263  &     7.2303  &     7.2320  &     7.2326  & 1.99 &     7.2340  \\
				 8.4596  &     8.4659  &     8.4685  &     8.4696  & 2.01 &     8.4716  \\
				 8.8327  &     8.8374  &     8.8396  &     8.8405  & 1.71 &     8.8427  \\
				\hline
				\multicolumn{6}{c}{$\nu=0.49999$}  \\
				\hline
				5.7740  &     5.7761  &     5.7770  &     5.7772  & 2.28 &     5.7777  \\
				7.0608  &     7.0636  &     7.0647  &     7.0651  & 2.18 &     7.0658  \\
				7.2336  &     7.2377  &     7.2394  &     7.2400  & 2.01 &     7.2413  \\
				8.4901  &     8.4964  &     8.4990  &     8.5000  & 2.01 &     8.5020  \\
				8.8426  &     8.8473  &     8.8496  &     8.8504  & 1.71 &     8.8526  \\
				\hline
				\hline
			\end{tabular}
	\end{center}}
	
\end{table}
\begin{figure}[hbt!]
\centering
\begin{minipage}{0.49\linewidth}
	\centering
	\includegraphics[scale=0.07,trim=35cm 10cm 35cm 10cm,clip]{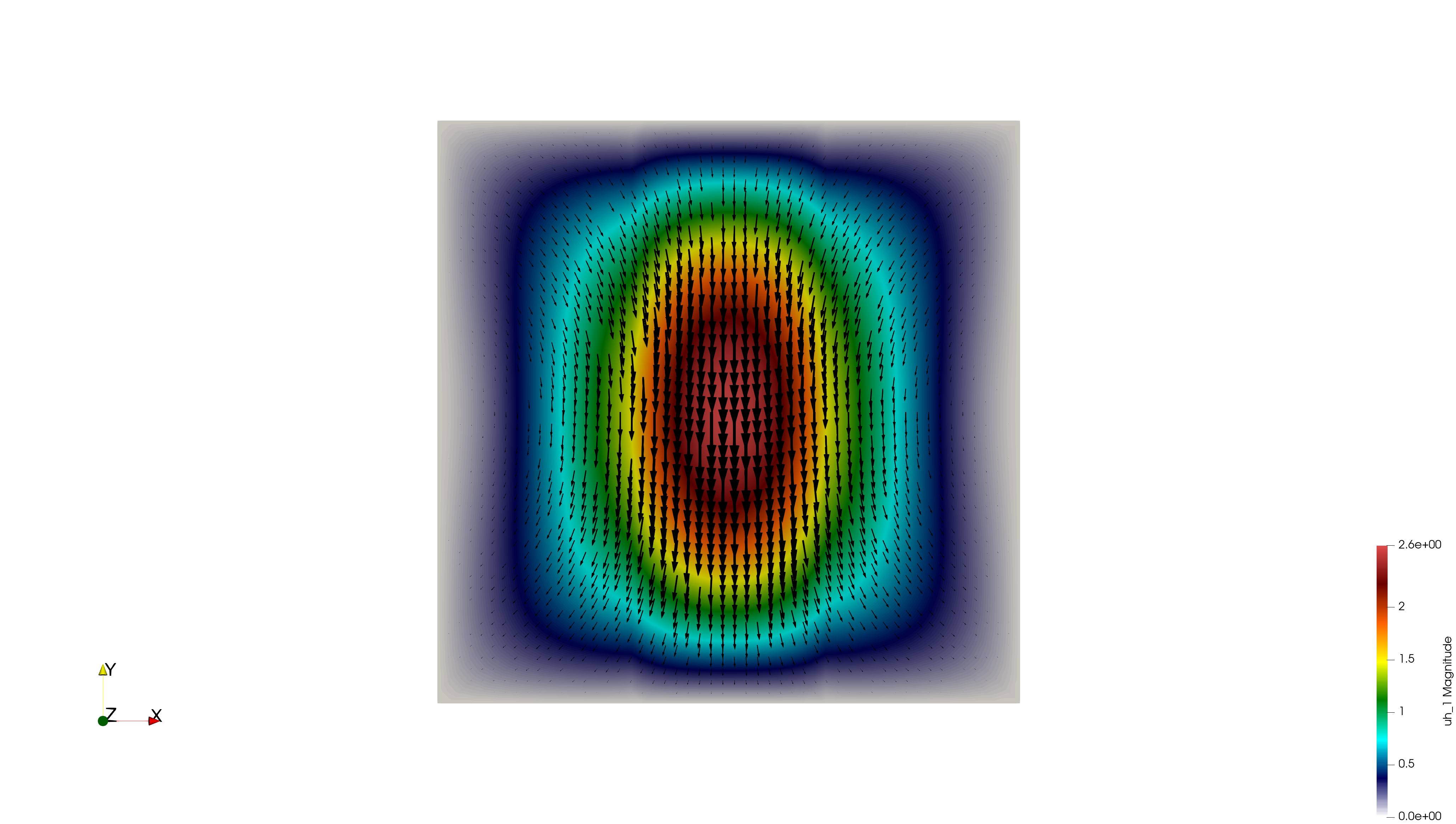}\\
	{\footnotesize $\bu_{1,h}$}
\end{minipage}
\begin{minipage}{0.49\linewidth}
	\centering
	\includegraphics[scale=0.07,trim=35cm 10cm 35cm 10cm,clip]{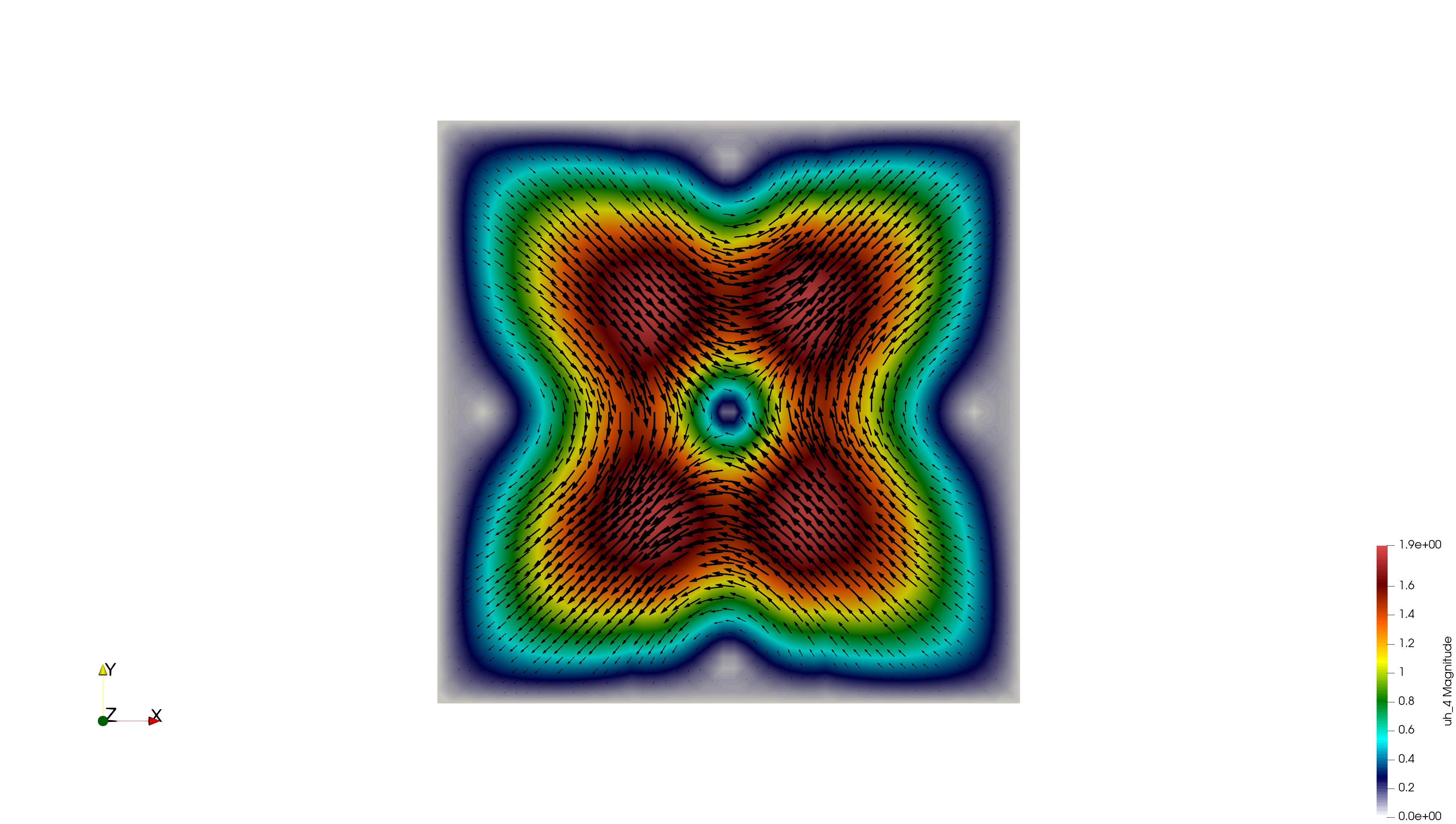}\\
	{\footnotesize $\bu_{4,h}$}
\end{minipage}\\
\begin{minipage}{0.49\linewidth}
	\centering
	\includegraphics[scale=0.07,trim=35cm 10cm 35cm 10cm,clip]{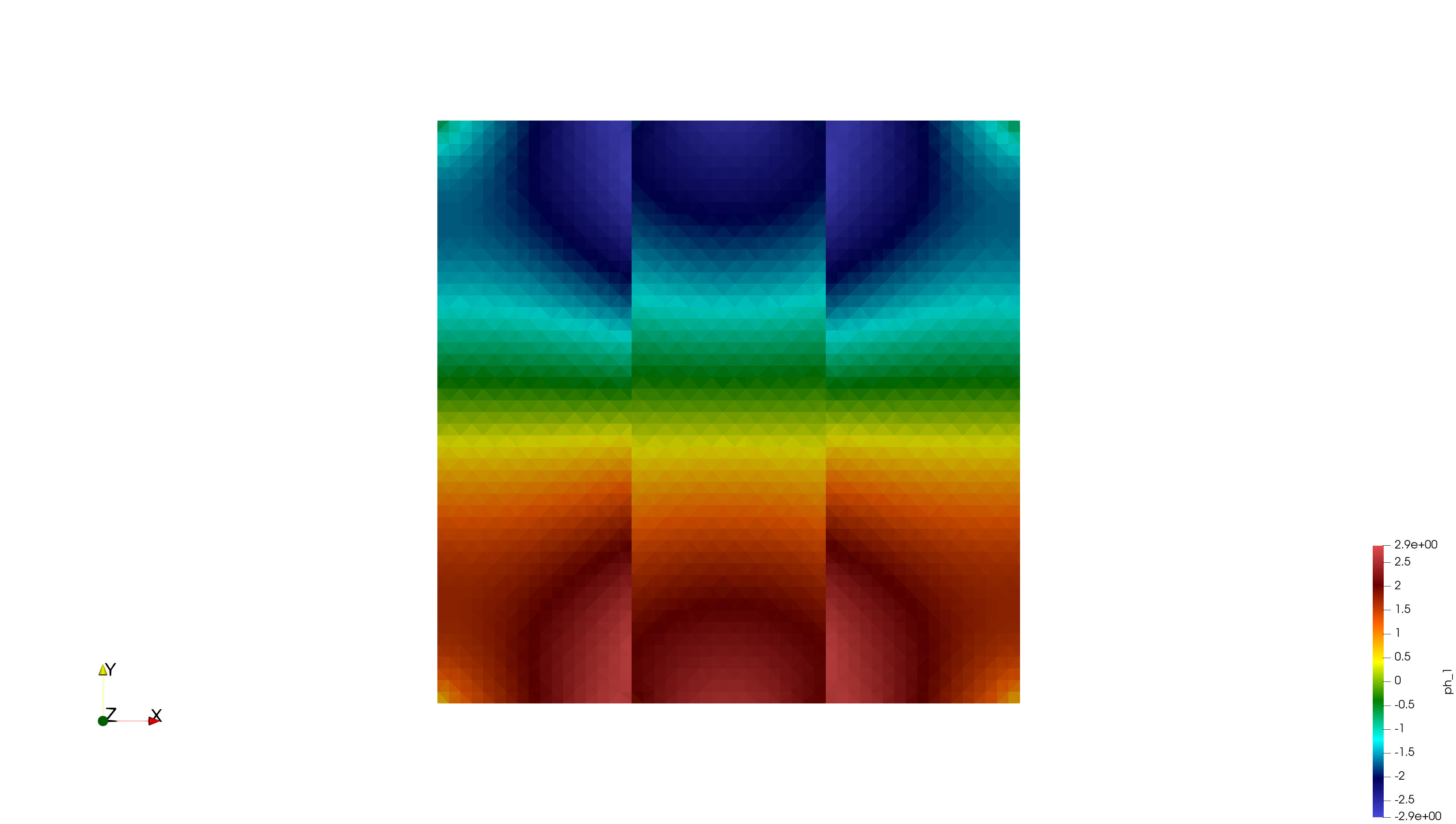}\\
	{\footnotesize $p_{1,h}$}
\end{minipage}
\begin{minipage}{0.49\linewidth}
	\centering
	\includegraphics[scale=0.07,trim=35cm 10cm 35cm 10cm,clip]{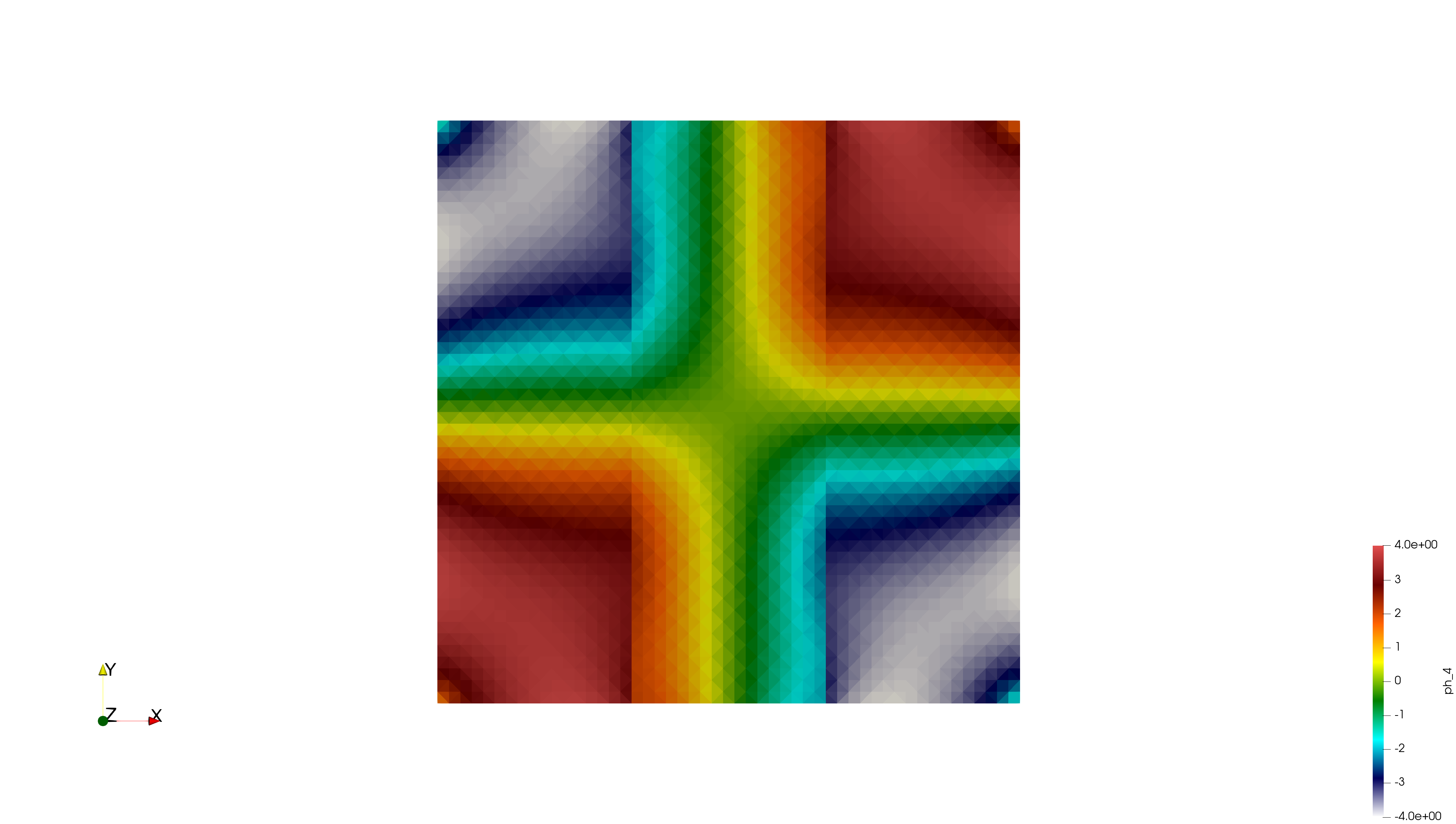}\\
	{\footnotesize $p_{4,h}$}
\end{minipage}
\caption{Test \ref{subsec:numerical-experiment-square2D-variablenu}. Displacement/velocity fields and pressure surface plot for the first and fourth eigenmodes for $\nu=0.35$ and discontinuous $\mu$ and $\lambda$.}
\label{fig:velocities-nu035}
\end{figure}
\begin{figure}[hbt!]
	\centering
	\begin{minipage}{0.49\linewidth}
		\centering
		\includegraphics[scale=0.07,trim=35cm 10cm 35cm 10cm,clip]{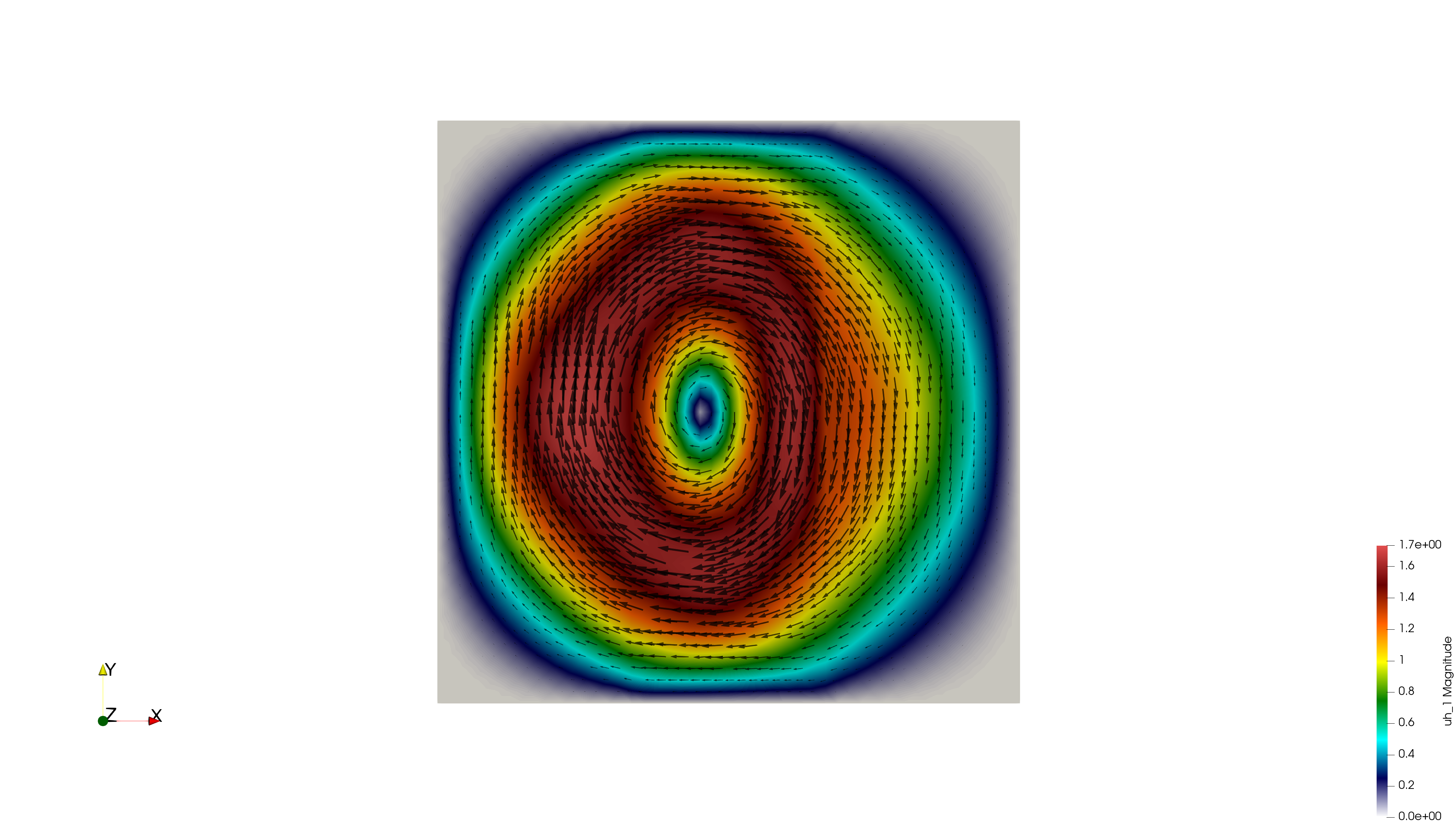}\\
		{\footnotesize $\bu_{1,h}$}
	\end{minipage}
	\begin{minipage}{0.49\linewidth}
		\centering
		\includegraphics[scale=0.07,trim=35cm 10cm 35cm 10cm,clip]{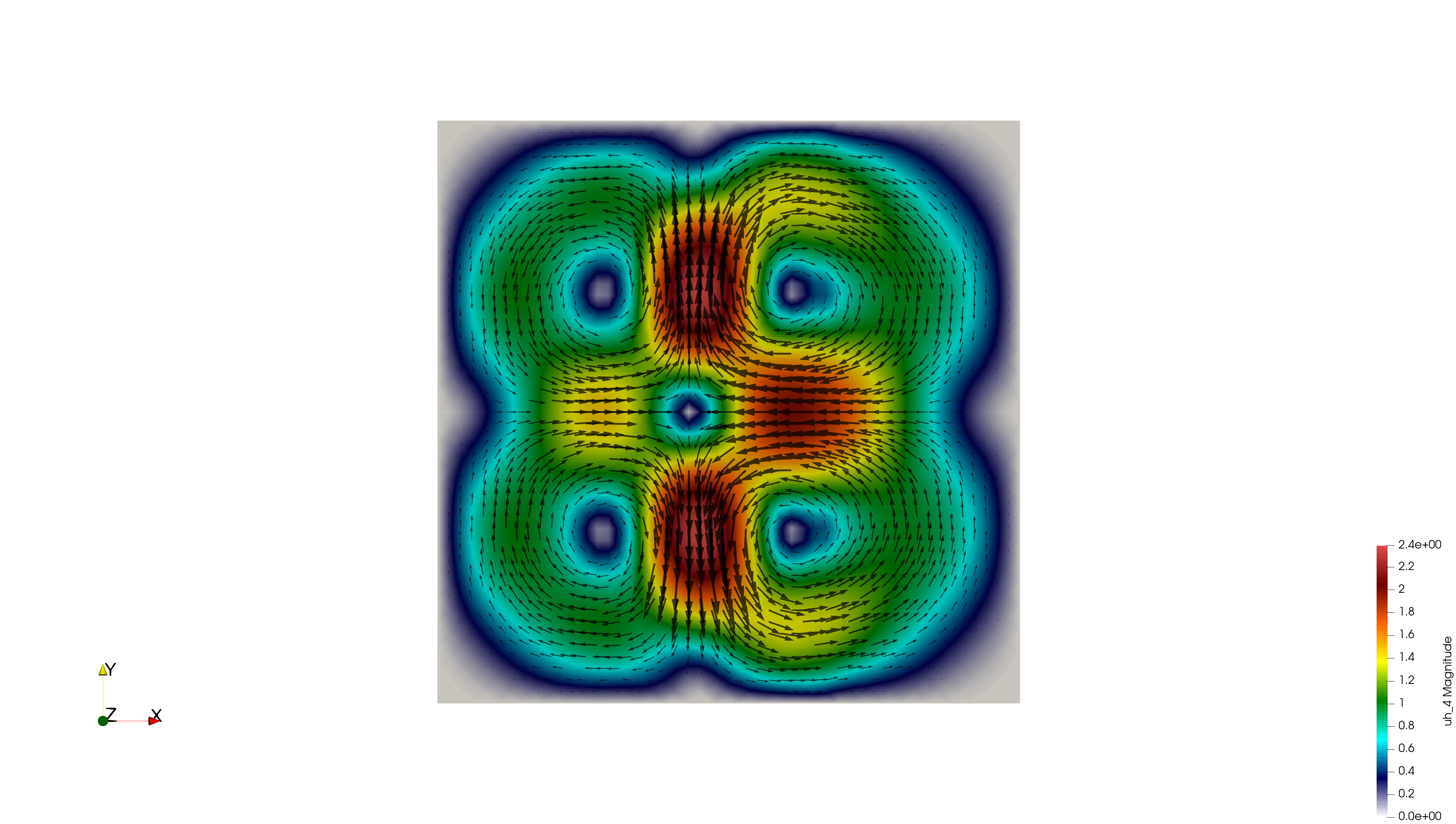}\\
		{\footnotesize $\bu_{4,h}$}
	\end{minipage}\\
	\begin{minipage}{0.49\linewidth}
		\centering
		\includegraphics[scale=0.07,trim=35cm 10cm 35cm 10cm,clip]{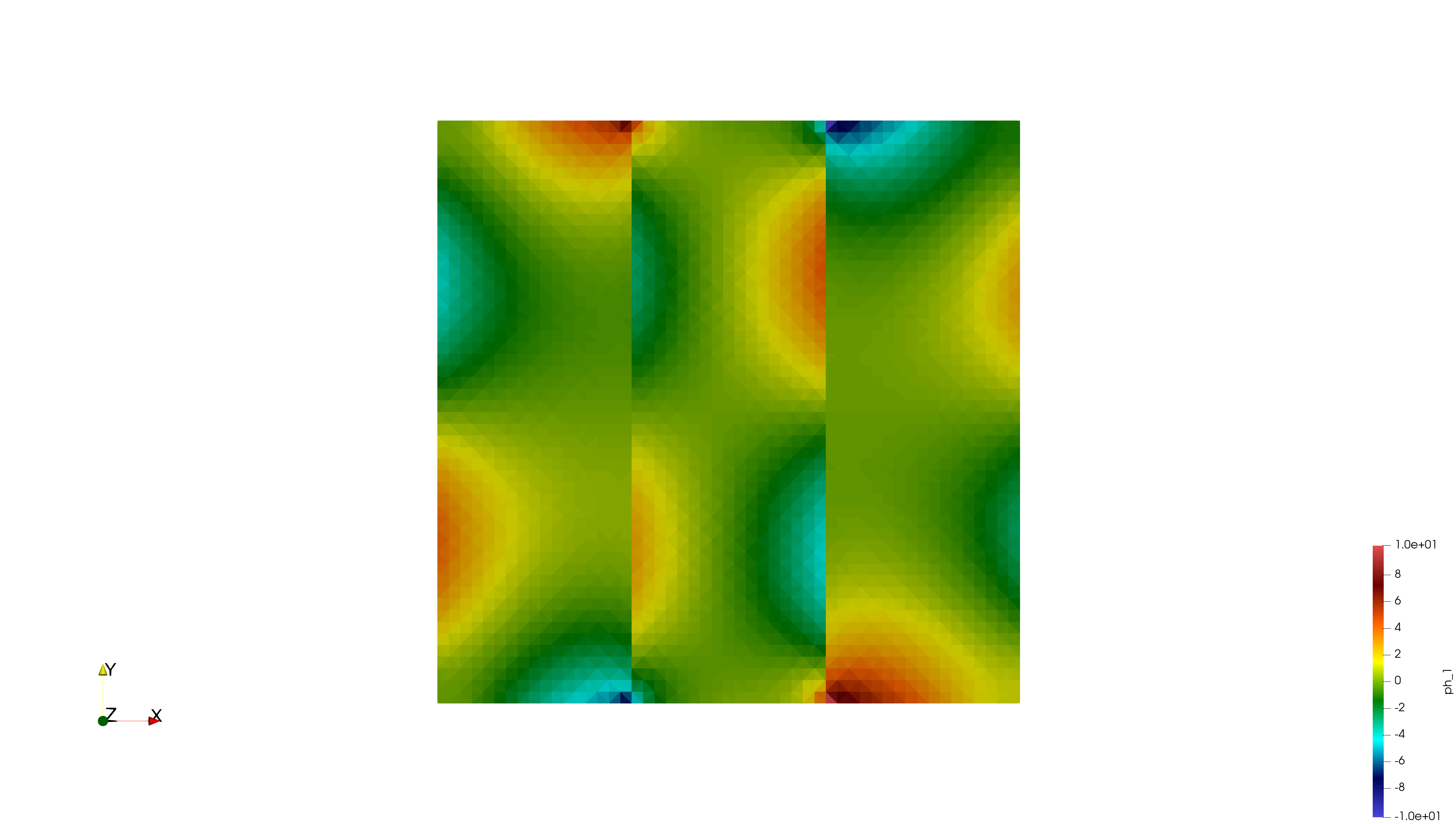}\\
		{\footnotesize $p_{1,h}$}
	\end{minipage}
	\begin{minipage}{0.49\linewidth}
		\centering
		\includegraphics[scale=0.07,trim=35cm 10cm 35cm 10cm,clip]{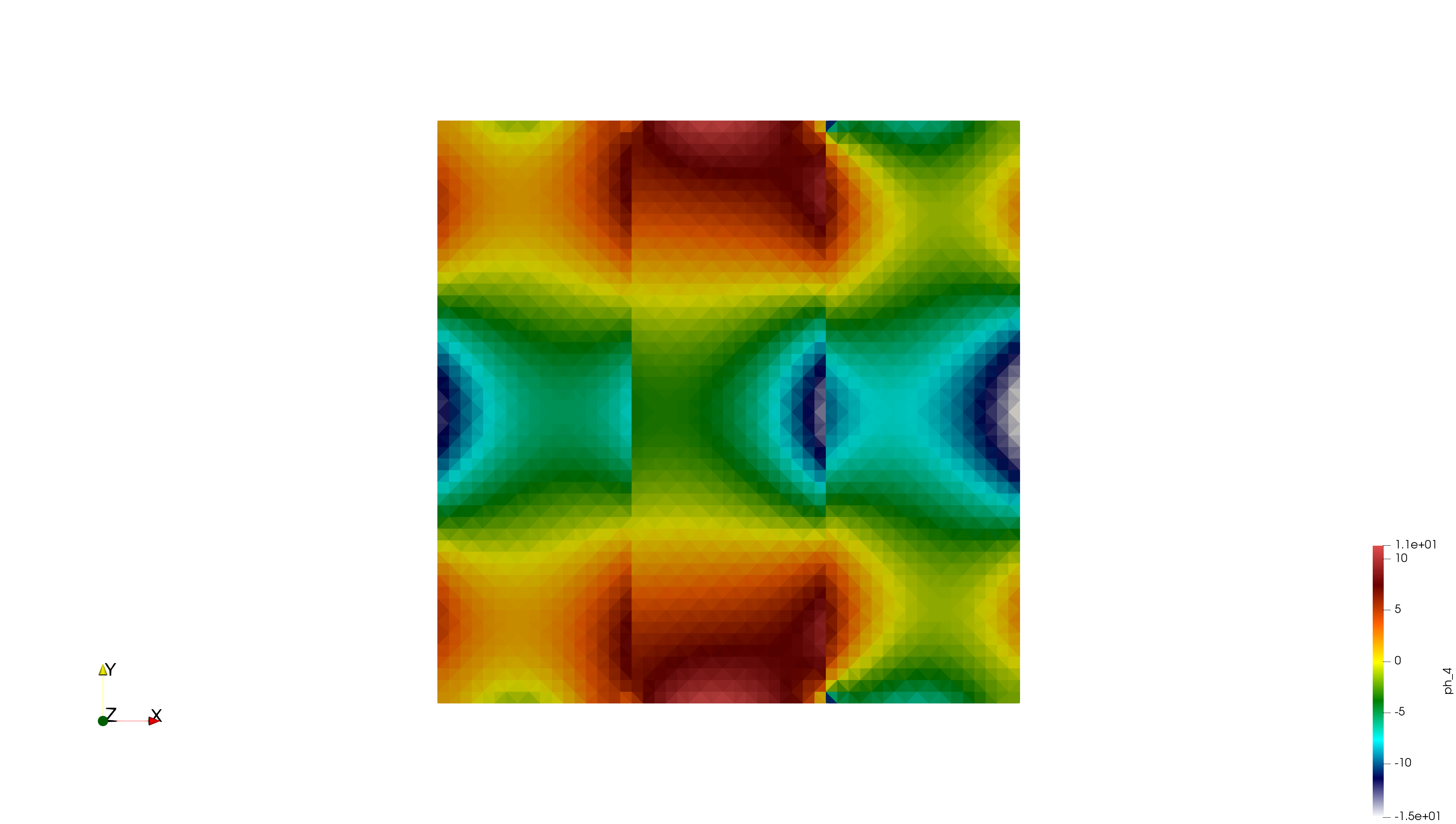}\\
		{\footnotesize $p_{4,h}$}
	\end{minipage}
	\caption{Test \ref{subsec:numerical-experiment-square2D-variablenu}. Displacement/velocity fields and pressure surface plot for the first and fourth eigenmodes for $\nu=0.49999$ and discontinuous $\mu$ and $\lambda$.}
	\label{fig:velocities-nu05}
\end{figure}

\subsection{2D L-shaped domain}\label{subsec:2d-lshape} In this experiment we put to the test the performance of the estimator when the domain present two types of singularities and also have discontinuous Lamé parameters. The domain in consideration is the L-shape domain $\Omega:=(-1,1)^2\backslash\left((0,1)\times (-1,0)\right)$. More precisely, we consider the domain to be free of stress at $y=-1$ and $x=1$, and clamped in the rest of the boundary. The Lamé parameters for this experiments depend on the Young modulus, which is given by
$$
E(x)=\left\{
\begin{aligned}
	&\sqrt{x^2+y^2 +4}, &\text{ if } x\in\Omega_1,\\
	&\sqrt{x^2+y^2 +2}, &\text{ if } x\in\Omega_2,\\
	&\sqrt{x^2+y^2 +4}, &\text{ if } x\in\Omega_3.\\
\end{aligned}
\right.
$$
A sample of the initial mesh and the subdomains is depicted in Figure \ref{fig:lshape2D-inicial}. Note that the domain contains a singularity at $(x,y)=(0,0)$, followed by the points where boundary conditions change from Neumann to Dirichlet type. Also, the Young modulus present discontinuities along the boundary of each subdomain. Hence, suboptimality with uniform refinements is assured. The exact lowest computed eigenvalues for each value of $\nu$ are given in Table \ref{tabla:lshape2D-references} below.

\begin{table}[hbt!]
	{\footnotesize
		\centering
		\begin{tabular}{|c|c|}
			\hline
			$\nu$             & $\sqrt{\widehat{\kappa}_1}$ \\
			\hline
			0.35  & $2.66229608762752$  \\
			0.49 & $2.56684460283184$ \\
			0.5 &   $2.56098604836068$  \\
			\hline
			\hline
		\end{tabular}
		\caption{Test \ref{subsec:2d-lshape}. Reference lowest computed eigenvalues for different values of $\nu$.}
		\label{tabla:lshape2D-references}}
\end{table}

\begin{figure}[hbt!]
	\centering
	\includegraphics[scale=0.0645,trim=30cm 2cm 38cm 2cm,clip]{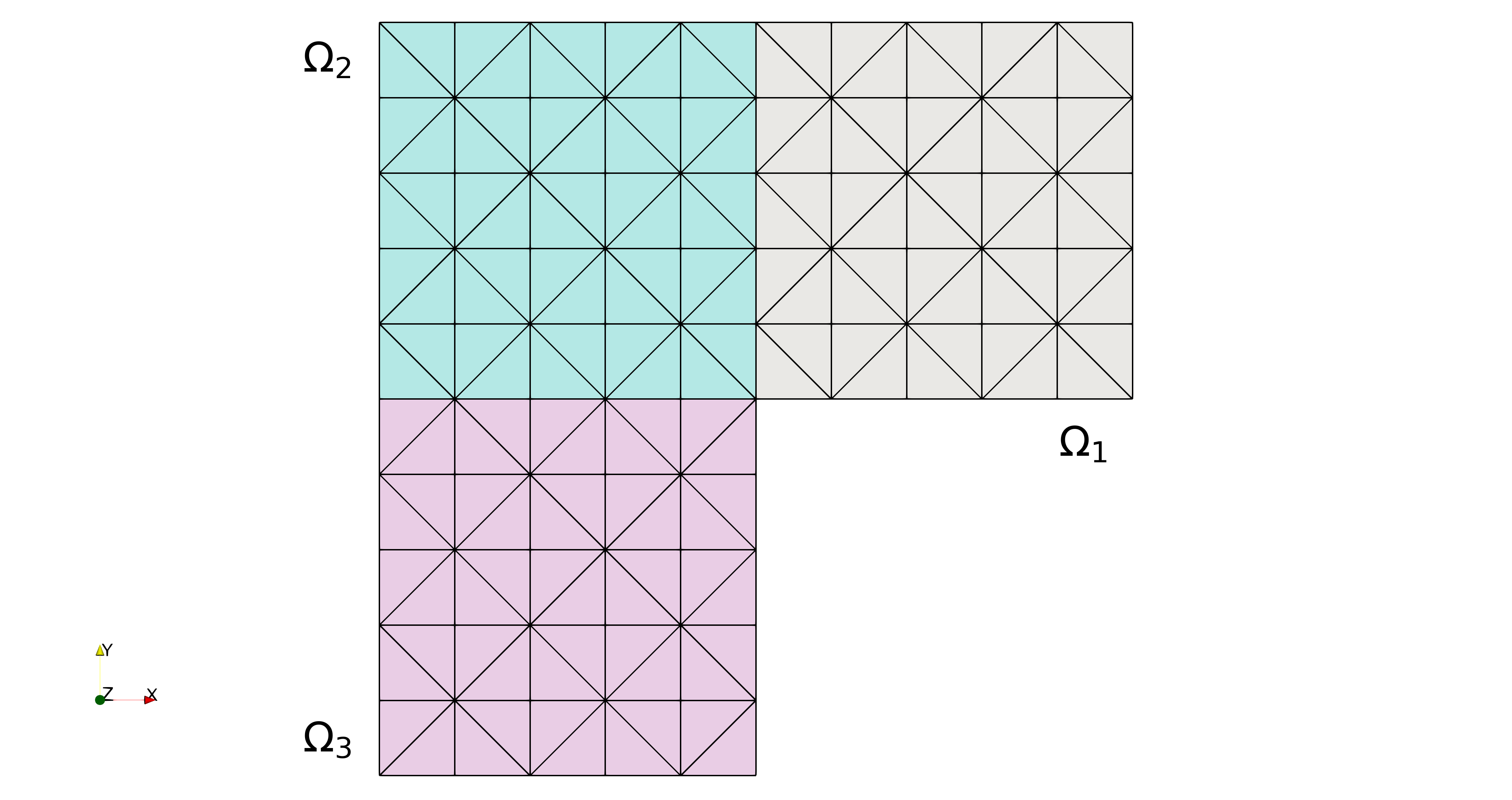}
	\caption{Test \ref{subsec:2d-lshape}. Initial mesh of the domain $\Omega$ subdivided in $\Omega_i$, $i=1,2,3$ subdomains.}
	\label{fig:lshape2D-inicial}
\end{figure}

Evidence of uniform and adaptive refinements are observed in Figure \ref{fig:lshape2D-error}, where suboptimal rate of convergence is observed if uniform refinements are used. Indeed, the convergence rate for $\nu=0.35$ is about $\mathcal{O}(\texttt{dof}^{-0.48})$, while $\mathcal{O}(\texttt{dof}^{-0.44})$ is observed for $\nu=0.49$ and $\nu=0.50$. For the adaptive refinements, an experimental rate $\mathcal{O}(\texttt{dof}^{-1.01})$ was observed for $\nu=0.35$ and $\nu=0.49$, whereas $\mathcal{O}(\texttt{dof}^{-0.99})$ was obtained for $\nu=0.5$. On Figure \ref{fig:lshape2D-effectivity} we observe that the estimator behaves similar to $\mathcal{O}(h^2)$ for all the selected values of $\nu$, obtaining a bounded efficiency between $0.0036$ and $0.0041$.  Figure \ref{fig:lshape-2D-adaptive} portray different meshes obtained in the adaptive refinements. It notes that the refinements are concentrated in the re-entrant corner, the four points where Dirichlet boundary changes to Neumann, and also on the subdomains boundary, where the physical parameters have jumps. Similar to the square domain case, the number of elements marked for $\nu=0.35$ is greater than those marked with $\nu=0.50$. Figure \ref{fig:lshape-2D-uh-ph} present the corresponding lowest computed velocity and pressure  eigenmodes, where we observe very small difference in the field $\bu_{1,h}$, and evident discontinuities for the pressure contour plot in the subdomain boundaries are observed.

\begin{figure}[t!]
	\centering
	\includegraphics[scale=0.45]{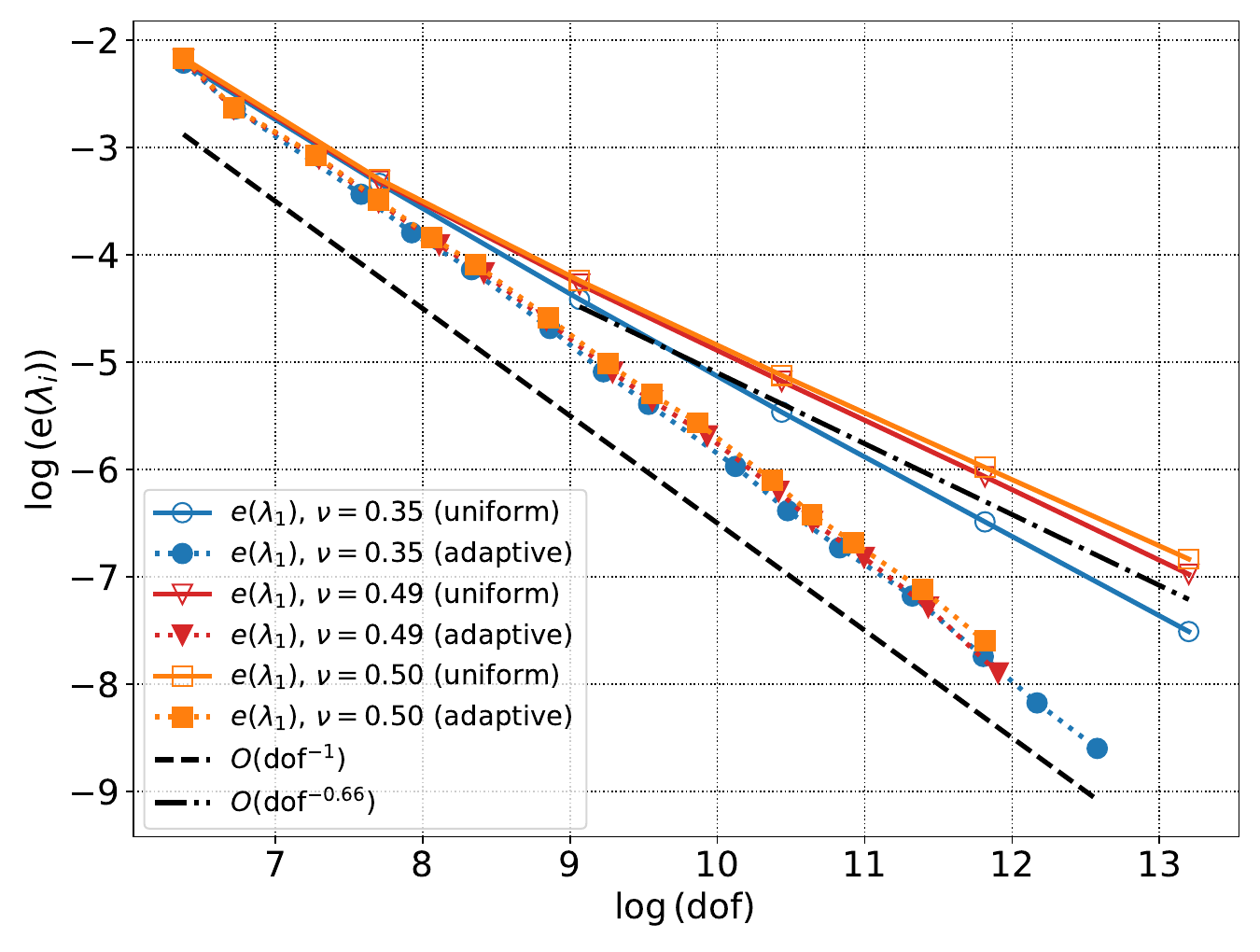}
	\caption{Example \ref{subsec:2d-lshape}.  Error curves obtained from the adaptive algorithm in the Lshaped domain compared with the lines $\mathcal{O}(\texttt{dof}^{-0.66})$ and $\mathcal{O}(\texttt{dof}^{-1})$.}
	\label{fig:lshape2D-error}
\end{figure}

\begin{figure}[t!]
	\centering
	\begin{minipage}{0.49\linewidth}\centering
		\includegraphics[scale=0.45]{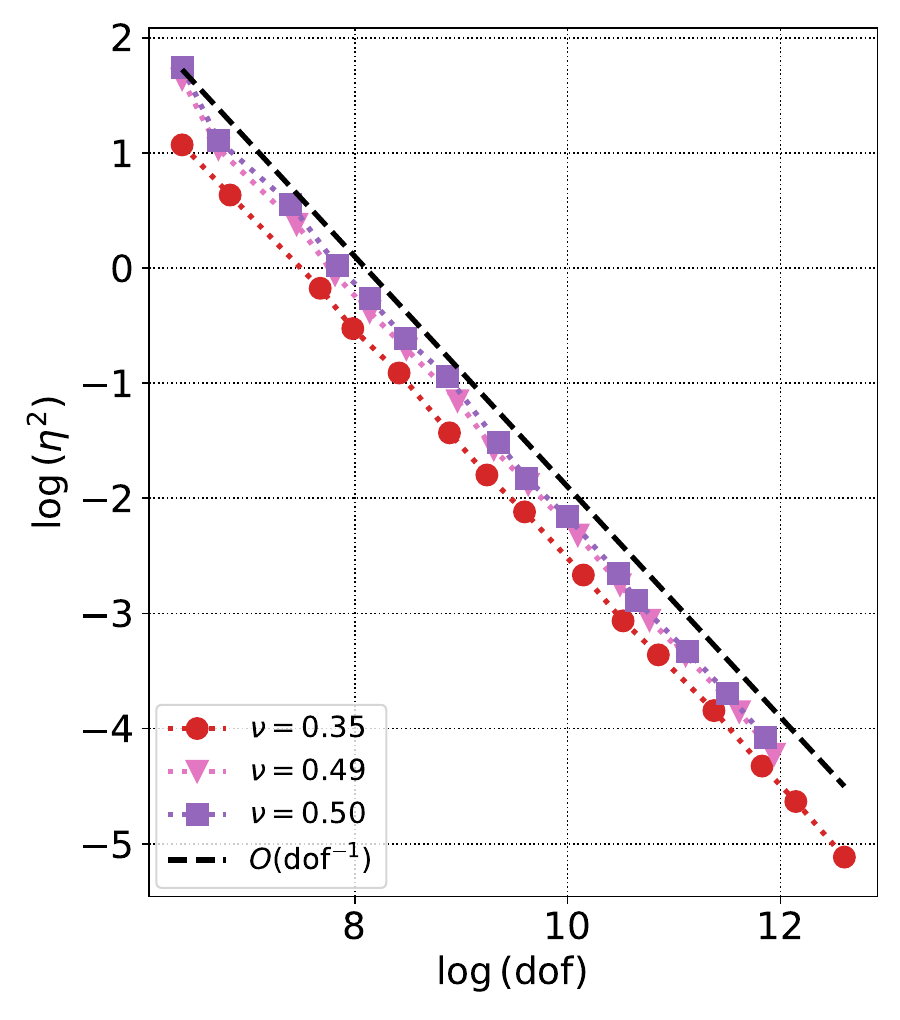}
	\end{minipage}
	\begin{minipage}{0.49\linewidth}\centering
		\includegraphics[scale=0.45]{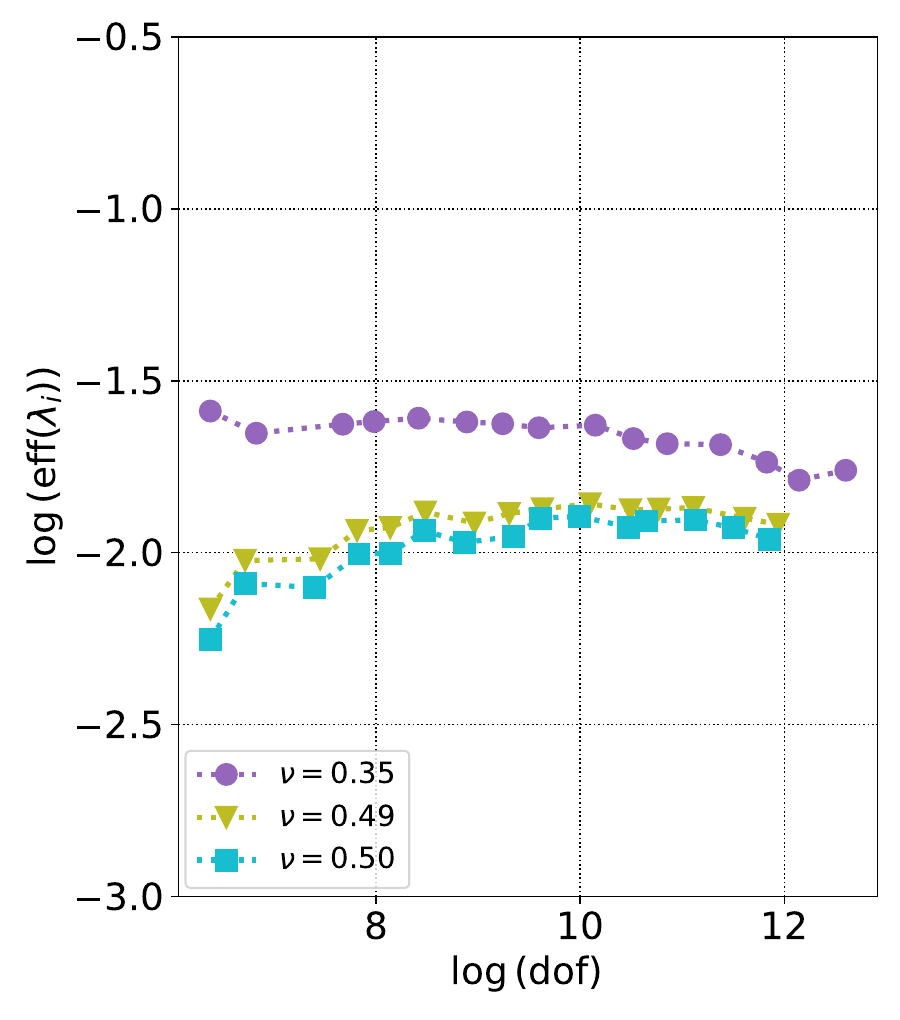}
	\end{minipage}
	\caption{Example \ref{subsec:2d-lshape}.  Estimator and efficiency curves obtained from the adaptive algorithm in the Lshaped domain with different values of $\nu$.}
	\label{fig:lshape2D-effectivity}
\end{figure}

\begin{figure}[hbt!]
	\centering
	\begin{minipage}{0.32\linewidth}
		\centering
		\includegraphics[scale=0.0645,trim=38cm 2cm 38cm 2cm,clip]{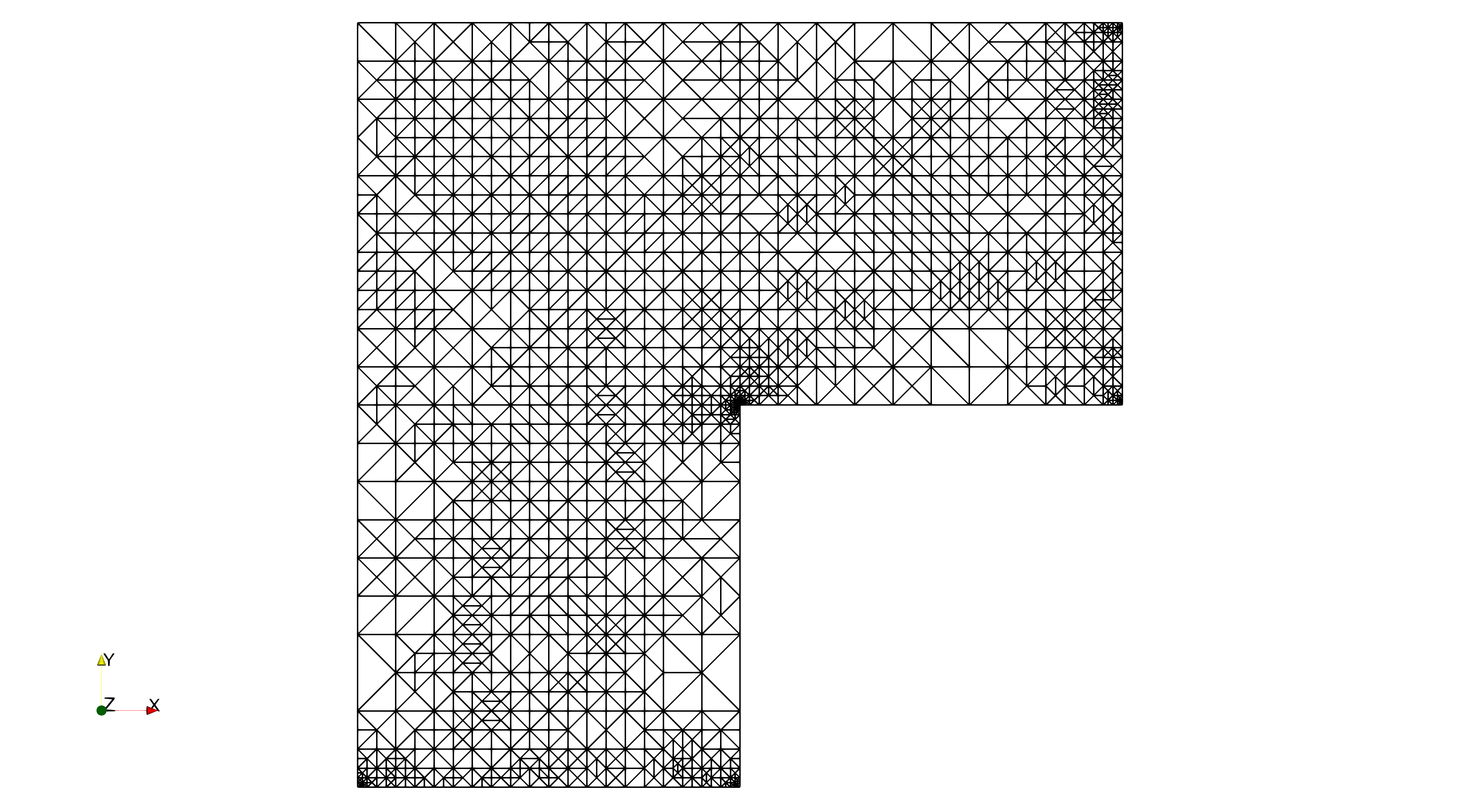}\\
	\end{minipage}
	\begin{minipage}{0.32\linewidth}
		\centering
		\includegraphics[scale=0.0645,trim=38cm 2cm 38cm 2cm,clip]{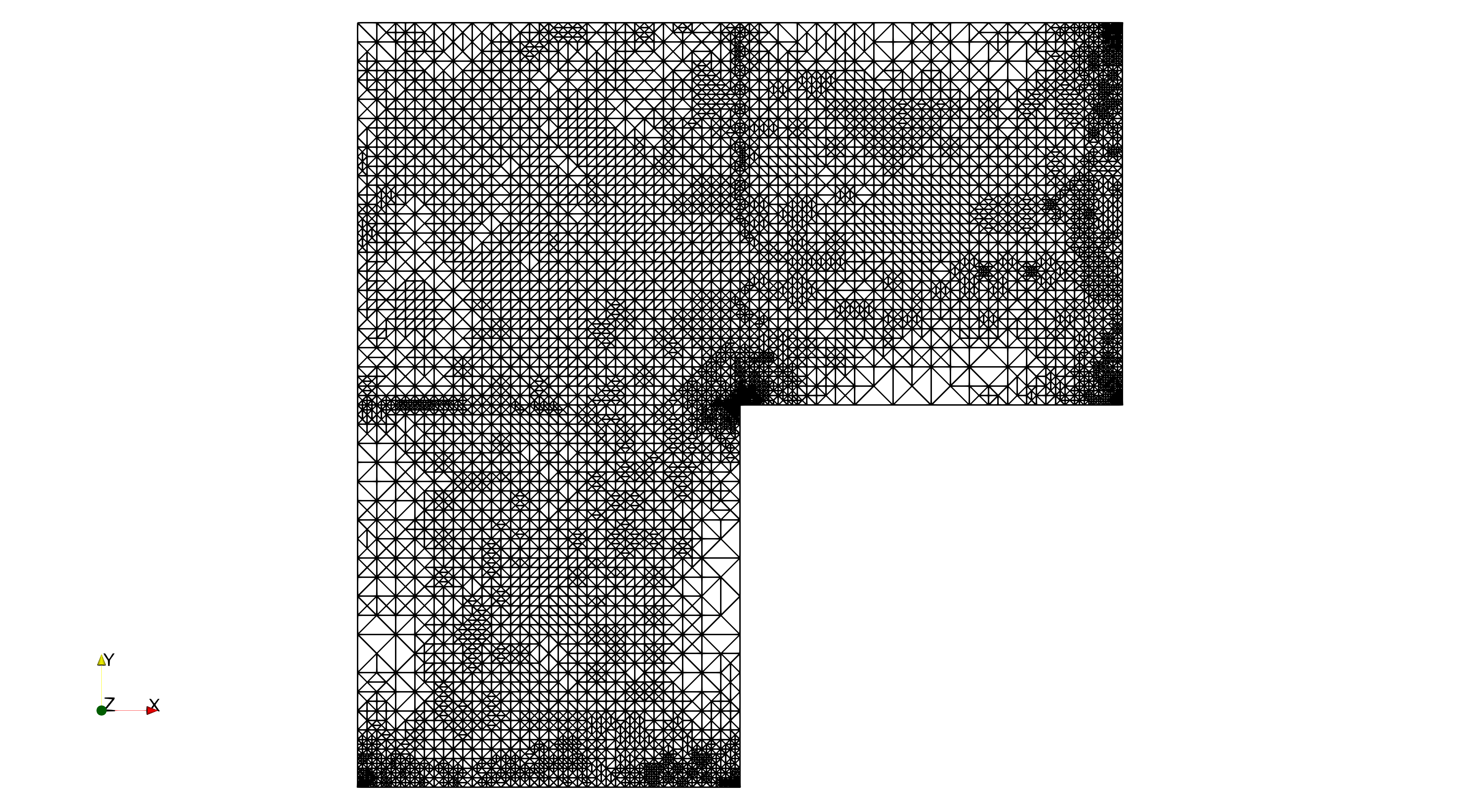}\\
	\end{minipage}
	\begin{minipage}{0.32\linewidth}
		\centering
		\includegraphics[scale=0.0645,trim=38cm 2cm 38cm 2cm,clip]{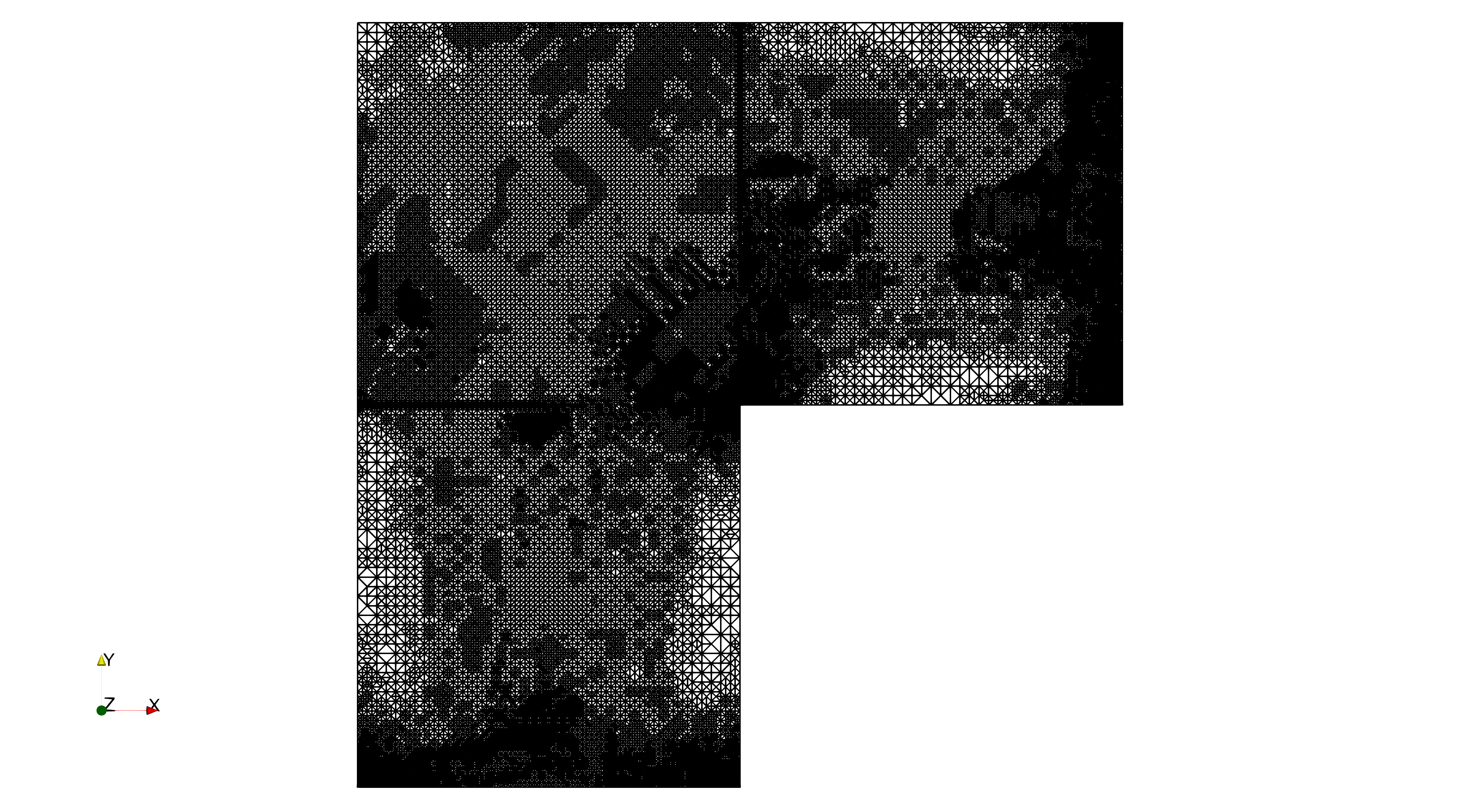}\\
	\end{minipage}\\
	\begin{minipage}{0.32\linewidth}
		\centering
		\includegraphics[scale=0.0645,trim=38cm 2cm 38cm 2cm,clip]{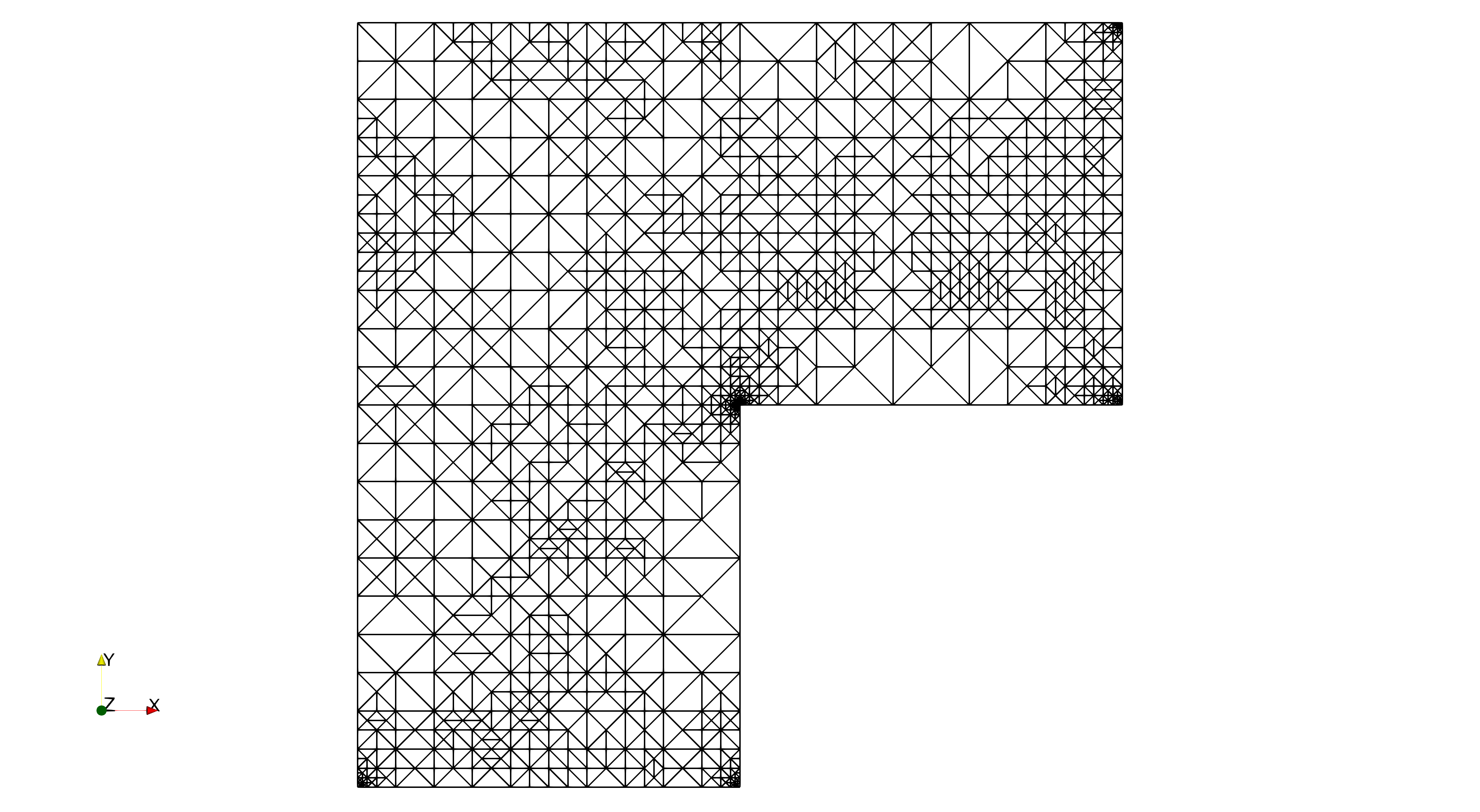}\\
	\end{minipage}
	\begin{minipage}{0.32\linewidth}
		\centering
		\includegraphics[scale=0.0645,trim=38cm 2cm 38cm 2cm,clip]{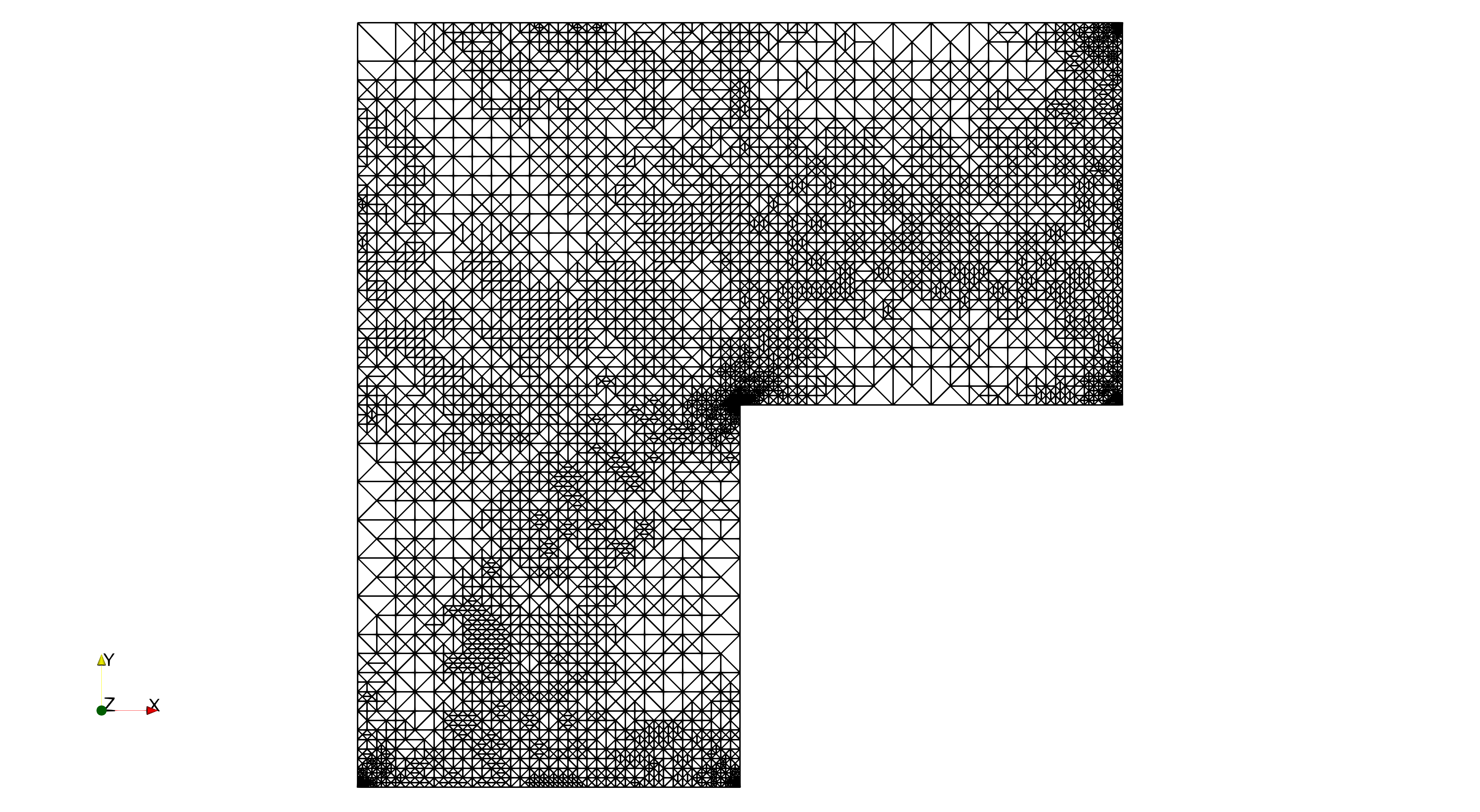}\\
	\end{minipage}
	\begin{minipage}{0.32\linewidth}
		\centering
		\includegraphics[scale=0.0645,trim=38cm 2cm 38cm 2cm,clip]{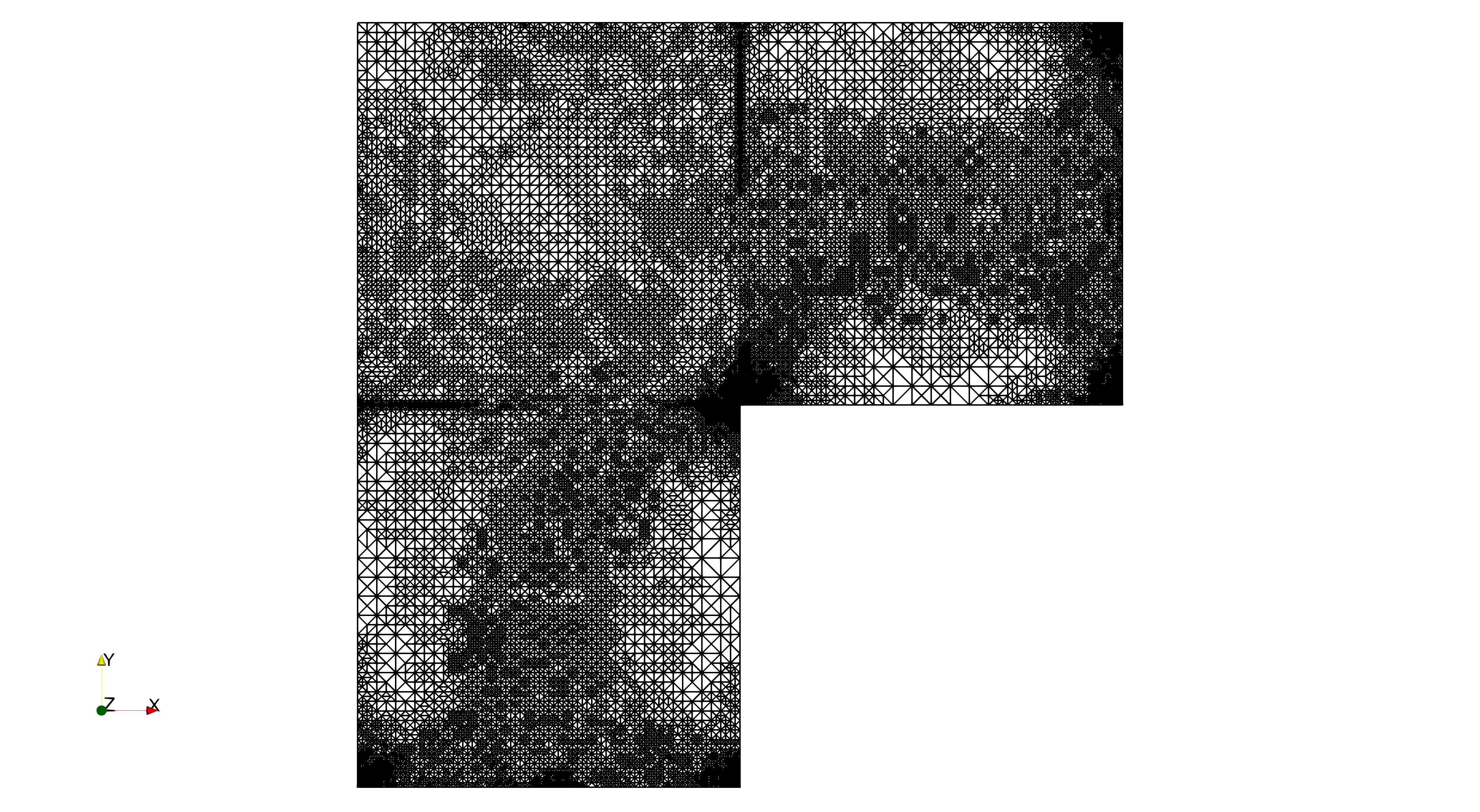}\\
	\end{minipage}
	\caption{Test \ref{subsec:2d-lshape}. Intermediate meshes of the Lshaped domain obtained with the adaptive algorithm and different values of $\nu$. Top row: Meshes with 10163, 50455 and  289949 dofs with $\eta=0.35$. Bottom row: Meshes with 7006, 31911  and 135217 dofs with $\eta=0.5$. }
	\label{fig:lshape-2D-adaptive}
\end{figure}

\begin{figure}[hbt!]
	\centering
	\begin{minipage}{0.32\linewidth}
		\centering
		\includegraphics[scale=0.0645,trim=38cm 0cm 38cm 2cm,clip]{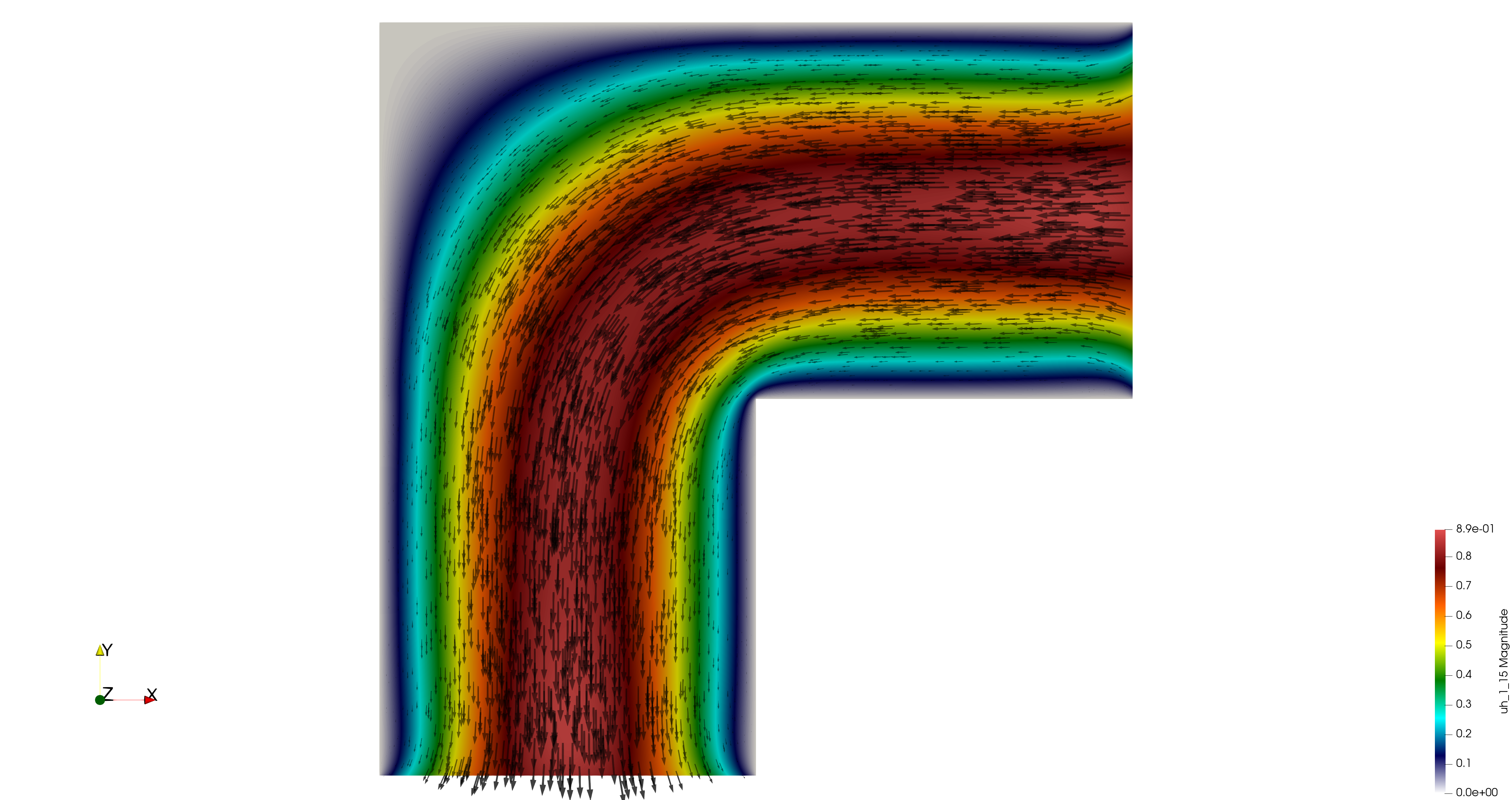}\\
	\end{minipage}
	\begin{minipage}{0.32\linewidth}
		\centering
		\includegraphics[scale=0.0645,trim=38cm 0cm 38cm 2cm,clip]{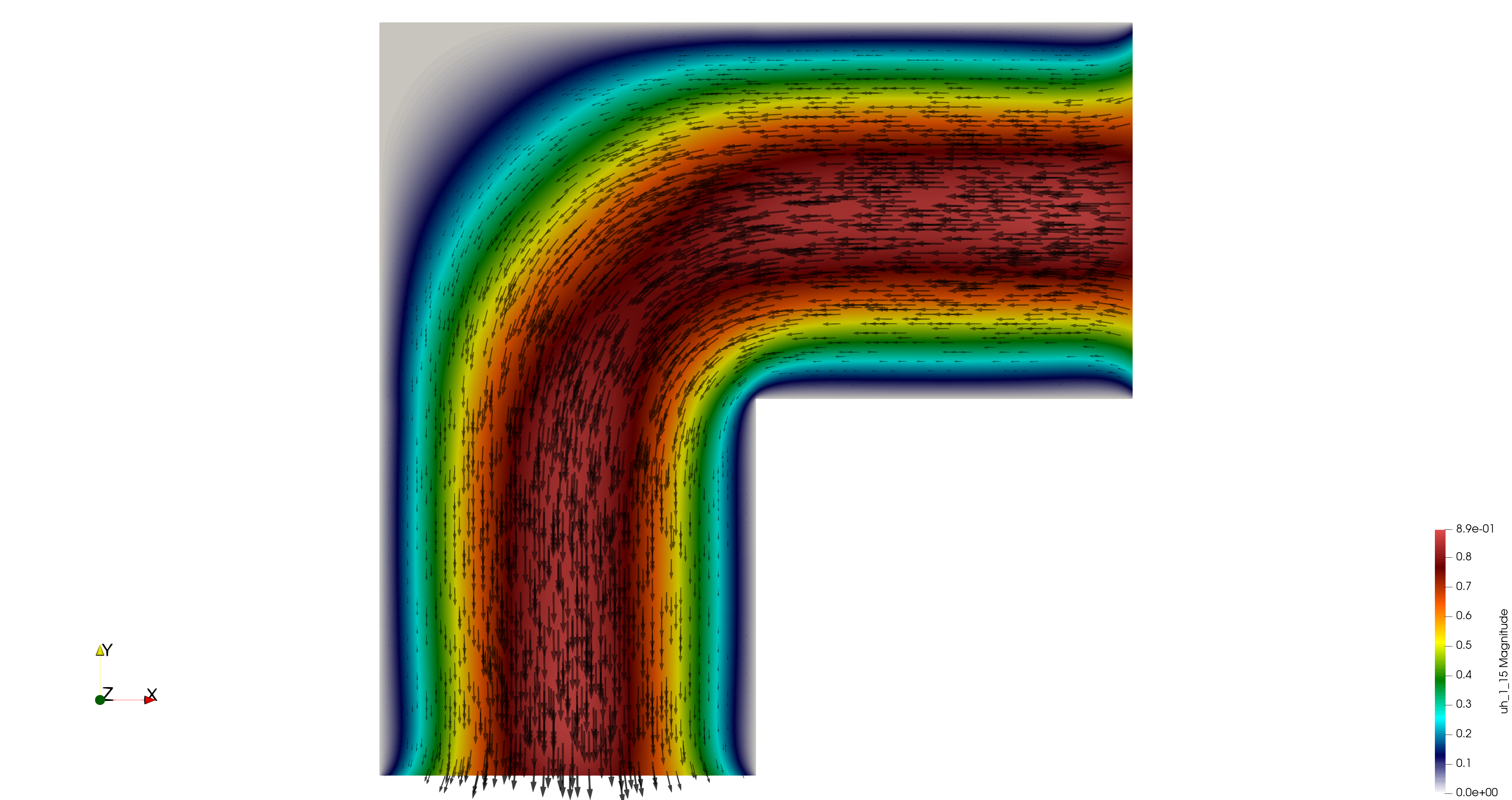}\\
	\end{minipage}
	\begin{minipage}{0.32\linewidth}
		\centering
		\includegraphics[scale=0.0645,trim=38cm 0cm 38cm 2cm,clip]{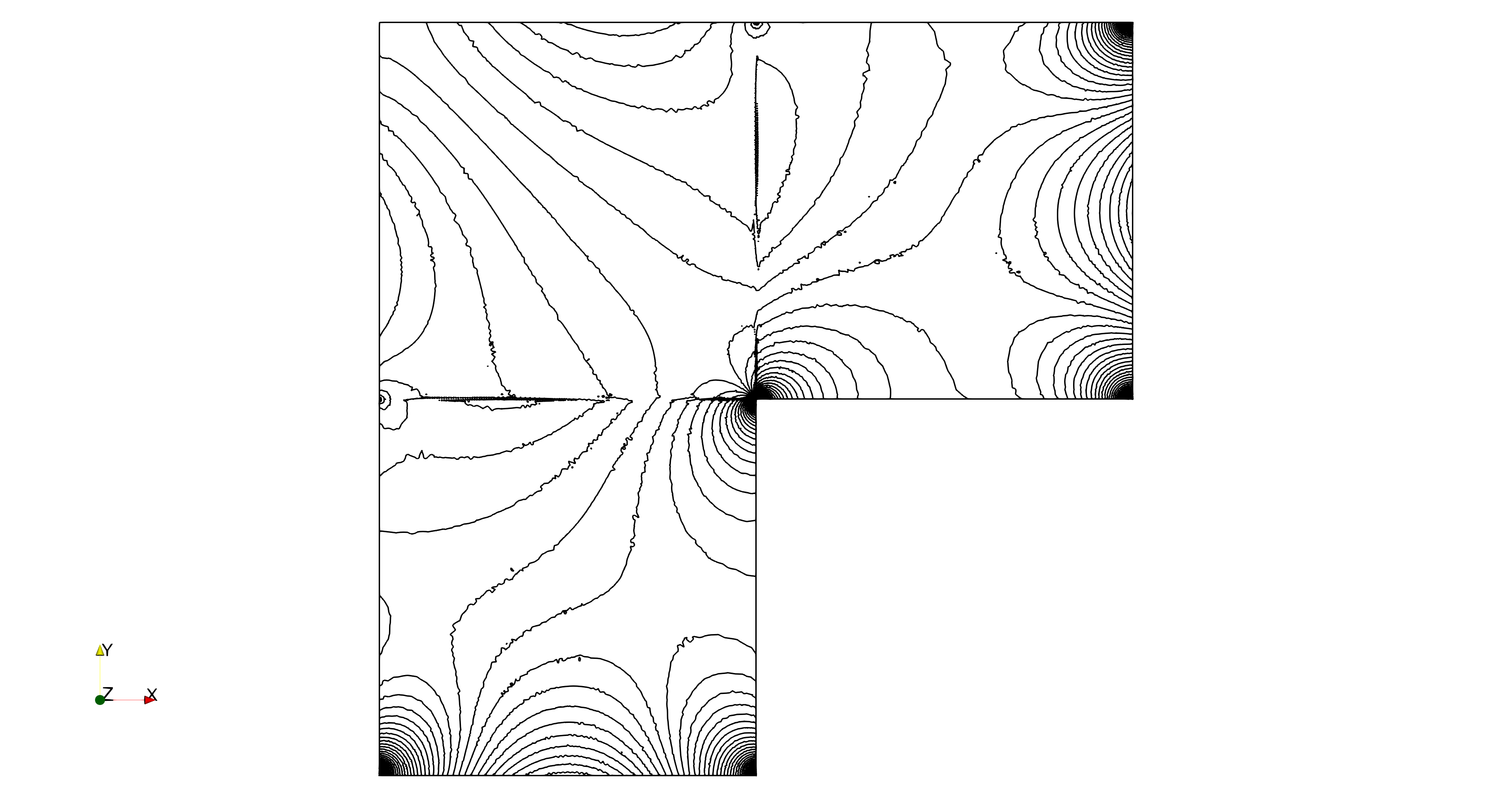}\\
	\end{minipage}
	\caption{Test \ref{subsec:2d-lshape}. Velicity fields for the first computed eigenfrequency for $\nu=0.35$ (left) and $\nu=0.50$ (right) together with the singular pressure contour plot, similar for both cases. }
	\label{fig:lshape-2D-uh-ph}
\end{figure}

\subsection{A three dimensional test}\label{subsec:numerical-experiment-cube3D-bottomfixed}
Now we report some results for a three dimensional test. the domain under consideration is the unit cube $\O:=(0,1)^3$. For simplicity, because of the high computational cost when a 3D test is performed for eigenvalue problems, we focus on the lowest order finite elements. The boundary conditions for this test is the same as Test \ref{subsec:numerical-experiment-square2D}, namely, we assume that the cube is fixed at the bottom and free of stress in the rest of its geometry. For simplicity, we consider $E=\rho=1$.

Tables \ref{tabla:square-MINI-3D-mixed-boundary} and \ref{tabla:square-TH-3D-mixed-boundary} shows the convergence behaviour when computing the first four eigenmodes. It notes that since we have edges singularities along the bottom facet of the cube,  the convergence rate behaves like $\mathcal{O}(h^r)$, with $r\geq 1.36$ for the elastic case, whereas $r\geq 1.2$ for the limit Stokes scenario. We also note that both methods deals with the singularity in a slightly different manner. For instance, a quadratic order is observed in the fourth eigenvalue when Taylor-Hood is used. However, this is still not optimal for this family. On the other hand, suboptimal convergence is also observe for this eigenvalue for the mini-element family, where the rate behaves like $\mathcal{O}(h^{1.7})$  for the three $\nu$ scenarios. Finally, we present in Figure \ref{fig:cubos-3D-bottomfix} the first and fourth lowest computed eigenmodes from Table \ref{tabla:square-TH-3D-mixed-boundary}. We depict the displacement ($\nu=0.35$) or the velocity field ($\nu=0.5$) accompanied with their corresponding pressures surface plot.

For the adaptive refinement, we consider the first eigenvalue. Note that this choice characterized by the fact this eigenvalue has multiplicity $2$, hence it is beyond the theoretical results since we have no control over the distance function between this two values on the a posteriori analysis. However, as we observe in Figure \ref{fig:cube3D-afem-error}, an optimal rate is obtained for 10 iterations of the adaptive algorithm. For instance, a rate of $\mathcal{O}(\texttt{dof}^{-0.66})$ was observed for the last 4 refinements. Evidence of the adaptive meshes are depicted in Figure \ref{fig:cubos-3D-bottomfix_adaptivemeshes}, where we observe that the refinements are performed near the edges where the boundary condition changes. On Figure \ref{fig:cubos-3D-bottomfix} we observe that the first eigenmode for each $\nu$ has pressure singularities precisely on the zones where the adaptive algorithm perform the refinements. We also present the fourth eigenmode, which shows a high pressure gradient near the bottom edges, resulting in the suboptimal behavior observed in Tables \ref{tabla:square-MINI-3D-mixed-boundary} and \ref{tabla:square-TH-3D-mixed-boundary}.

\begin{table}[hbt!]
	{\footnotesize
		\begin{center}
			\caption{Test \ref{subsec:numerical-experiment-cube3D-bottomfixed}. Lowest computed eigenvalues in the unit cube domain with mixed boundary conditions using the mini-element family for different values of $\nu$.}
			\label{tabla:square-MINI-3D-mixed-boundary}
			\begin{tabular}{|c c c c |c| c| }
				\hline
				\hline
				$N=10$             &  $N=15$         &   $N=20$         & $N=25$ & Order & $\sqrt{\widehat{\kappa}_{extr}}$  \\ 
				\hline
				\multicolumn{6}{c}{$\nu=0.35$}  \\
				\hline
				    0.6794  &     0.6736  &     0.6711  &     0.6698  & 1.51 &     0.6666  \\
				0.6830  &     0.6753  &     0.6721  &     0.6705  & 1.59 &     0.6666  \\
				0.9234  &     0.9066  &     0.9002  &     0.8972  & 1.85 &     0.8911  \\
				1.6203  &     1.6138  &     1.6110  &     1.6094  & 1.41 &     1.6052  \\
				1.7831  &     1.7661  &     1.7596  &     1.7564  & 1.77 &     1.7496  \\
				\hline
				\multicolumn{6}{c}{$\nu=0.49$}  \\
				\hline
				    0.6946  &     0.6842  &     0.6796  &     0.6770  & 1.34 &     0.6696  \\
				0.6985  &     0.6861  &     0.6807  &     0.6777  & 1.41 &     0.6697  \\
				0.8833  &     0.8665  &     0.8600  &     0.8567  & 1.77 &     0.8500  \\
				1.6838  &     1.6670  &     1.6597  &     1.6557  & 1.42 &     1.6450  \\
				1.7479  &     1.7282  &     1.7204  &     1.7166  & 1.72 &     1.7082  \\
				\hline
				\multicolumn{6}{c}{$\nu=0.5$}  \\
				\hline    
				0.6974  &     0.6862  &     0.6812  &     0.6784  & 1.33 &     0.6702  \\
				0.7015  &     0.6882  &     0.6823  &     0.6791  & 1.40 &     0.6703  \\
				0.8811  &     0.8641  &     0.8575  &     0.8543  & 1.75 &     0.8473  \\
				1.6912  &     1.6728  &     1.6649  &     1.6605  & 1.42 &     1.6488  \\
				1.7464  &     1.7262  &     1.7182  &     1.7143  & 1.71 &     1.7055  \\
				\hline
				\hline
			\end{tabular}
	\end{center}
	}
\end{table}
\begin{table}[hbt!]
	{\footnotesize
		\begin{center}
			\caption{Test \ref{subsec:numerical-experiment-cube3D-bottomfixed}. Lowest computed eigenvalues in the unit cube domain with mixed boundary conditions using the Taylor-Hood family for different values of $\nu$.}
			\label{tabla:square-TH-3D-mixed-boundary}
			\begin{tabular}{|c c c c |c| c| }
				\hline
				\hline
				$N=10$             &  $N=15$         &   $N=20$         & $N=25$ & Order & $\sqrt{\widehat{\kappa}_{extr}}$  \\ 
				\hline
				\multicolumn{6}{c}{$\nu=0.35$}  \\
				\hline
				 0.6680  &     0.6673  &     0.6670  &     0.6669  & 1.42 &     0.6664  \\
				0.6681  &     0.6674  &     0.6671  &     0.6669  & 1.47 &     0.6664  \\
				0.8920  &     0.8916  &     0.8915  &     0.8914  & 2.28 &     0.8914  \\
				1.6073  &     1.6064  &     1.6060  &     1.6057  & 1.39 &     1.6051  \\
				1.7509  &     1.7505  &     1.7503  &     1.7502  & 2.44 &     1.7502  \\
				\hline
				\multicolumn{6}{c}{$\nu=0.49$}  \\
				\hline
				    0.6730  &     0.6715  &     0.6708  &     0.6704  & 1.27 &     0.6692  \\
				0.6731  &     0.6716  &     0.6709  &     0.6705  & 1.32 &     0.6693  \\
				0.8513  &     0.8508  &     0.8506  &     0.8505  & 1.85 &     0.8502  \\
				1.6493  &     1.6471  &     1.6461  &     1.6454  & 1.20 &     1.6434  \\
				1.7100  &     1.7094  &     1.7091  &     1.7090  & 1.99 &     1.7088  \\
				\hline
				\multicolumn{6}{c}{$\nu=0.5$}  \\
				\hline
				    0.6738  &     0.6723  &     0.6715  &     0.6711  & 1.26 &     0.6698  \\
				0.6740  &     0.6723  &     0.6716  &     0.6711  & 1.31 &     0.6699  \\
				0.8487  &     0.8481  &     0.8479  &     0.8478  & 1.83 &     0.8476  \\
				1.6532  &     1.6509  &     1.6498  &     1.6491  & 1.18 &     1.6469  \\
				1.7074  &     1.7067  &     1.7065  &     1.7064  & 1.96 &     1.7061  \\
				\hline
				\hline
			\end{tabular}
	\end{center}
	}
\end{table}

\begin{figure}[hbt!]
	\centering
	\begin{minipage}{0.49\linewidth}\centering
		\includegraphics[scale=0.07,trim=30cm 2cm 30cm 0cm,clip]{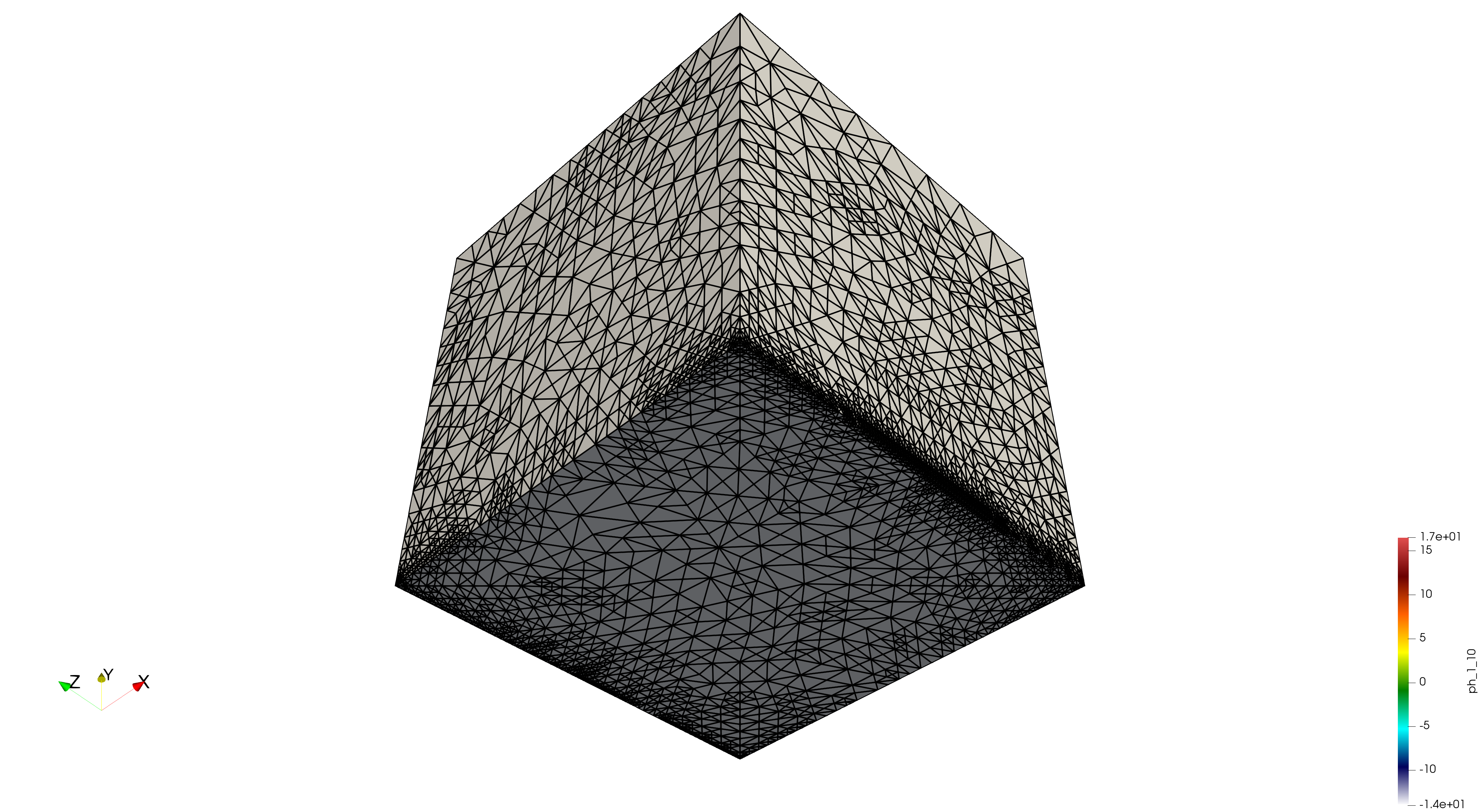}\\
	\end{minipage}
	\begin{minipage}{0.49\linewidth}\centering
		\includegraphics[scale=0.07,trim=30cm 2cm 30cm 0cm,clip]{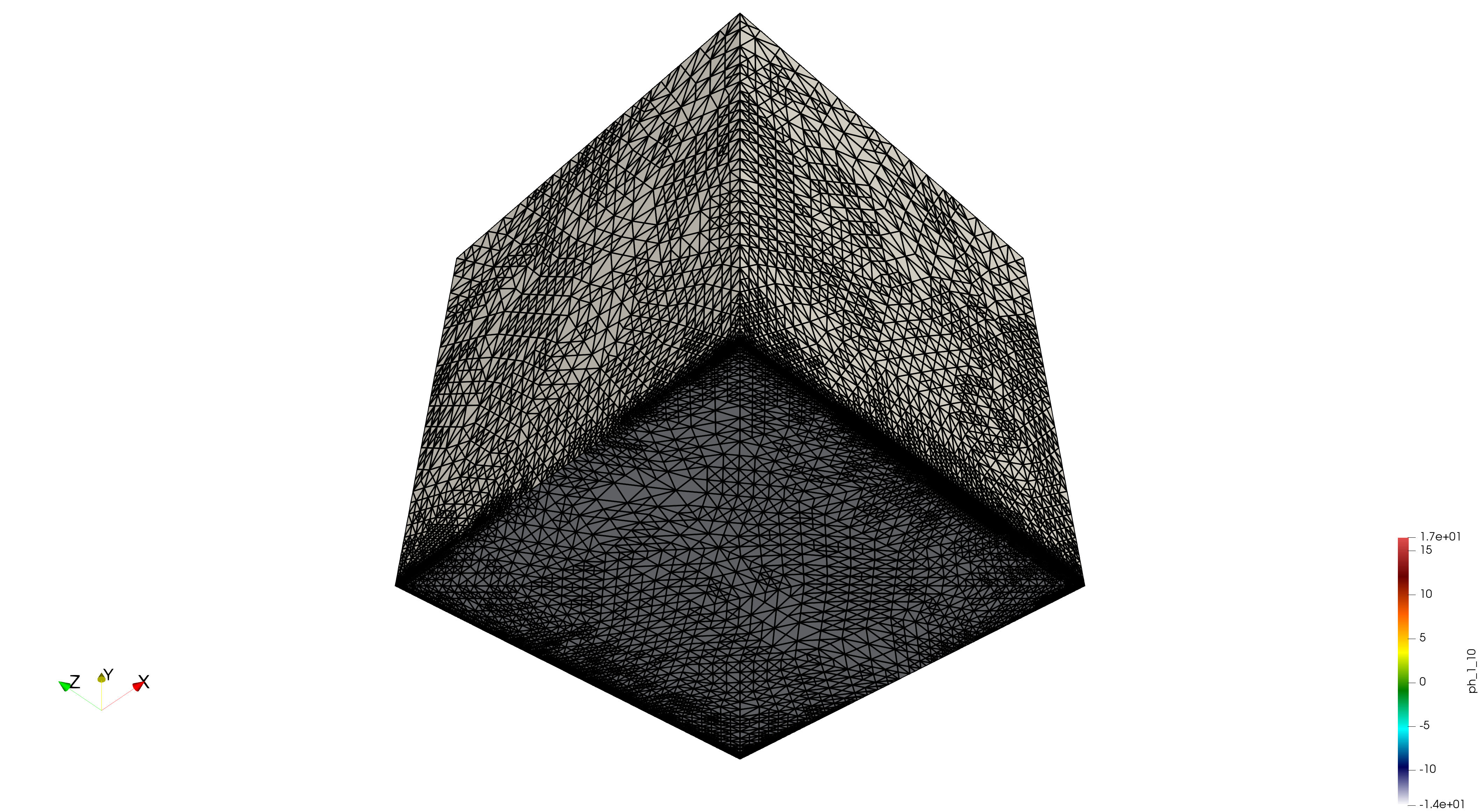}\\
	\end{minipage}\\
	\begin{minipage}{0.49\linewidth}\centering
		\includegraphics[scale=0.07,trim=30cm 2cm 30cm 0cm,clip]{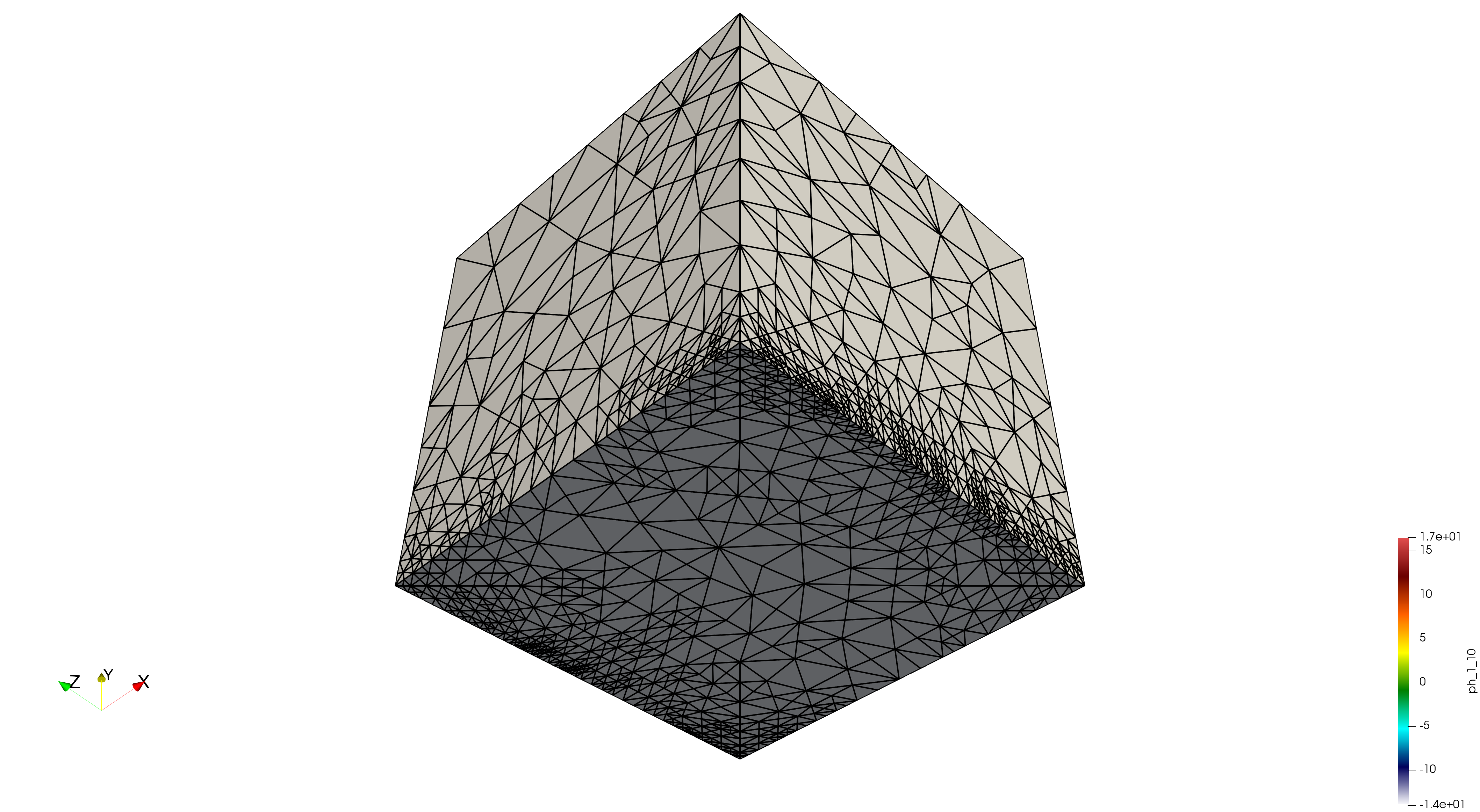}\\
	\end{minipage}
	\begin{minipage}{0.49\linewidth}\centering
		\includegraphics[scale=0.07,trim=30cm 2cm 30cm 0cm,clip]{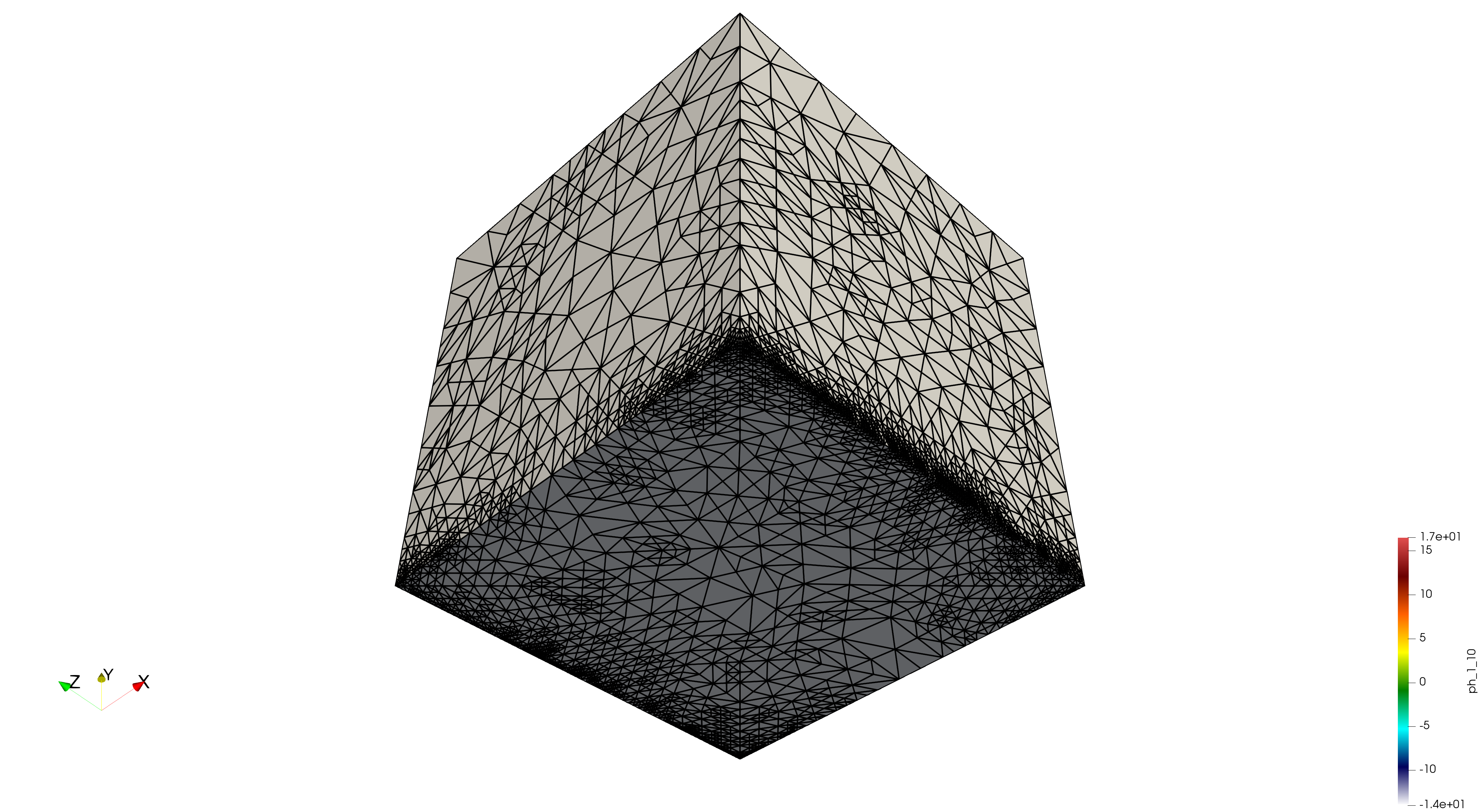}
	\end{minipage}
	\caption{Test \ref{subsec:numerical-experiment-cube3D-bottomfixed}. Bottom view of intermediate meshes of the bottom-clamped unit cube domain obtained with the adaptive algorithm and different values of $\nu$. Top row: Meshes with 473830 and 2146942 dofs with $\eta=0.35$. Bottom row: Meshes with 97753 and 314288 dofs with $\eta=0.5$}
	\label{fig:cubos-3D-bottomfix_adaptivemeshes}
\end{figure}

\begin{figure}[hbt!]
	\centering
	\begin{minipage}{0.49\linewidth}\centering
	\includegraphics[scale=0.08,trim=30cm 5cm 30cm 2cm,clip]{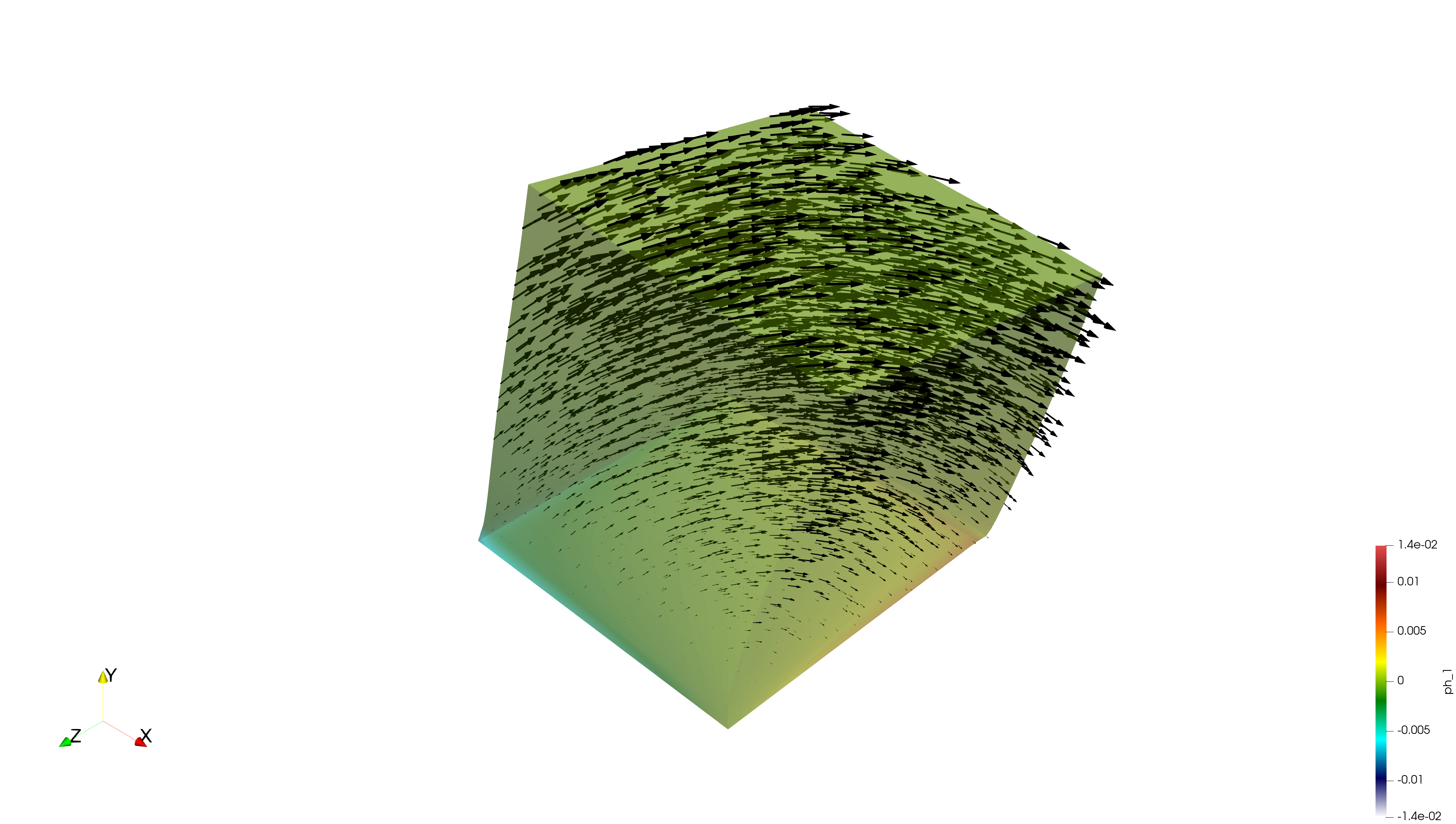}\\
	{\footnotesize $p_{1,h}, \bu_{1,h}, \nu=0.35$}
	\end{minipage}
	\begin{minipage}{0.49\linewidth}\centering
		\includegraphics[scale=0.08,trim=30cm 5cm 30cm 2cm,clip]{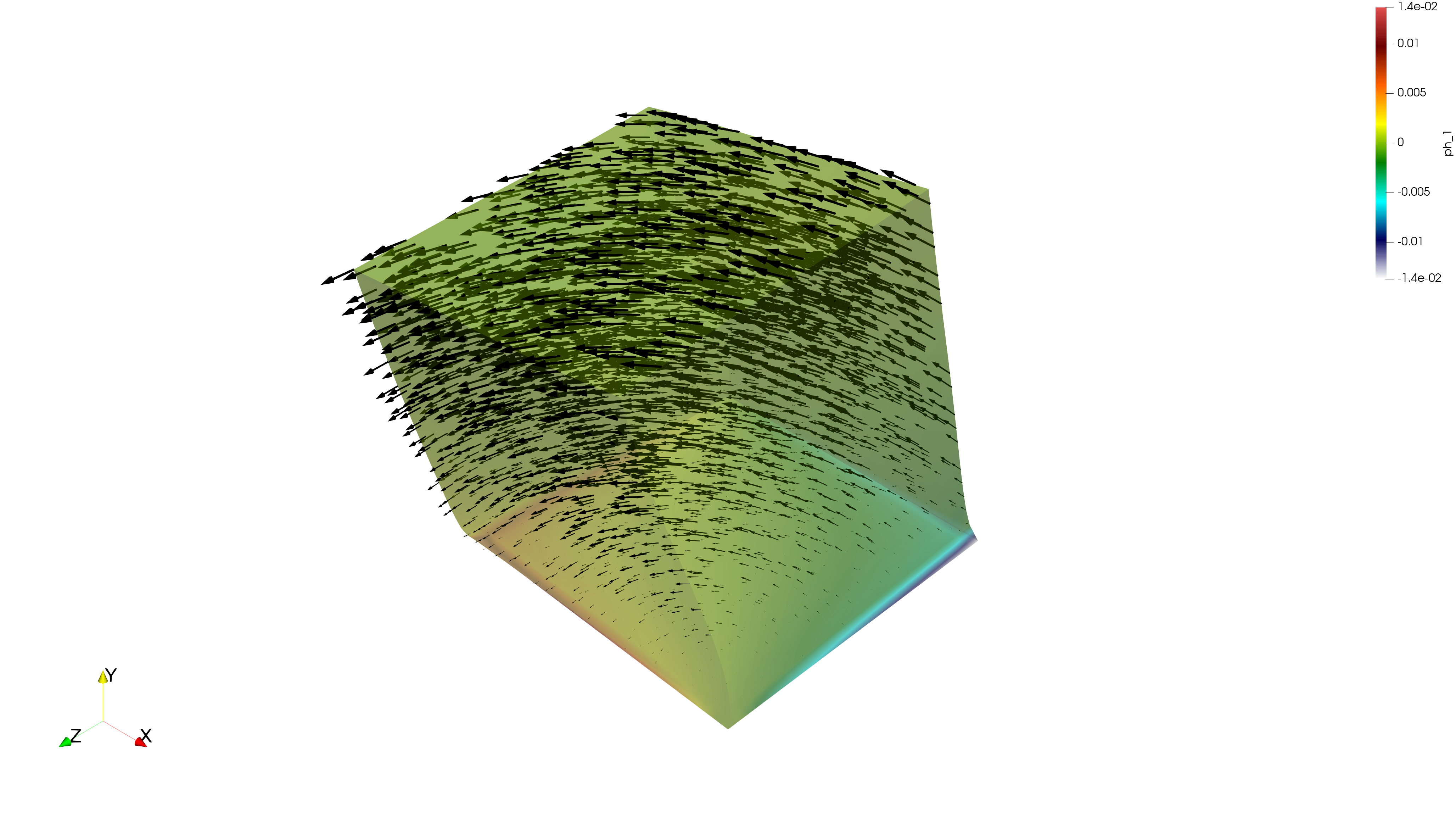}\\
		{\footnotesize $p_{1,h}, \bu_{1,h}, \nu=0.50$}
	\end{minipage}\\
	\begin{minipage}{0.49\linewidth}\centering
		\includegraphics[scale=0.08,trim=30cm 5cm 30cm 2cm,clip]{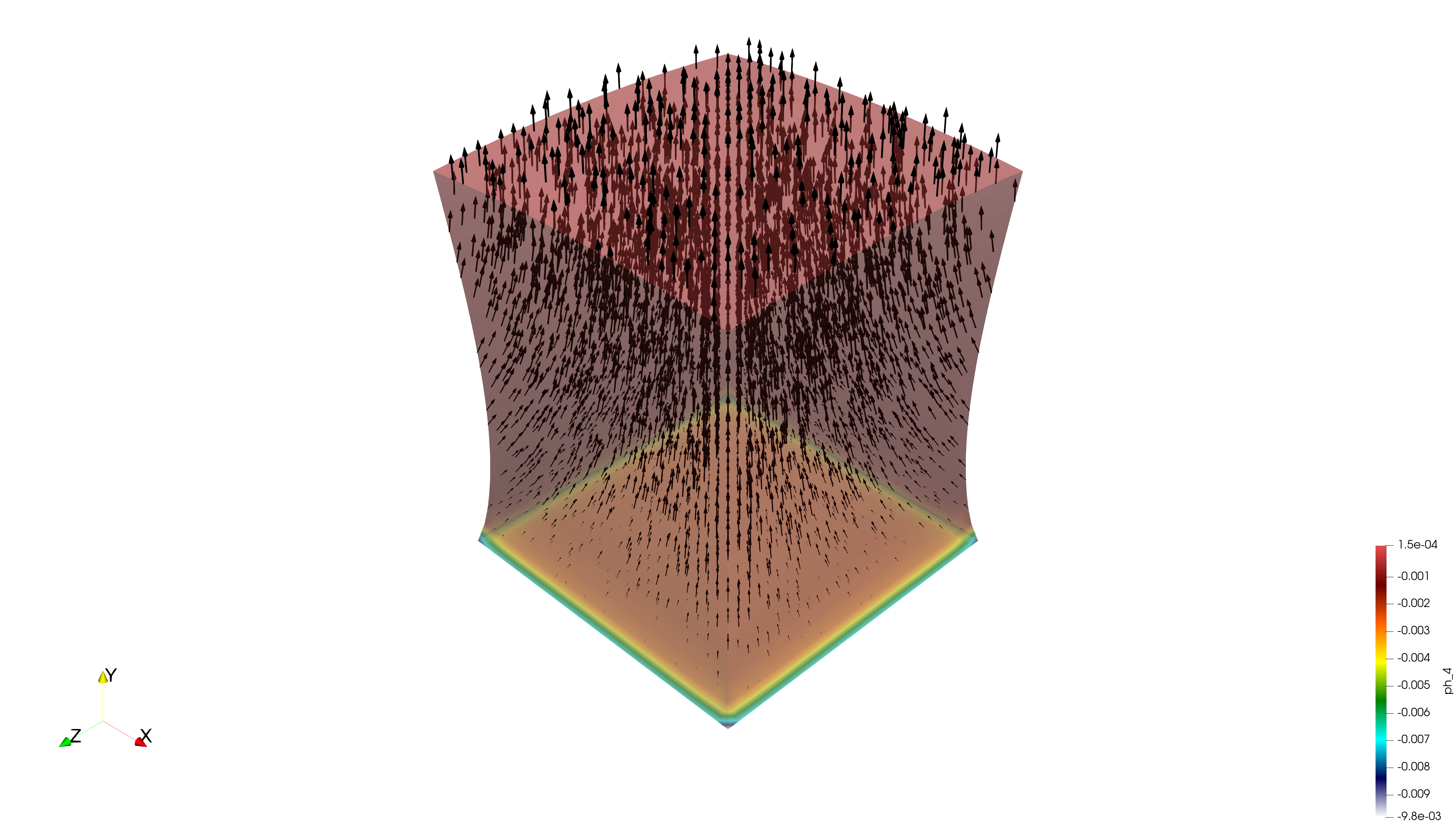}
		{\footnotesize $p_{4,h}, \bu_{4,h}, \nu=0.35$}
	\end{minipage}
	\begin{minipage}{0.49\linewidth}\centering
		\includegraphics[scale=0.08,trim=30cm 5cm 30cm 2cm,clip]{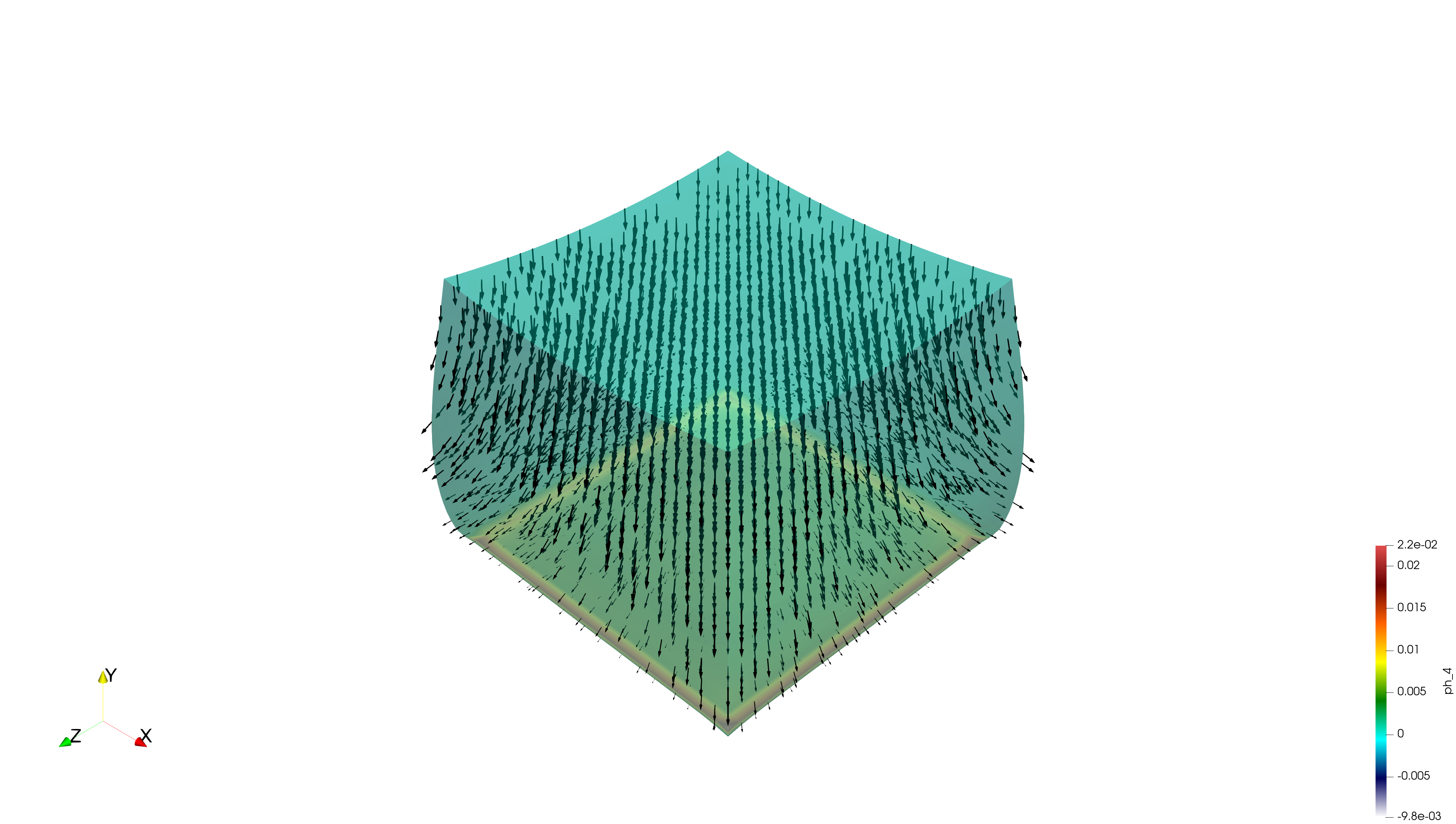}
		{\footnotesize $p_{4,h}, \bu_{4,h}, \nu=0.50$}
	\end{minipage}
	\caption{Test \ref{subsec:numerical-experiment-cube3D-bottomfixed}. Top view of the unit cube 3D representation of the displacement/velocity field, together with the corresponding pressure surface plots for the first and fourth lowest order eigenmodes with $\nu=0.35$ and $\nu=0.50$.}
	\label{fig:cubos-3D-bottomfix}
\end{figure}

\begin{figure}[hbt!]
	\centering
	\includegraphics[scale=0.45]{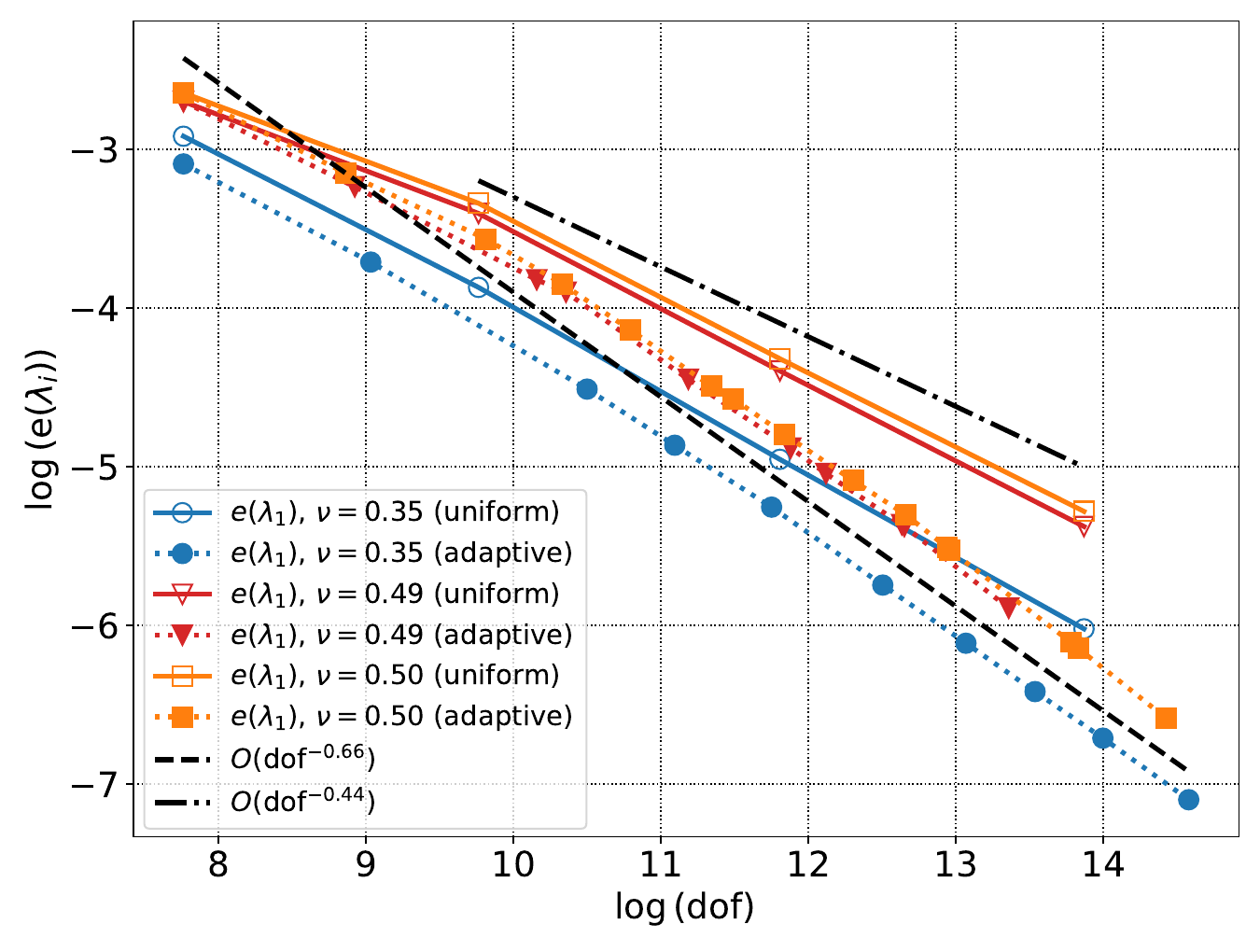}
	\caption{Example \ref{subsec:2d-lshape}. Error curves obtained from the adaptive algorithm in the unit cube domain compared with the lines $\mathcal{O}(\texttt{dof}^{-0.45})$ and $\mathcal{O}(\texttt{dof}^{-0.66})$.}
	\label{fig:cube3D-afem-error}
\end{figure}

\begin{figure}[hbt!]
	\centering
	\begin{minipage}{0.49\linewidth}\centering
	\includegraphics[scale=0.45]{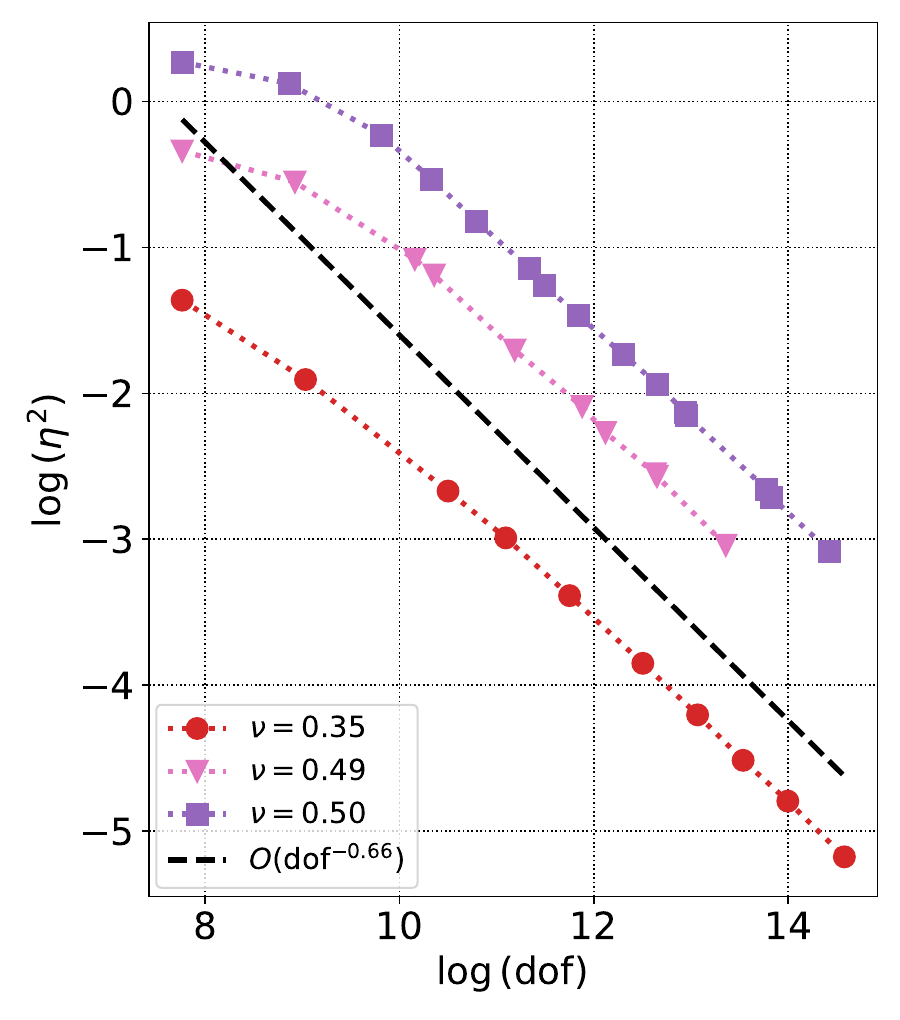}
\end{minipage}
\begin{minipage}{0.49\linewidth}\centering
	\includegraphics[scale=0.45]{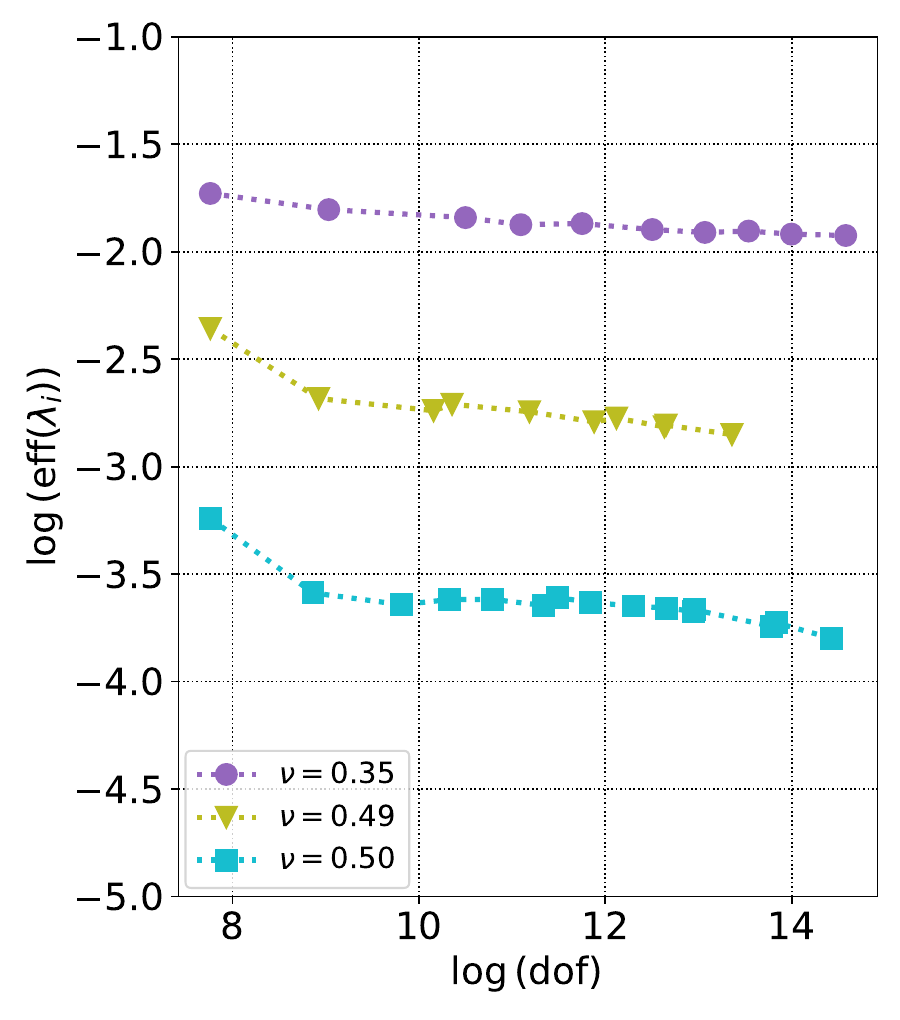}
\end{minipage}
	\caption{Example \ref{subsec:2d-lshape}. Estimator and efficiency curves obtained from the adaptive algorithm in the unit cube domain with different values of $\nu$.}
	\label{fig:cube3D-afem-effectivity}
\end{figure}

\subsection{Robustness test on 2D and 3D geometries}\label{subsec:numerical-experiment-robustness}
In this last experiment we assess the robustness of our estimator in two and three dimensions. More precisely, we consider the unit square $\Omega:=(0,1)^2$ and the unit cube domain $\Omega:=(0,1)^3$, respectively. We will give different constant values to Young's modulus $E$ in order to observe the behavior of $\texttt{eff}$. Let us note that from \eqref{eq:scaled} we have $\mu$ and $\lambda$ we have that $\mu=E/2$ and $\lambda=\frac{2\mu\nu}{1-2\nu}$. For this experiment, $\rho$ is set to be 1. The boundary conditions for this test are the same as Section \ref{subsec:numerical-experiment-square2D} for the square domain and Section \ref{subsec:numerical-experiment-cube3D-bottomfixed} for the cube domain, i.e., we assume that both geometries are fixed at the bottom plane and free of stress in the rest of the facets. Because of the results obtained in previous sections, it is enough to the test the estimator using uniform refinements.

The extrapolated values for this experiment are calculated from a reference value for each value of $\nu$. More precisely, in each test we have the following extrapolated values for $E=10^j, j=1,2,3...$.

\begin{table}[hbt!]
	{\footnotesize
		\centering
		\begin{tabular}{|c|c| c |}
			\hline
			\hline
			&Square domain & Cube domain\\
			\hline
			$\nu$             & $\widehat{\kappa}_1$ & $\widehat{\kappa}_1$ \\
			\hline
			0.35  & $0.46355423498481496\cdot E$ &$0.444317882233217\cdot E$ \\
			0.49 & $0.48938358373431\cdot E$   &$0.44833605908518\cdot E$  \\
			0.5 &   $0.492273855811713\cdot E$  &$0.44915941661888\cdot E$  \\
			\hline
			\hline
		\end{tabular}
		\caption{Test \ref{subsec:numerical-experiment-robustness}. Reference lowest computed eigenvalues for different values of $\nu$ and $E=10^j, j=1,2,3...$ .}
		\label{table:robustness-references}}
\end{table}

Although uniform refinements have been used, giving a convergence like $\mathcal{O}(h^{1.32})$, the estimator remains robust with respect to the physical parameters. In fact,  in Table \ref{tabla:robustness-square-2D3D} we observe different efficiency indexes obtained by manipulating the value of $E$ with several powers of $10$. For each value of $\nu$, properly bounded efficiency indexed are obtained, which are shown to be independent of the value of $E$, even in the limit $\nu=0.5$. We remark that similar values were obtained when testing with $E=10^{-j},j=0,1,2$, confirming the robustness of the proposed estimator.

\begin{table}[hbt!]
	{\footnotesize
		\centering
		\begin{tabular}{|l|ccc|ccc|ccc|}
			\hline
			&\multicolumn{3}{c}{$\nu=0.35$} & \multicolumn{3}{|c}{$\nu=0.49$} & \multicolumn{3}{|c|}{$\nu=0.50$}\\
			\hline
			\multicolumn{10}{|c|}{Unit square domain $\Omega=(0,1)^2$}\\
			\hline
			\texttt{dof}	&\multicolumn{3}{c}{$\texttt{eff}(\widehat{\kappa}_1)$} & \multicolumn{3}{|c}{$\texttt{eff}(\widehat{\kappa}_1)$} & \multicolumn{3}{|c|}{$\texttt{eff}(\widehat{\kappa}_1)$}\\
			\hline
			84 & 1.94e-01 & 1.94e-01 & 1.94e-01 & 9.05e-02 & 9.05e-02 & 9.05e-02 & 8.25e-02 & 8.25e-02 & 8.25e-02 \\
			208 & 1.96e-01 & 1.96e-01 & 1.96e-01 & 8.95e-02 & 8.95e-02 & 8.95e-02 & 8.07e-02 & 8.07e-02 & 8.07e-02 \\
			624 & 2.01e-01 & 2.01e-01 & 2.01e-01 & 9.07e-02 & 9.07e-02 & 9.07e-02 & 8.12e-02 & 8.12e-02 & 8.12e-02 \\
			2128 & 2.07e-01 & 2.07e-01 & 2.07e-01 & 9.28e-02 & 9.28e-02 & 9.28e-02 & 8.27e-02 & 8.27e-02 & 8.27e-02 \\
			7824 & 2.10e-01 & 2.10e-01 & 2.10e-01 & 9.42e-02 & 9.42e-02 & 9.42e-02 & 8.39e-02 & 8.39e-02 & 8.39e-02 \\
			29968 & 2.11e-01 & 2.11e-01 & 2.11e-01 & 9.44e-02 & 9.44e-02 & 9.44e-02 & 8.41e-02 & 8.41e-02 & 8.41e-02 \\
			\hline
			\multicolumn{10}{|c|}{Unit cube domain $\Omega=(0,1)^3$}\\
			\hline
			\texttt{dof}	&\multicolumn{3}{c}{$\texttt{eff}(\widehat{\kappa}_1)$} & \multicolumn{3}{|c}{$\texttt{eff}(\widehat{\kappa}_1)$} & \multicolumn{3}{|c|}{$\texttt{eff}(\widehat{\kappa}_1)$}\\
			\hline
			2351  & 6.82e-01 & 6.82e-01 & 6.82e-01 & 1.77e-01 & 1.77e-01 & 1.77e-01 & 1.42e-01 & 1.42e-01 & 1.42e-01 \\
			17384 & 6.45e-01 & 6.45e-01 & 6.45e-01 & 1.02e-01 & 1.02e-01 & 1.02e-01 & 7.76e-02 & 7.76e-02 & 7.76e-02 \\
			134172 & 7.73e-01 & 7.73e-01 & 7.73e-01 & 1.04e-01 & 1.04e-01 & 1.04e-01 & 7.84e-02 & 7.84e-02 & 7.84e-02 \\
			1055284 & 8.21e-01 & 8.21e-01 & 8.21e-01 & 9.86e-02 & 9.86e-02 & 9.86e-02 & 7.32e-02 & 7.32e-02 & 7.32e-02 \\
			\hline
			$E$& $10$ &$10^2$ & $10^4$  & $10$ &$10^2$ & $10^4$ & $10$ &$10^2$ & $10^4$ \\
			\hline
			\hline
		\end{tabular}
		\caption{Test \ref{subsec:numerical-experiment-robustness}. Efficiency indexes for different values of $\nu$ and $E$ on the unit square and unit cube domain.}
		\label{tabla:robustness-square-2D3D}}
\end{table}

\begin{acknowledgements}

This work was advanced during the ICOSAHOM 2023 conference, Seoul, Korea. AK was  partially
supported by the Sponsored Research \& Industrial Consultancy (SRIC), Indian Institute of Technology Roorkee,
India through the faculty initiation grant MTD/FIG/100878; by SERB MATRICS grant
MTR/2020/000303; by SERB Core research grant CRG/2021/002569;
FL was partially supported by DIUBB through project 2120173 GI/C Universidad del B\'io-B\'io
and ANID-Chile through FONDECYT project 11200529 (Chile).
DM was partially supported by the National Agency for Research and Development, ANID-Chile through project Anillo of
Computational Mathematics for Desalination Processes ACT210087, by FONDECYT project 1220881,
and by project Centro de Modelamiento Matemático (CMM), FB210005, BASAL funds for centers of excellence.
JV was partially supported by the National Agency for Research and Development, ANID-Chile through project Anillo of
Computational Mathematics for Desalination Processes ACT210087, by FONDECYT Postdoctorado project 3230302,
and by project Centro de Modelamiento Matemático (CMM), FB210005, BASAL funds for centers of excellence.
\end{acknowledgements}

\bibliographystyle{siam}
\bibliography{references}

\end{document}